\documentclass[11pt,namelimits,sumlimits]{amsart}
\usepackage{amssymb,amsmath,amsthm,amsrefs}
\usepackage[mathscr]{eucal}
\usepackage{comment}
\usepackage{hyperref}
\usepackage{color}
\usepackage{fullpage}
\usepackage{enumerate}
\usepackage{verbatim}
\usepackage{graphicx,float}

\usepackage{enumitem}

\usepackage{psfrag}

\newtheorem{theorem}[equation]{Theorem}
\newtheorem{lemma}[equation]{Lemma}
\newtheorem{prop}[equation]{Proposition}
\newtheorem{cor}[equation]{Corollary}
\newtheorem{corollary}[equation]{Corollary}

\newtheorem{definition}[equation]{Definition}

\theoremstyle{remark}

\newtheorem{notation}[equation]{Notation}

\newtheorem{assumption}[equation]{Assumption}

\numberwithin{equation}{section}

\newcommand{\R}{\mathbb{R}}
\newcommand{\N}{\mathbb{N}}
\newcommand{\B}{\mathbb{B}}
\newcommand{\K}{\mathbb{K}}
\newcommand{\HH}{\R_+\times\R}
\newcommand{\Ktop}{{\mathbb{K}'_{\top}}}

\newcommand{\KBp}{{\mathbb{K}_{\B}'}}

\newcommand{\betah}{\widehat{\beta}}
\newcommand{\D}{\mathbb{D}}
\newcommand{\Z}{\mathbb{Z}}

\newcommand{\Sph}{\mathbb{S}}

\newcommand{\sgn}{\operatorname{sgn}}

\newcommand{\Grp}{\mathscr{G}}

\newcommand{\Ccal}{\mathcal{C}}
\newcommand{\Acal}{\mathcal{A}}

\newcommand{\sigmaunder}{{\underline{\sigma}}}

\newcommand{\Wcal}{\mathcal{W}}

\newcommand{\rr}{\ensuremath{\mathrm{r}}}

\newcommand{\rot}{\mathsf{R}}

\newcommand{\refl}{{\underline{\mathsf{R}}}}

\newcommand{\abs}[1]{\left\lvert#1\right\rvert}
\newcommand{\norm}[1]{\left\|#1\right\|}

\newcommand{\skernel}{\mathscr{K}}

\newcommand{\sech}{\operatorname{sech}}

\newcommand{\arcosh}{\mathrm{arcosh}}

\newcommand{\supp}{\mathrm{supp}}

\newcommand{\cateConf}{\Wcal[\betah  ; \sigmaunder ]}
\newcommand{\Conf}{\mathcal{W}[\underline{\sigma} {:} k]}
\newcommand{\zConf}{\mathcal{W}[0 {:} k]}
\newcommand{\Confarg}{[\underline{\sigma} {:} k]}
\newcommand{\zConfarg}{[0 {:} k]}

\newcommand{\iniDsingsur}{\Sigma}

\newcommand{\preiniDsingsur}{\Sigma}

\newcommand{\Scherk}{\Sigma(\alpha^+)}
\newcommand{\Scherkin}{\Sigma(\alpha_i^+(\xi))}

\newcommand{\Dsingarg}{[\underline{\alpha},\beta,\underline{\phi},\tau]}
\newcommand{\apDsingarg}{[\underline{\alpha},\underline{\phi}]}

\newcommand{\abDsingarg}{[\underline{\alpha},\beta]}
\newcommand{\aDsingarg}{[\underline{\alpha}]}

\newcommand{\inDsingarg}{[\underline{\alpha}_i(\xi),\beta_i(\xi),\underline{\tilde{\phi}}_i(\xi),\tau]}
\newcommand{\iinDsingarg}{[\underline{\alpha}_{i+1}(\xi),\beta_{i+1}(\xi),\underline{\tilde{\phi}}_{i+1}(\xi),\tau]}
\newcommand{\finDsingarg}{[\underline{\alpha}_1(\xi),\beta_1(\xi),\underline{\tilde{\phi}}_1(\xi),\tau]}
\newcommand{\linDsingarg}{[\underline{\alpha}_k(\xi),\beta_k(\xi),\underline{\tilde{\phi}}_k(\xi),\tau]}

\newcommand{\insurfarg}{[\underline{\sigma},\underline{\varphi}]}

\newcommand{\wrap}{\mathcal{Z}}
\newcommand{\bend}{Z_y[\alpha^+-\alpha^-]}

\newcommand{\barw}{\overline{w}}
\newcommand{\baru}{\overline{u}}

\newcommand{\betain}{{\beta_{in}}}
\newcommand{\alphain}{{\alpha_{in}}}
\newcommand{\pin}{{p_{in}}}
\newcommand{\betaex}{{\beta_{ex}}}
\newcommand{\alphaex}{{\alpha_{ex}}}

\newcommand{\betainK}{{\beta_{in}^\K}}
\newcommand{\alphainK}{{\alpha_{in}^\K}}
\newcommand{\pinK}{{p_{in}^\K}}
\newcommand{\betaexK}{{\beta_{ex}^\K}}
\newcommand{\alphaexK}{{\alpha_{ex}^\K}}
\newcommand{\pexK}{{p_{ex}^\K}}
\newcommand{\pnorth}{{p_{north}}}
\newcommand{\psouth}{{p_{south}}}
\newcommand{\Mmin}{{\mathbf{M}}}

\begin{document}

\title[Free Boundary]{Free Boundary Minimal surfaces in the Euclidean Three-Ball close to the boundary}  

\author[N.~Kapouleas]{Nikolaos~Kapouleas}
\address{Department of Mathematics, Brown University, Providence, RI 02912} 
\email{nicolaos\_kapouleas@brown.edu}

\author[J.~Zou]{Jiahua~Zou} 
\address{Department of Mathematics, Brown University, Providence, RI 02912} 
\email{jiahua\_zou@brown.edu}




\date{\today}



\begin{abstract}
We construct free boundary minimal surfaces (FBMS) embedded in the unit ball in the Euclidean three-space 
which are compact, lie arbitrarily close to the boundary unit sphere, are of genus zero, and their boundary has an arbitrarily large number of connected boundary components. 
The construction is by PDE gluing methods and the surfaces are desingularizations of unions of many catenoidal annuli and two flat discs. 
The union of the boundaries of the catenoidal annuli and discs is the union of a large finite number of parallel circles contained in the unit sphere, 
with each parallel circle contained in the boundary of exactly two of the catenoidal annuli and discs. 
\end{abstract}

\maketitle
\section{Introduction}
\label{Sintro}
\nopagebreak

\subsection*{Brief discussion of the results}
$\phantom{ab}$
\nopagebreak

In recent years there has been much interest on free boundary minimal surfaces (FBMS) in the unit ball $\B^3\subset\R^3$, 
as for example in 
\cite{FS2016,FS2012,FS2011,M2020,GP2010,petrides2014existence,petrides2018,petrides2019maximizing,matthiesen2020free,karpukhin2020minmax,karpukhin,GL2021,GKL2021,zolotareva,kapouleas:li,kapouleas:wiygul}.
In this article we construct by PDE gluing methods compact embedded FBMS in the unit ball $\B^3$ which lie arbitrarily close to the boundary $\Sph^2=\partial\B^3$.   
The construction uses desingularization ideas from \cite{kapouleas:finite},  FBMS construction ideas from \cite{kapouleas:li,kapouleas:wiygul}, 
and some auxiliary ideas from \cite{kapouleas:ii:mcgrath,LDg}. 
Note in particular that although the desingularization approach of \cite{kapouleas:finite} has been used in various constructions before, for example in 
\cite{nguyenIII,kapouleas:kleene:moller,kapouleas:li}, 
this is the first time since \cite{kapouleas:finite} it is used without imposing simplifying extra symmetries. 

The first step of our construction is to determine for each $k\in \N\setminus\{1,2\}$ a 
\emph{family of configurations $\Wcal[\sigmaunder {:} k ]$} parametrized by $\sigmaunder$ (see \ref{lem::}). 
This step resembles the construction of RLD's for $O(2)\times\Z_2$-symmetric backgrounds \cite{kapouleas:ii:mcgrath,LDg}, 
but the minimality ODE is nonlinear resulting in a different proof and construction.  
Each $\Wcal[\sigmaunder {:} k ]$ we construct is $O(2)\times\Z_2$-symmetric and consists of two flat discs (at the top and bottom of the configuration) and $k-1$ catenoidal annuli.  
Here $O(2)$ acts on $\R^3$ by rotations and reflections fixing the $y$-axis pointwise, 
and the generator of $\Z_2$ by reflection with respect to the $xz$-plane (Definition \ref{D:Grp}). 
The boundaries of the discs and the catenoidal annuli are contained in the union of $k$ parallel circles on $\Sph^2$,  
with each of the circles contained in the boundary of exactly two of the discs and annuli. 
Moreover $\Wcal[0{:}k]$ is balanced in the sense that the angles at a common boundary circle the catenoidal annuli (or disc) make with $\Sph^2$ are equal. 
Small \emph{unbalancing} (in the sense of slightly violating the equality) is prescribed continuously by $\sigmaunder$ in order to determine $\Wcal[\sigmaunder {:} k ]$.     
Finally for $k$ large enough the configurations $\Wcal[\sigmaunder {:} k ]$    
are arbitrarily close to $\Sph^2$, the catenoidal annuli arbitrarily narrow, and the discs arbitrarily small  
(Proposition \ref{propconvconf}).  

Moving forward (leading to the construction of our FBMS),  
we impose the symmetries of a finite subgroup of $O(2)\times\Z_2$, namely $D_{2m}\times\Z_2$, where $D_{2m}$ is a dihedral group of order $2m$ (Definition \ref{D:Grp}). 
Imposing such a dihedral subgroup of $O(2)$ symmetries is common for constructions where the background is rotationally symmetric but the surfaces constructed are not, 
for example \cite{kapouleas:finite}, \cite{kapouleas:ii:mcgrath}, or \cite[Parts II and III]{LDg}. 
The $\Z_2$ symmetry is imposed to simplify the presentation and can be removed as for example in \cite{kapouleas:zou}. 
Another simplifying assumption we adopt in this article is using exactly $m$ ``necks'' in the desingularizations of each circle 
(unlike in \cite{kapouleas:finite,kapouleas:david1} for example where arbitrary multiples of $m$ depending on the circle are used). 
As a result we produce FBMS's of genus zero with exactly $km$ connected boundary components. 

In the second step we construct the initial surfaces. 
Their construction is based on appropriately modifying the singly periodic Scherk surfaces so that they can ``desingularize'' the boundary circles of the configurations.  
As in \cite{kapouleas:finite}, the initial surfaces consist of three regions smoothly joined together. 
The first region consists of the modified \emph{cores} of \emph{half} singly periodic Scherk surfaces.  
Here the Scherk surfaces are subdivided in two halves by a symmetry plane through their axis (the $yz$-plane in \ref{D:scherk}). 
The ``waists'' of the Scherk surfaces on this plane are modified in the construction to form the boundary of the initial surfaces. 
Particular care is taken to ensure as in \cite{kapouleas:wiygul} that the initial surfaces are orthogonal to $\Sph^2$ at the boundary.  

The second region consists of graphical regions over catenoidal annuli imitating the graphical property of the wings of the Scherk surfaces over their asymptotic half-planes. 
``Bending'' the asymptotic half-planes to catenoids instead of cones is an important feature of the construction and as in \cite{kapouleas:finite} ensures good estimates for the mean curvature 
of the initial surfaces. 
The second region consists of \emph{two} (unlike in \cite{kapouleas:finite} with \emph{four})  
annular regions (or \emph{wings}) for each desingularizing modified half Scherk surface. 
The third region consists of catenoidal annuli close to the ones of the configuration being desingularized. 

In the third step we estimate the mean curvature $H$ of the initial surfaces. 
In Proposition \ref{propmeancurvdesing} $H$ is effectively decomposed into two parts: 
a part in the \emph{extended substitute kernel} $\skernel$, 
which is desirable and amounts more or less to the linearized change of the mean curvature under the action of the unbalancing parameters, 
and an error part which has to be corrected by solving the PDE.  
The error part is the sum of the gluing error, which is of order $\tau=1/m$, 
and the unbalancing higher order terms, which are  controlled by the squares of the unbalancing parameters 
(which in Proposition \ref{propmeancurvdesing} are $\alpha^+-\alpha^-$, $\phi^+$, and $\phi^-$). 

Note that the action of the unbalancing parameters is inspired by the \emph{geometric principle} (see \cite{kapouleas:clay,kapouleas:rs} for a general discussion)  
and relocates the various regions of the initial surfaces relative to each other as in \cite{kapouleas:finite,kapouleas:general}.   
There are \emph{three} unbalancing parameters for each desingularizing surface (equivalently per neck modulo symmetries), unlike \cite{kapouleas:finite} where there are \emph{six}, 
or \cite{kapouleas:general} where there are \emph{seven} (see \cite[page 519]{kapouleas:clay} or \cite[section 5.3]{kapouleas:rs}). 
One of them corresponds to the translation consistent with the construction and unbalances the wings creating substitute kernel (see \ref{defw}), 
and the other two reposition the two wings relative to the core creating extended substitute kernel (see \ref{defubar}) used to ensure decay along the wings. 

In the fourth step we solve the linearized PDE globally on the initial surfaces (Proposition \ref{propLEM}).
By solving the linearized PDE modulo $\skernel$ we can make the inhomogeneous term orthogonal to the corresponding eigenfunctions (Lemma \ref{lemsubkernel}\ref{itemsubkernelproj}) 
so we can find a solution and moreover (using the ``extended'' part of $\skernel$) we can ensure exponential decay of the solution along the wings (Lemma \ref{lemextker}\ref{itemextkerucomp}).  
The proof of the main Proposition \ref{propLEM}) is in the style of (for example) \cite{kapouleas:equator,kapouleas:li} rather than \cite{kapouleas:finite}: 
We first solve the linearized equation on the ``standard models'', that is on the half Scherk surfaces (Proposition \ref{propLES} and catenoidal annuli or discs (Lemma \ref{lemLEN}). 
By comparing the operators on the standard models and the initial surfaces (Lemma \ref{lemopcompare}), 
we can transfer the results on the standard models to the initial surfaces, and combine them to prove the main Proposition \ref{propLEM} using an iteration to correct small errors. 

In the fifth and final step we first estimate the nonlinear terms of the mean curvature (Proposition \ref{propnonlinear}) of graphs over the initial surfaces.
Note that following \cite{kapouleas:wiygul}, we use a simple auxiliary metric, which is a product metric,  to define these graphs; 
this ensures a simple description of the free boundary condition. 
We then apply Schauder's fixed point theorem to prove 
the main theorem of this article which less formally is as follows.

\begin{theorem}[Theorem \ref{thm}]
\label{thmA}
Given $k\in \N$ large enough in absolute terms and $m\in\N$ large enough in terms of $k$, 
there is a compact embedded two-sided free boundary minimal smooth surface (FBMS) $\Mmin_{k,m}$ of genus zero and $km$ connected boundary components 
(with $m$ of them in the vicinity of each of the $k$ boundary circles of $\Wcal\zConfarg$), 
which is symmetric under the action of $\Grp=D_{2m}\times\Z_2$ (Definition \ref{D:Grp}) and is the graph of a small function $\tilde{v}$ 
in an auxiliary ambient metric $g_A$ (Definition \ref{defga}) over a smooth initial surface $M[\xi_0]$ (see \ref{defini} and \ref{propini}).  
Moreover $\Mmin_{k,m}$ converges on compact subsets of the interior of $\B^3$ in all norms to the configuration $\Wcal\zConfarg$ as $m\to\infty$; 
the Hausdorff distance of $\Mmin_{k,m} $ from $\Wcal\zConfarg$ tends to $0$ as $m\to\infty$; 
and in turn the Hausdorff distance of $\Wcal\zConfarg$ from $\Sph^2=\partial\B^3$ tends to $0$ as $k\to\infty$.  
Finally $\lim_{k\to\infty}\lim_{m\to\infty} \Mmin_{k,m} = \Sph^2$ in the varifold sense. 
\end{theorem}

The construction here can be generalized to produce FBMS close to the boundary in rotationally symmetric convex domains in {E}uclidean three-space \cite{kapouleas:zou}.   
We remark also that 
very recently Karpukhin-Stern proved that the FBMS in the unit ball $\B^3$, 
maximizing the normalized first Steklov eigenvalue among oriented surfaces of genus zero and $k_\partial$ boundary components,   
converge as $k_\partial\to\infty$ to $\Sph^2=\partial\B^3$ in the sense of varifolds 
\cite[Corollary 1.4]{karpukhin}.    
(Fraser-Schoen had already proved that such maximizers exist and can be realized as compact embedded smooth FBMS in the unit ball $\B^3$ \cite{FS2016}.)  
At the moment however little is known about the geometry of these eigenvalue maximizing surfaces, 
unlike the ones constructed here which are small perturbations of explicit \emph{initial surfaces}.  
It seems unlikely that the surfaces constructed here are the maximizers of the normalized first Steklov eigenvalue, 
but it may be possible to construct some of the maximizers by enhancing our approach.

\subsection*{General notation and conventions}
$\phantom{ab}$
\nopagebreak

\begin{definition}
\label{D:EBS}
We identify the Euclidean space $E^3$ with $\R^3$ by using 
the standard Cartesian coordinate system $O.xyz$.  
We define the unit ball 
$\B^3:=\{(x,y,z)\in \R^3: x^2+y^2+z^2\le1\}$ and its boundary $\Sph^2:=\partial\B^3$. 
For convenience we will identify $\R^2$ with the $xy$-plane in $\R^3$ by identifying $(x,y)$ with $(x,y,0)$. 
Moreover denoting this identification by ``$\simeq$'', we introduce the notation  
\begin{equation} 
\label{R2+} 
\R^2_+:= \{(x,y,0)\in \R^3 \,:\, x\ge0\} \simeq  \{(x,y)\in \R^2 \,:\, x\ge0\}. 
\end{equation} 
Finally we take the north pole of the sphere to be the point $\pnorth:=(0,1,0)\in\R^2\subset\R^3$ and the south pole to be the point $\psouth:=(0,-1,0)\in\R^2\subset\R^3$. 
\end{definition}

\begin{definition}[Action of symmetry groups]   
\label{D:Grp} 
We let $O(2)\times\Z_2$ act on $\R^3$ (with coordinates as in \ref{D:EBS}) by assuming that $O(2)$ acts the usual way on the $xz$-plane while fixing the $y$-axis pointwise, 
and the generator of $\Z_2$ by reflection with respect to the $xz$-plane. 
Given 
$m\in \mathbb{N}\setminus\{1,2\}$, 
we let $D_{2m}$ denote the dihedral subgroup of $O(2)$ of order $2m$ 
whose action on the $xz$-plane includes reflections with respect to the lines $\{\cos{\frac{l\pi}{m}z}-\sin{\frac{l\pi}{m}x}=0\}$ for $l\in \mathbb{Z}$. 
We finally define $\Grp := D_{2m} \times \Z_2$ with action on $\R^3$ as a subgroup of $O(2)\times\Z_2$.  
\end{definition}

\begin{definition}
\label{D:newweightedHolder}
Assuming that $\Omega$ is a domain inside a manifold,
$g$ is a Riemannian metric on the manifold, 
$f:\Omega\to(0,\infty)$ are given functions, 
$k\in \N$, 
$\beta\in[0,1)$, 
$u\in C^{k,\beta}_{loc}(\Omega)$ 
or more generally $u$ is a $C^{k,\beta}_{loc}$ tensor field 
(section of a vector bundle) on $\Omega$, 
and that the injectivity radius in the manifold around each point $x$ in the metric $g$
is at least $1/10$,
$\|u: C^{k,\beta} ( \Omega,g,f)\|$ is defined by
$$
\|u: C^{k,\beta} ( \Omega,g,f)\|:=
\sup_{x\in\Omega}\frac{\,\|u:C^{k,\beta}(\Omega\cap B_x, g)\|\,}{f(x) },
$$
where $B_x$ is a geodesic ball centered at $x$ and of radius $1/100$ in the metric $g$.
For simplicity any of $\beta$ or $f$ may be omitted, 
when $\beta=0$ or $f\equiv1$, respectively.
\end{definition}

$f$ can be thought of as a ``weight'' function because $f(x)$ controls the size of $u$ in the vicinity of
the point $x$.
at the vicinity of each point $x$.
Note that from the definition it follows that 
\begin{equation}
\label{E:norm:derivative}
\| \, \nabla u: C^{k-1,\beta}(\Omega,g,f)\|
\le
\|u: C^{k,\beta}(\Omega,g,f)\|, 
\end{equation}
and the multiplicative property 
\begin{equation}
\label{E:norm:mult}
\| \, u_1 u_2 \, : C^{k,\beta}(\Omega,g,\, f_1 f_2 \, )\|
\le
C(k)\, 
\| \, u_1 \, : C^{k,\beta}(\Omega,g,\, f_1 \, )\|
\,\,
\| \, u_2 \, : C^{k,\beta}(\Omega,g,\, f_2 \, )\|.
\end{equation}

Cut-off functions will be used extensively,
and for this reason the following is adopted.

\begin{definition}
\label{DPsi} 
A smooth function $\Psi:\R\to[0,1]$ is fixed with the following properties:
\newline
(i).
$\Psi$ is nondecreasing.
\newline
(ii).
$\Psi\equiv1$ on $[1,\infty]$ and $\Psi\equiv0$ on $(-\infty,-1]$.
\newline
(iii).
$\Psi-\frac12$ is an odd function.
\end{definition}

Given now $a,b\in \R$ with $a\ne b$, the smooth function
$\psi[a,b]:\R\to[0,1]$ is defined
by
\begin{equation}
\label{Epsiab}
\psi[a,b]:=\Psi\circ L_{a,b},
\end{equation}
where $L_{a,b}:\R\to\R$ is the linear function defined by the requirements $L(a)=-3$ and $L(b)=3$.

Clearly then $\psi[a,b]$ has the following properties:
\newline
(i).
$\psi[a,b]$ is weakly monotone.
\newline
(ii).
$\psi[a,b]=1$ on a neighborhood of $b$ and 
$\psi[a,b]=0$ on a neighborhood of $a$.
\newline
(iii).
$\psi[a,b]+\psi[b,a]=1$ on $\R$.

\begin{definition}
$g$ is define to be the Euclidean metric in $E^3$ and $g_M$ is define to be the restriction of $g$ on the embedding surface $M$ in $E^3$. 
\end{definition}

\begin{definition}
For an oriented embedding surface $M$ in $E^3$, $A_{M}$ is defined to be the second fundamental form of $M$ and $\abs{A}^2_{M}$ is defined to be the square of the second fundamental form. Also the operator $\mathcal{L}_M$ is defined by
\begin{equation*}
    \mathcal{L}_M=\Delta_M+\abs{A}^2_{M},
\end{equation*}
where $\Delta_M$ is the Laplacian on $M$ defined by $g_M$.
\end{definition}

\begin{notation}[Symmetric functions] 
\label{N:G} 
Given a manifold $M$ invariant under the action of some group $G$ and a space of functions $\mathcal{X}\subset C^0(\Omega)$,  
where $\Omega$ is a subset of $M$ invariant under the action of $G$, 
we use a subscript ``$G$'' to denote the subspace $\mathcal{X}_{G}\subset\mathcal{X}$ consisting of the functions in $\mathcal{X}$ which are invariant under the action of $G$.
\end{notation}

\subsection*{Acknowledgments}
The authors would like to thank Fernando Marques and Richard Schoen for suggesting this problem to one of them during his visit to IAS in Fall 2018.

\section{The Initial Configurations}

Recall that we use the standard Cartesian coordinate system $O.xyz$ in $\R^3$ as in \ref{D:EBS}. 
In this section we concentrate on objects in $\R^3$ which are rotationally invariant around the $y$-axis (recall \ref{D:EBS}) and so are determined uniquely by 
their intersection with $\R^2_+$. 

\begin{definition}[Polar coordinates] 
	\label{D:polar} 
	We define polar coordinates $(r,\beta)$ on the $xy$-plane by 
	$(x,y)\, =\,r\, (\sin\beta,\cos\beta)$ and so $\beta=0$ at the north pole and $\beta=\pi$ at the south pole. 
Moreover given a function $\rr:B\to\R_+$ where $B\subset[0,\pi]$, we define the \emph{radial graph of $\rr$} to be the set 
$$ 
\{\, \rr(\beta)\,(\sin\beta,\cos\beta)\, : \, \beta\in B \} \subset \R^2_+. 
$$ 
\end{definition}

\subsection*{The Catenoidal Annuli}

\begin{lemma}[Catenoids around the $y$-axis] 
	\label{L:BK} 
	For a catenoid $\K\subset \R^3$ which is rotationally invariant around the $y$-axis the following hold.  
	\begin{enumerate}[label={(\roman*)},ref={(\roman*)}]
		\item 
		There are unique $a = a^\K \in\R_+$ and $b = b^\K \in\R$ such that 
		$\K':= \K\cap \R^2_+ =$ 
		\\ 
		$\phantom{kk}$ \hfill  
		$=\left\{(x,y)\in\R^2_+ : x=a\cosh(\frac{y-b}{a})\right\}$.  
		\item 
		There are exactly two straight lines through the origin which are tangent to $\K'$. 
		We call the points of tangency $p_{\top^\pm}^\K \in \K'$, chosen so that their $y$-coordinates $y_{\top^\pm}^\K\in\R$ 
		satisfy $y_{\top^-}^\K < y_{\top^+}^\K$.  
		\item 
		The arc $\Ktop$ of $\K'$ with endpoints $p_{\top^\pm}^\K$ is graphical in polar coordinates (that is $r$ is a function of $\beta$ on $\Ktop$). 
		The inward normal $\nu=\nu^\K$ and the position vector $X$ of $\K$ satisfy $\nu\cdot X>0$ on the interior of $\Ktop$ and 
		$\nu\cdot X<0$ on $\K'\setminus\Ktop$. 
		\item 
		\label{itemBKv}
		$v:=\frac{d}{d\beta}\log\frac1{r}$ is a strictly decreasing function of $\beta$ on the interior of $\Ktop$---equivalently $\log r$ is strictly convex on $\Ktop$---and 
		satisfies the first order equation 
		\begin{equation}\label{eqcatev}
			\frac{\mathrm{d}v}{\mathrm{d}\beta}=-\cot\beta v(v^2+1)-2(v^2+1).  
		\end{equation}
		\item 
		\label{iBKc}
		$\K\cap\B^3$ is connected and when nonempty   
		it is an annulus or a circle, and there exist (uniquely determined) $y_-^\K , y_+^\K \in (-1,1)$ with $y_-^\K \le y_+^\K$, 
		such that 
		\begin{enumerate}[label=\emph{(\alph*)}]
			\item 
			$\K\cap\B^3 = \{(x,y,z)\in\K:  y_-^\K  \le y \le y_+^\K \}$, 
			\item 
			$\K \cap \Sph^2 = \{ (x,y,z)\in \Sph^2 : y = y_-^\K \text{ or } y = y_+^\K \}$,    
			\item 
			\label{isign} 
			$\sgn b^\K = \sgn \, ( y_-^\K + y_+^\K ) $.    
		\end{enumerate}
	\end{enumerate}
\end{lemma} 

\begin{proof} 
	(i) follows by the equation of catenary $\K'$.
	
	For (ii), consider the function $\frac{y}{x}$ on $\K'$. By definition, the tangency happens exactly at the critical point of the function $\frac{y}{x}$. By (i),
	\begin{align*}
		\frac{\mathrm{d}}{\mathrm{d}y}\frac{y}{x}=\frac{\mathrm{d}}{\mathrm{d}y}\frac{y}{a\cosh(\frac{y-b}{a})}=\frac{1-\frac{y}{a}\tanh(\frac{y-b}{a})}{a\cosh(\frac{y-b}{a})}.
	\end{align*}
	However, 
	\begin{align*}
		\frac{\mathrm{d}}{\mathrm{d}y}y\tanh\left(\frac{y-b}{a}\right)=\sech^2\left(\frac{y-b}{a}\right)\left(\sinh\left(\frac{y-b}{a}\right)\cosh\left(\frac{y-b}{a}\right)+\frac{y}{a}\right).
	\end{align*}
	The function $\sinh(\frac{y-b}{a})\cosh(\frac{y-b}{a})+\frac{y}{a}$ is increasing and only has one zero, say $y_0$. Therefore, the function $1-\frac{y}{a}\tanh(\frac{y-b}{a})$ has only one critical point $y_0$, which is a maximum, and it is increasing in $(-\infty,y_0)$ and decreasing in $(y_0,\infty)$. And at $y=0$, it equals $1$; moreover, $\lim_{y\to\pm\infty}1-\frac{y}{a}\tanh(\frac{y-b}{a})=\mp\infty$. Thus it has only two zeroes. Finally, these two points are the only two critical points of $\frac{y}{x}$, which correspond to $p_{\top^\pm}^\K$.
	
	The fact that the arc $\Ktop$ of $\K'$ is graphical in polar coordinates follows by (ii) because the graphical property fails only at the tangency points. A straightforward calculation shows that $X=(a\cosh(\frac{y-b}{a}),y)$ and $\nu=(\sech(\frac{y-b}{a}),-\tanh(\frac{y-b}{a}))$, and thus $X\cdot\nu=a-y\tanh(\frac{y-b}{a})$, which has the same sign as the function $\frac{\mathrm{d}}{\mathrm{d}y}\frac{y}{x}$. The rest part of (iii) then follows by the results in the proof of (ii).
	
	The catenary $\K'$ satisfies the differential equation 
	\begin{equation}\label{eqxycate}
		\frac{\mathrm{d}^2x}{\mathrm{d}y^2}=\frac{1+(\frac{\mathrm{d}x}{\mathrm{d}y})^2}{x}.
	\end{equation}
	
	From the definitions of $v,\beta$ and their relationships with $x,y$, the equation could be rewritten with the new variables. The equation \eqref{eqcatev} then follows. 
	
	Assume that $\beta\leq\pi/2$. From the equation, When $v\geq 0$, by \eqref{eqcatev}, $\frac{\mathrm{d}v}{\mathrm{d}\beta}<0$, and thus $v$ is decreasing. Now suppose that at some $\beta^1\leq\pi/2$, $\frac{\mathrm{d}v}{\mathrm{d}\beta}(\beta^1)>0$ and $v(\beta^1)<0$. By the definitions, at $p_{\top^+}^\K$, $v=+\infty$. Thus there is $\beta^0<\beta^1$, such that $\frac{\mathrm{d}v}{\mathrm{d}\beta}(\beta^0)=0$, $v(\beta^0)<0$ and $\frac{\mathrm{d}v}{\mathrm{d}\beta}(\beta)>0$ for all $\beta\in(\beta^0,\beta^1)$. Therefore, $v(\beta^0)<v(\beta^1)$. Moreover, $0=\frac{\mathrm{d}v}{\mathrm{d}\beta}(\beta^0)=-\cot\beta^0 v(\beta^0)((v(\beta^0))^2+1)-2((v(\beta^0))^2+1)$, or $\cot\beta^0 v(\beta^0)=-2$. However, as $v(\beta^0)<v(\beta^1)<0$, $\cot\beta^0>\cot\beta^1>0$
	\begin{align*}
		\frac{\mathrm{d}v}{\mathrm{d}\beta}(\beta^1)=-(\cot\beta^1v(\beta^1)+2) ((v(\beta^1))^2+1)<0,
	\end{align*}
	which contradicts the assumption. This shows that $v$ is a nonincreasing functions when $\beta\leq\pi/2$. As the equation does not have constant solution, $v$ must be strictly decreasing. The case when $\beta>\pi/2$ then simply follows by the symmetry.
	
	Finally, by the fact that the function $r^2(y):=x^2(y)+y^2$ on $\K'$ is convex, where $x(y):=a\cosh{(\frac{y-b}{a})}$ as in (i), $\K'\cap \D^2$ is connected, where $\D^2$ is the unit disk. Moreover, as $\K'$ and $\Sph^1\cap \R^2_+$ are both graphical over the $y$-axis, there exist uniquely determined $y_-^\K , y_+^\K \in (-1,1)$ with $y_-^\K \le y_+^\K$ such that 
	\begin{align*} 
		\K'\cap\D^2 =& \{(x,y)\in\K':  y_-^\K  \le y \le y_+^\K \}, 
		\\ 
		\K' \cap \Sph^1 =& \{ (x,y)\in \Sph^1 : y = y_-^\K \text{ or } y = y_+^\K \}.   
	\end{align*} 
	The first two items in (v) then follows from the relationship between $\K'$ and $\K$. Finally, suppose that $\K' \cap \Sph^1 = \{(x_-^\K, y_-^\K), (x_+^\K, y_+^\K) \}$, if $y_-^\K + y_+^\K \geq 0$, then $x_-^\K\leq x_+^\K$. By looking $x(y):=a\cosh{(\frac{y-b^\K}{a^\K})}$ as a graph over $y$-axis, the symmetry axis $y=b^\K$ is above the line $y=(y_-^\K + y_+^\K)/2$. Thus $b^\K>0$ and the last item of (v) follows by symmetry.
\end{proof}

\begin{definition} 
	\label{D:alphabeta}
	Given a catenoid $\K$ as in \ref{L:BK} with $\K\cap\B^3$ an annulus, we define the arc $\KBp:=\K\cap\R^2_+\cap\B^3$.  
	We define $\betainK,\betaexK \in(0,\pi)$ with $\betainK<\betaexK$ to be the latitudes of the endpoints of $\KBp$, 
	and hence by \ref{L:BK}\ref{iBKc} we have $\cos\betainK=y_+^\K$ and $\cos\betaexK=y_-^\K$. 
	The endpoints of $\KBp$ will be denoted by $\pinK:=( \sin\betainK , \cos\betainK \,) \in \Sph^1$ and $\pexK:=( \sin\betaexK , \cos\betaexK \,) \in \Sph^1$. 
	Finally we define $\alphainK$ ($\alphaexK$) to be the angle at $\pinK$ ($\pexK$) 
	between $\KBp$ and the arc of $\Sph^1$ with endpoints $\pinK,\,\psouth$ ($\pexK,\,\pnorth$).  
\end{definition}

\begin{lemma}     
	\label{L:Kin} 
	Given a catenoid $\K$ as in \ref{D:alphabeta} we have $\alphainK \in (0, \pi- \betainK)$. 
	Conversely given $\betain\in(0,\pi)$ and $\alphain \in (0, \pi- \betain)$ 
	there is a unique catenoid which will be denoted by $\K[\betain,\alphain]$ and is as in \ref{D:alphabeta} and satisfies 
	$\beta_{in}^{ \K[\betain,\alphain] } =\betain$ and $\alpha_{in}^{ \K[\betain,\alphain] }    =\alphain$.  
\end{lemma} 

\begin{proof} 
	Note that the condition $\alphainK <\pi-\betainK $ amounts by Lemma \ref{L:BK} to the fact that $\KBp$ lies below $\pin$. $\pi-\betainK$ in the angle between the line segment $\{(x,y):y=\sin{\betainK},0<x\leq\cos{\betainK}\}$ and the arc of $\Sph^1$ with endpoints $\psouth$. Conversely, given $\alphain$ and $\betain$ with $0<\alphain<\pi-\betain$,
	consider the ordinary differential equation \eqref{eqxycate} with the initial conditions: $x=\sin{\betain}$, $y=\cos{\betain}$ and $\frac{\mathrm{d}x}{\mathrm{d}y}=\tan(\alphain+\betain-\frac{\pi}{2})$. The graph $(x(y),y)$ of the solution is a catenary $\K'$ going through the point $\pin:=(\sin{\betain}, \cos{\betain})$. Moreover, at $\pin$, the inner product $(x,y)\cdot(\frac{\mathrm{d}x}{\mathrm{d}y},1)=\sin\betain(\cot{\betain}-\cot(\alphain+\betain))>0$ because of the condition $\alphain+\betain<\pi$. This means that the graph is going in the circle $\Sph^1$ at $\pin$ if $y$ is decreasing. A calculation along with the definition of $\betainK$ and $\alphainK$ in \ref{D:alphabeta} shows that $\pinK=\pin$, $\betainK=\betain$ and $\alphainK=\alphain$. 
	
	On the other hand, each catenary $\K'$ corresponding to the catenoid $\K$ in the lemma must satisfy the equation and initial conditions as above, the uniqueness then follows.
\end{proof}

\begin{notation} 
	\label{N:sqbr} 
	We will usually simplify the notation related to $\K=\K[\betain,\alphain]$ (defined as in \ref{L:Kin}) by using ``$[\betain,\alphain]$'' 
	to emphasize the dependence on the parameters, 
	for example we may write $\betaex[\betain,\alphain]$ instead of $\beta_{ex}^{\K[\betain,\alphain] }$. 
\end{notation} 

\begin{lemma}     
	\label{L:Kex} 
	For $\alphain,\betain$ as in \ref{L:Kin} we have 
	$\KBp[\betain,\alphain] \subset \Ktop[\betain,\alphain]$ if and only if $\alphain <\pi/2$ and $\alphaex[\betain,\alphain]<\pi/2$.  
\end{lemma} 
\begin{proof} 
	By the definition \ref{D:alphabeta}, the conditions $\alphain <\pi/2$ and $\alphaex[\betain,\alphain]<\pi/2$ are equivalent with $X\cdot \nu>0$ at $\pinK$ and $\pexK$. The result then follows by \ref{L:BK}(iii).
\end{proof}

\begin{lemma} 
	\label{lemcate}
	The following hold for $\K= \K[\betain,\alphain]$ as in \ref{L:Kin}, $a=a[\betain,\alphain]$, $b=b[\betain,\alphain]$ as in \ref{L:BK}, $\betaex=\betaex[\betain,\alphain]$, 
	and $\alphaex=\alphaex[\betain,\alphain]$.   
	\begin{enumerate}[label={(\roman*)},ref={(\roman*)}]
		\item 
		\label{itemcatepara}
		$\phantom1\quad a=\sin(\beta_{in}) \, \sin(\alpha_{in}+\beta_{in}) = \sin(\beta_{ex}) \, \sin(-\alpha_{ex}+\beta_{ex}) $,\\
		$\phantom1\quad b=
		\cos(\beta_{in}) - \sgn(  \alpha_{in}+\beta_{in} -  \pi/2  )  \, \sin(\beta_{in})\,\sin(\alpha_{in}+\beta_{in})\,\arcosh( {1}/ {\sin(\alpha_{in}+\beta_{in})} ) $. \\ 
		Moreover when $\betaex\leq \pi/2$, \\ 
		$\phantom1 \quad b=\cos(\beta_{ex})+\sin(\beta_{ex}) \, \sin(-\alpha_{ex}+\beta_{ex}) \, \arcosh  (  {1} / {\sin(-\alpha_{ex}+\beta_{ex})} ) $.
		\item 
		\label{itemdecreasing} 
		$\quad \sgn (\alphain-\alphaex) = \sgn b = \sgn (\pi-\betain-\betaex ) $. 
		\item If $\alphain <\pi/2$ and $\alphaex[\betain,\alphain]<\pi/2$,   
		$v:=\frac{d}{d\beta}\log\frac1{r}$ is strictly decreasing on $\KBp[\betain,\alphain]$ 
		and satisfies \eqref{eqcatev} with boundary conditions  $v(\beta_{in})=\tan(\alpha_{in})$, $v(\beta_{ex})=-\tan(\alpha_{ex})$.\label{itemconcavity}
	\end{enumerate}
\end{lemma}

\begin{proof}
	At $\pinK=( \sin\betainK , \cos\betainK \,)$, the equation of $\K'$ satisfies $\frac{\mathrm{d}x}{\mathrm{d}y}=\tan(\alphain+\betain-\frac{\pi}{2})$; while At $\pexK=( \sin\betaexK , \cos\betaexK \,)$, the equation of $\K'$ satisfies $\frac{\mathrm{d}x}{\mathrm{d}y}=\tan(-\alphaex+\betaex-\frac{\pi}{2})$. (i) then follows by the equation and definitions directly. 
	
	By the definition of $v$ in \ref{L:BK}(iv), $\frac{\mathrm{d}x}{\mathrm{d}y}=\frac{v\sin\beta-\cos\beta}{v\cos\beta+\sin\beta}$. The formulae for $v(\betain)$ and $v(\betaex)$ in (iii) then follows. The rest part of (iii) follows by \ref{L:BK}\ref{itemBKv} and \ref{L:Kex}.
	
	If $\betain+\betaex\leq \pi$, then $\pi/2>\pi/2-\betain\geq\betaex-\pi/2>-\pi/2$, thus $y_-^\K=\cos\betain=\sin(\pi/2-\betain)\geq\sin(\betaex-\pi/2)=-\cos\betaex=-y_+^\K$, or $y_-^\K+y_+^\K\geq0$. The second inequality of \ref{itemdecreasing} then follows by \ref{L:BK}\ref{iBKc}\ref{isign}.
	
	From a symmetry argument, the first equality of \ref{itemdecreasing} holds as long as $\alphain>\alphaex$ if $b>0$. After a scaling and translation, this could be reduced to the problem of the comparison of the two angles made by a line $x=py+q$ 
	and the standard catenary $x=\cosh{y}$ with the assumptions that $p>0,q>0$.
	
	Suppose the line intersects with the catenary at points $P_1(\cosh(y_1),y_1)$ and $P_2(\cosh(y_2),y_2)$, with $y_2>y_1$ and $\cosh(y_2)>\cosh(y_1)$. Then the unit tangent vectors at $P_1$, $P_2$ that direct to the positive $x$-axis could be written as $\mathbf{v}_1=\left(\frac{\sinh(y_1)}{\cosh(y_1)},\frac{1}{\cosh(y_1)}\right)$ and $\mathbf{v}_2=\left(\frac{\sinh(y_2)}{\cosh(y_2)},\frac{1}{\cosh(y_2)}\right)$. Therefore, for $i=1,2$,
	\begin{equation*}
		\abs{\overrightarrow{P_1P_2}}\cos\left(\langle\widehat{\overrightarrow{P_1P_2},\mathbf{v}_i}\rangle\right)=\overrightarrow{P_1P_2}\cdot\mathbf{v}_i=\frac{y_2-y_1}{\cosh(y_i)}+\frac{(\cosh(y_2)-\cosh(y_1))\sinh(y_i)}{\cosh(y_i)},
	\end{equation*}
	where $\langle\widehat{\overrightarrow{P_1P_2},\mathbf{v}_i}\rangle$ is the angle made by the vectors $\overrightarrow{P_1P_2}$ and $\mathbf{v}_i$; and thus
	\begin{align*}
		&\cosh(y_2)\cosh(y_1)\abs{\overrightarrow{P_1P_2}}\left(\cos\left(\langle\widehat{\overrightarrow{P_1P_2},\mathbf{v}_2}\rangle\right)-\cos\left(\widehat{\langle\overrightarrow{P_1P_2},\mathbf{v}_1}\rangle\right)\right)\\
		=&-(y_2-y_1)(\cosh(y_2)-\cosh(y_1))+(\sinh(y_2)\cosh(y_1)-\sinh(y_1)\cosh(y_2))(\cosh(y_2)-\cosh(y_1))\\
		=&(\cosh(y_2)-\cosh(y_1))(\sinh(y_2-y_1)-(y_2-y_1))>0,
	\end{align*}
	or
	\begin{equation*}
		\langle\widehat{\overrightarrow{P_1P_2},\mathbf{v}_1}\rangle>\langle\widehat{\overrightarrow{P_1P_2},\mathbf{v}_2}\rangle.
	\end{equation*}
This shows the first inequality in \ref{itemdecreasing}.
\end{proof}

\begin{prop}
	\label{propmonotone}
	If $\alphain <\pi/2$ and $\alphaex[\betain,\alphain]<\pi/2$,
	then 
	$\betaex=\betaex[\betain,\alphain]$ and $\alphaex=\alphaex[\betain,\alphain]$ are strictly increasing functions of $\beta_{in}$ and $\alpha_{in}$. 
\end{prop}
\begin{proof}
	By the definition of $v$ in \ref{L:BK}(iv), $\int_{\beta_{in}}^{\beta_{ex}}v(\beta)\mathrm{d}\beta=0$. By the monotonicity in \ref{L:BK}(iv), $v(\beta)$ has a unique inverse function $\beta(v)$ and thus
	\begin{equation}\label{eq1}
		\int_0^{v(\beta_{in})}\beta(v)-\beta_{in}\mathrm{d}v=\int_{v(\beta_{ex})}^0\beta_{ex}-\beta(v)\mathrm{d}v.
	\end{equation} 
	
	From the theory of first order ordinary differential equations, 
	if $v^1$ and $v^2$ are two solutions of the equation \eqref{eqcatev}, 
	the curves of the solutions $(\beta,v^1(\beta))$ and $(\beta,v^2(\beta))$ do not intersect with each other. 
	Therefore, if $\beta^1_{in}<\beta^2_{in}$, $\alpha^{1}_{in}=\alpha^{2}_{in}$, $v^1,v^2$ are corresponding solutions, then
	\begin{equation*}
		v^1(\beta)<v^2(\beta),
	\end{equation*}
	when $\beta\in (\beta^1_{in},\beta_{ex}^1)\cap(\beta^2_{in},\beta_{ex}^2)$. Moreover, when $v^1(\beta)=v^2(\beta')$, $\beta'>\beta$. Therefore, from \eqref{eqcatev},
	\begin{equation*}
		\frac{\mathrm{d}v^1}{\mathrm{d}\beta}(\beta)<\frac{\mathrm{d}v^2}{\mathrm{d}\beta}(\beta')<0,
	\end{equation*}
	when $v^1(\beta)=v^2(\beta')>0$; and
	\begin{equation*}
		0>\frac{\mathrm{d}v^1}{\mathrm{d}\beta}(\beta)>\frac{\mathrm{d}v^2}{\mathrm{d}\beta}(\beta'),
	\end{equation*}
	when $v^1(\beta)=v^2(\beta')<0$. 
	Let $\beta^1$, $\beta^2$ denote the two inverse functions for $v^1$, $v^2$, 
this means $0>\frac{\mathrm{d}\beta^1}{\mathrm{d}v}>\frac{\mathrm{d}\beta^2}{\mathrm{d}v}$ when $v>0$ 
and $0>\frac{\mathrm{d}\beta^2}{\mathrm{d}v}>\frac{\mathrm{d}\beta^1}{\mathrm{d}v}$ when $v<0$. 
As $v^1(\beta_{in}^1)=\tan\alpha_{in}^1=\tan\alpha_{in}^2=v^2(\beta_{in}^2)$ from \ref{lemcate}\ref{itemconcavity}, 
for any $v>0$, $\beta^1(v)-\beta^1_{in}<\beta^2(v)-\beta^2_{in}$. Moreover, when $v'<v<0$, $\beta^1(v')-\beta^1(v)>\beta^2(v')-\beta^2(v)$. 
This means
	\begin{equation}\label{eq4}
		\int_0^{v^1(\beta^1_{in})}\beta^1(v)-\beta^1_{in}\mathrm{d}v<\int_0^{v^2(\beta^2_{in})}\beta^2(v)-\beta^2_{in}\mathrm{d}v
	\end{equation}
	as $v^1(\beta^1_{in})=v^2(\beta^2_{in})$. Moreover,
	\begin{equation}\label{eq5}
		\int^0_{v^2(\beta^2_{ex})}\beta^2_{ex}-\beta^2(v)\mathrm{d}v < \int^0_{v^1(\beta^1_{ex})}\beta^1_{ex}-\beta^1(v)\mathrm{d}v,
	\end{equation}
	if $\beta^2_{ex}\leq \beta^1_{ex}$ or $-v^1(\beta^1_{ex})\geq -v^2(\beta^2_{ex})$.
	
	But by the equation \eqref{eq1} for $v^1$ along with \eqref{eq4} and \eqref{eq5}, 
	\begin{align*}
		&\int_{v^2(\beta_{ex}^2)}^0\beta^2_{ex}-\beta^2(v)\mathrm{d}v<\int^0_{v^1(\beta^1_{ex})}\beta^1_{ex}-\beta^1(v)\mathrm{d}v\\
		=&\int_0^{v^1(\beta^1_{in})}\beta^1(v)-\beta^1_{in}\mathrm{d}v<\int_0^{v^2(\beta^2_{in})}\beta^2(v)-\beta^2_{in}\mathrm{d}v,
	\end{align*}
	which contradicts the equation \eqref{eq1} for $v^2$.
	
	Therefore, it must be $\beta^1_{ex}<\beta^2_{ex}$ and $\tan\alpha^{1}_{ex}=-v^1(\beta^1_{ex})<\tan \alpha^{2}_{ex}=-v^2(\beta^2_{ex})$ from \ref{lemcate}\ref{itemconcavity}, i.e. $\beta_{ex}$ and $\alpha_{ex}$ are increasing functions of $\beta_{in}$.
	
	By the same argument, $\beta_{ex}$ and $\alpha_{ex}$ are also strictly increasing functions of $\alpha_{in}$. 
\end{proof}

\begin{assumption} 
	\label{A:ba}
	We assume from now on that $\betain\in(0,\pi/2]$, $\alphain\in(0,\pi/3]$.
\end{assumption} 

\begin{cor}\label{corgraph}
	If \ref{A:ba} holds
	then 
	$\alphain <\pi/2$ and $\alphaex[\betain,\alphain]<\pi/2$,    
	and so \ref{lemcate}\ref{itemconcavity} holds. 
\end{cor}
\begin{proof}
	First it can be seen that the result holds if and only if $\alphaex[\betain,\alphain]< \pi/2$ under the assumptions \ref{A:ba}. A direct calculation using \ref{lemcate}\ref{itemcatepara} shows that the result is correct when $\betain=\pi/2$, $\alphain=\pi/3$. Define $\beta_{in}^0$ to be 
	\begin{equation*}
		\beta_{in}^0:=\mathrm{inf}\{\betain\in(0,\pi/2]: \text{On } \KBp[\betain,\pi/3] \text{ $r$ is a single-valued function with $\abs{\mathrm{d}r/\mathrm{d}\beta}<\infty$.} \}
	\end{equation*}
	If $\beta_{in}^0>0$, then $\alphaex[\beta_{in}^0,\pi/3]\geq\pi/2$. However, for $\betain>\beta_{in}^0$, on $\KBp[\betain,\pi/3]$ $r$ is a single-valued function. Thus \ref{propmonotone} implies that $\alphaex[\betain,\pi/3]\leq\alphaex[\pi/2,\pi/3]<\pi/2$. Therefore, $\alphaex[\beta_{in}^0,\pi/3]\leq\alphaex[\pi/2,\pi/3]<\pi/2$ by the continuity, which contradicts the assumption. And thus the result holds for $\KBp[\betain,\pi/3]$ with $\betain\in(0,\pi/2]$. A similar argument for $\alphain$ then implies the result for any $\betain,\alphain$ under the assumptions \ref{A:ba}.
\end{proof}

\subsection*{Existence and Uniqueness of the Initial Configurations}

\begin{definition}[Catenoidal configurations] 
\label{defcat} 
We define a \emph{catenoidal configuration} to be a union of a disc and catenoidal annuli $\Wcal:=\bigcup_{i=0}^{k'} \Acal_i$ 
such that $\Wcal\cap \Sph^2= \bigcup_{j=1}^{k'+1} \Ccal_j$ and the following hold. 
\begin{enumerate}[label={(\roman*)},ref={(\roman*)}] 
\item For each $i=1,\dots,k'+1$, $\mathcal{C}_i$ is a parallel circle lying on  $\mathbb{S}^2$ at latitude $\beta_i$. 
\item $0<\beta_1<\dots < \beta_{k'} < \pi/2 $. 
\item 
$\mathcal{A}_0$ is a flat disk 
and $\partial \mathcal{A}_0=\mathcal{C}_1$.     
\item 
For each $i=1,\dots,k'$ we have $\partial \mathcal{A}_i=\mathcal{C}_i\cup\mathcal{C}_{i+1}$ 
and $\mathcal{A}_i = \K[\beta_i,\alpha_i^+] \cap \B^3$ for a unique $\alpha_i^+\in(0,\pi-\beta_i)$ (recall \ref{L:Kin}).  
\item For each $i=1,\dots,k'$, $\mathcal{C}_i=\mathcal{A}_{i-1}\cap\mathcal{A}_i$.
\end{enumerate}
For each $i=1,\dots,k'$ we have then $\beta_{i+1} = \betaex[\beta_i,\alpha_i^+]$ and we define $\alpha_{i+1}^- := \alphaex[\beta_i,\alpha_i^+]$ and $\alpha_1^-:=\beta_1$.  
$k'\in\N$ is called the \emph{order of $\Wcal$}. 
Finally we define the \emph{unbalancing parameters of $\Wcal$}, 
$\underline{\sigma}_\Wcal  =\{\sigma_{\Wcal,i} \}_{i=1}^{k'}\in \mathbb{R}^{k'}$,  
by requesting $e^{ \sigma_{\Wcal,i} } = \alpha^+_i / \alpha^-_i$ for $i=1,\dots,k'$;  
if $\underline{\sigma}_\Wcal$ vanishes we call $\Wcal$ \emph{balanced}.   
\end{definition}

Clearly we can try to construct catenoidal configurations by starting with a disk $\Acal_0$ and then proceeding inductively to 
construct $\Acal_i$ assuming $\underline{\sigma}_\Wcal$ given. 
It is not clear however how this process will end. 
One possibility is that some $\beta_{k'+1}$ will exceed $\pi/2$ which is what we hope for. 
Although we will prove later that there is always a finite $k'$ where this happens, 
we cannot exclude apriori the possibility that the $\Acal_i$'s become very narrow for large $i$ and the $\beta_i$'s  
form an increasing infinite sequence which is bounded above by $\pi/2$. 
This motivates the following definition. 

\begin{definition}[Complete catenoidal configurations] 
	\label{defcatcom} 
	We define a \emph{complete catenoidal configuration} to be either a catenoidal configuration as in \ref{defcat} which moreover satisfies $\beta_{k'+1} \geq \pi/2$, 
	or a pair $\Wcal := \{\underline{\mathcal{A}},\underline{\mathcal{C}} \}$, 
	where 
	$\underline{\mathcal{A}}=\{\mathcal{A}_i\}_{i\in\N_0}$ and $\underline{\mathcal{C}}=\{\mathcal{C}_j\}_{j\in\N}$, 
	and such that for $\forall k'\in\N$  the pair 
	$\left\{ \{\mathcal{A}_i\}_{i=0}^{k'} \, , \, \{\mathcal{C}_j\}_{j=1}^{k'+1} \right\}$ 
	is a catenoidal configuration as in \ref{defcat}. 
	In the latter case we say that \emph{the order of $\Wcal$} is $\infty$.  
\end{definition}

\begin{definition} 
	\label{defl}
	We define $\ell^1$ the space of $\R$-valued sequences of finite $\ell^1$ norm equipped with the $\ell^1$ norm which is defined by 
	$\|\, \{a_i\}_{i\in\N} \, : \, \ell^1 \|:=\sum_{i=1}^\infty |a_i|$. 
	We identify $\R^k$ with a subspace of $\ell^1$ by the map which sends $\{a_i\}_{i=1}^k\in\R^k$ to the sequence $\{a_i\}_{i\in\N}\in\ell^1$ 
	with $a_i=0$ for $i>k$. 
	Finally given $\underline{a}=\{a_i\}_{i\in\N}\in\ell^1$ we define $\left. \underline{a}\right|_k:= \{a_i\}_{i=1}^k\in\R^k$. 
\end{definition}

\begin{lemma}[Existence and uniqueness of {$\Wcal[\betah  ; \sigmaunder ]$}]  
	\label{lem:;}
	There is an absolute constant $\delta'_\sigma>0$ such that $\forall\betah \in(0,2\pi/7)$ and 
	$\sigmaunder =\{\sigma_i\}_{i\in\N}\in \ell^1$ 
	with $\| \underline{\sigma} : \ell^1 \|< \delta'_\sigma$, 
	there is a complete catenoidal configuration as in \ref{defcatcom} 
	which we denote by $\Wcal[\betah  ; \sigmaunder ]$, 
	and is uniquely determined by the following properties, 
	where we may use ``$ [\betah  ; \sigmaunder ]$'' to specify the quantities as in \ref{defcat} associated to 
	$\Wcal[\betah  ; \sigmaunder ]$.  
	\begin{enumerate}[label={(\roman*)},ref={(\roman*)}]
		\item 
		$\beta_1[\betah  ; \sigmaunder ] =\betah $.
		\item 
		$\sigma_{\Wcal[\betah  ; \sigmaunder ],i}=\sigma_i$ 
		for $i\le k' [\betah  ; \sigmaunder ]$ if the order $k'[\betah  ; \sigmaunder ] < \infty$ and $\forall i\in\N$ otherwise.   
	\end{enumerate}
	Moreover we have the following. 
	\begin{enumerate}[label={(\alph*)},ref={(\alph*)}]
		\item 
		If the order $k' [\betah  ; \sigmaunder ] <\infty$, then $\Wcal[\betah  ; \sigmaunder ]$  depends only on $\betah $ and $\{\sigma_i\}_{i=1}^{k'}$. 
		\item 
		\label{itemcateconfassump}
		The assumption \ref{A:ba} holds for all catenoidal annuli in $\cateConf$, 
and moreover $\cateConf\cap\R^2_+$ is the radial graph of a piecewise smooth function $\rr [ \betah; \sigmaunder] :B\to (0,1]$ (recall \ref{D:polar}), 
where $B=[0,\beta_{k'+1}]$ if the order $k' [\betah  ; \sigmaunder ] <\infty$ and $B=[0,\sup_{i\in\N} \beta_i)$ otherwise. 
		\item 
		\label{ialphacontrol}
We have $\alpha^{-}_{i+1} < \alpha^{+}_{i}$ and $\alpha^{+}_{i}<e^{\sigma_{1}+\dots+\sigma_{i}}\betah$, for $1{\leq} i {\leq} k' [\betah  ; \sigmaunder ]$ 
if the order $k'[\betah  ; \sigmaunder ] < \infty$ and $\forall i\in\N$ otherwise; 
and we have $\alpha^{+}_i=e^{\sigma_{i}}\alpha^{-}_i<e^{\sigma_{i}}\alpha^{+}_{i-1}$, 
for $2{\leq} i {\leq} k' [\betah  ; \sigmaunder ]$ if the order $k'[\betah  ; \sigmaunder ] < \infty$ and $\forall i\in\N\setminus\{1\}$ otherwise. 
	\end{enumerate}	
\end{lemma} 

\begin{proof}
	From the definition \ref{defcat} and the assumption $\betah \in(0,2\pi/7)$, $\alpha_{1}^+=e^{\sigma_1}\alpha_{1}^-=e^{\sigma_1}\beta_{1}<\pi-\beta_{1}$ 
	if $\abs{\sigma_1}$ is smaller than a constant, then the existence and uniqueness of $\mathcal{A}_1$ simply follows by \ref{L:Kin}. 
	Suppose now the catenoidal annulus $\mathcal{A}_n$ exists in the definition \ref{defcatcom}. 
	We have 
	for $i=2,\dots,n$
	\begin{align*}
		\alpha^{+}_i=e^{\sigma_{i}}\alpha^{-}_i<e^{\sigma_{i}}\alpha^{+}_{i-1},
	\end{align*}
	where the equality follows from \ref{defcat} and the inequality from \ref{lemcate}\ref{itemdecreasing}; 
	and thus by induction 
	\begin{equation*}
		\alpha^{+}_n<e^{\sigma_{1}+\dots+\sigma_n}\betah .
	\end{equation*}
	Therefore, if $\| \underline{\sigma} : \ell^1 \|$ is small enough, $\alpha^{+}_n<\pi-\betah <\pi-\beta_n$ as $\betah <2\pi/7$. The existence and uniqueness of $\mathcal{A}_{n+1}$ then follows by \ref{L:Kin}. (a) and \ref{ialphacontrol} then follow. Moreover, \ref{ialphacontrol} along with the assumption that $\betah <2\pi/7$  implies that the assumption \ref{A:ba} holds for all $\mathcal{A}_i\subset\K[\beta_i,\alpha_i^+]$ in $\cateConf$ if $\| \underline{\sigma} : \ell^1 \|$ is small enough. Therefore, by \ref{corgraph} and then \ref{L:Kex}, the function $\rr [ \betah; \sigmaunder]$ exists as the radial graph of $\cateConf\cap\R^2_+$. This proves \ref{itemcateconfassump}.
\end{proof}

\begin{prop} 
	\label{propsollimit}
	For $\betah $ and $\sigmaunder$ as in \ref{lem:;}
	the following hold for each fixed $\sigmaunder$. 
	\begin{enumerate}[label={(\roman*)},ref={(\roman*)}]
		\item \label{inoaccum}
		The order $k' [\betah  ; \sigmaunder ] <\infty$ and it is a decreasing function of $\betah $  
		with 
		$k' [\betah  ; \sigmaunder ] \to \infty$ as $\betah \to0$. 
		\item 
		\label{L:increasing}
		$\beta_i[\betah  ; \sigmaunder ]$, $\alpha^{\pm}_i[\betah  ; \sigmaunder ]$ 
		(when defined) are continuous strictly increasing functions of $\betah$.
		\item \label{isollimit}
		$\forall i\in \N$ we have 
		$\lim_{\betah \to 0}\beta_i[\betah  ; \sigmaunder ]=0$.
	\end{enumerate}
\end{prop}

\begin{proof}
	By the theory of ordinary differential equations and an inductive construction of $\cateConf$ 
	we conclude that $\beta_i[\betah  ; \sigmaunder ]$, $\alpha^+_i[\betah  ; \sigmaunder ]$, and $\alpha^-_i[\betah  ; \sigmaunder ]$ 
	are continuous functions of $\beta_{1}$ and $\underline{\sigma}$.  
	Moreover, they are strictly increasing functions of $\betah$ 
	by \ref{lem:;}\ref{itemcateconfassump}, \ref{corgraph} and \ref{propmonotone}
	and hence (ii) follows. 
	
	(iii) for $i=1$ is clear by the definition. 
	Now for $i\geq 2$, suppose the result holds for $\beta_1,\dots,\beta_{i-1}$. 
	Consider the chord tangent to $\mathcal{A}'_{i-1}$ at the position $r=1$, $\beta=\beta_{i-1}$. 
	Clearly the chord will also intersect $\mathbb{S}^1$ at $\beta=\beta_{i-1}+2\alpha^{+}_{i-1}$, 
	and thus by convexity of the catenoid 
	\begin{equation}
		\label{eqbetadiff}
		\beta_i<\beta_{i-1}+2\alpha^{+}_{i-1}.
	\end{equation}
	Therefore, the result for $\beta_i$ follows by the assumption for $\beta_{i-1}$ and \ref{lem:;}\ref{ialphacontrol}.
	
	Clearly the definition implies that $k' [\betah  ; \sigmaunder ] \in \N\cup \{\infty\}$,  
	by (ii) it is a decreasing function of $\betah $,   
	and by (iii) its limit as $\betah \to0$ is $\infty$.  
	It remains to prove that it is never $\infty$. 
	We argue by contradiction and assume that it is for some $\betah$ which can be assumed as small as needed its monotonicity. 
	$\{\beta_i\}$ is increasing and bounded by $\pi/2$ and so it converges to some $\beta_\infty\le \pi/2$. By \ref{L:BK} $\forall i\in\N$ there is a catenoid $\K_i\supset \mathcal{A}_i$.  
	Let $a_i:=a^{\K_i}$ and $b_i:=b^{\K_i}$. 
	Clearly they are both positive and uniformly bounded above.

If any subsequence of $\{\alpha_i^+\}$ have a positive lower bound, then $\sum_{i=1}^\infty \alpha_{i}^+ = \infty$. On the other hand, if $\lim_{i\to\infty}\alpha_i^+=0$, then an elementary calculation now using Taylor's expansion for $\cosh$ and the uniform bounds for $a_i$ and $b_i$ implies that 
	for some $C>0$ independent of $i$ we have 
	$$ 
	|\alpha_i^+ -\alpha_{i+1}^-| \leq C\, (\cos \beta_{i+1}-\cos \beta_{i})^2. 
	$$
	By \eqref{eqbetadiff} and \ref{lem:;}\ref{ialphacontrol} we have 
	$ \cos \beta_{i}- \cos \beta_{i+1} \le \beta_{i+1}- \beta_{i} < 2 \alpha_{i}^+$ 
	and hence
	$$ 
	\alpha_i^+ -4 C\, (\alpha_{i}^+)^2\leq \alpha_{i+1}^-=e^{-\sigma_{i+1}} \alpha_{i+1}^+. 
	$$
	which implies since $\lim_{i\to \infty} \alpha_i^+=0$ that 
	\begin{align*}
		\frac1{\alpha_{i+1}^+} \le \frac{e^{-\sigma_{i+1}}}{\alpha_{i}^+(1-4C\alpha_{i}^+)} = \frac{e^{-\sigma_{i+1}}}{\alpha_{i}^+} + \frac{4e^{-\sigma_{i+1}}C}{1-4C\alpha_{i}^+}\le e^{-\sigma_{i+1}}\left(\frac{1}{\alpha_{i}^+} +8C \right) \le e^{\|\sigmaunder : \ell^1\|} \left( \frac1{\betah} + 8Ci \right) ,
	\end{align*}
	where the last inequality follows by induction.  
	We conclude that $\sum_{i=1}^\infty \alpha_{i}^+ = \infty$ by comparing with the harmonic series.  
	
 If $\beta_\infty<\pi/2$, we can assume that $\beta_i+\alpha_i^+<(\pi/2+\beta_{\infty})/2<\pi/2$. Therefore, by \ref{lemcate}\ref{itemcatepara} and mean value theorem,
	\begin{equation} 
		\label{eqdiffa}
		a_{i}-a_{i-1}= 2\sin(\beta_i) \cos\left(\frac{2\beta_i+\alpha_i^+-\alpha_i^-}{2}\right)\sin \left(\frac{\alpha_i^++\alpha_i^-}{2} \right) >C\left(\alpha_i^- + \alpha_i^+ \right)>0
	\end{equation}
as $2\beta_i+\alpha_i^+-\alpha_i^-<\pi$. Thus $a_i$ is increasing, and there is a limit $a_{\infty}:=\lim_{n\to \infty} a_n$. 
We conclude from \ref{eqdiffa} that $\sum_{i=1}^\infty \alpha_{i}^+ < \infty$, a contradiction. 

	Hence $\beta_\infty=\pi/2$.      
	By applying \ref{L:Kex}\ref{itemcatepara} with $\betain=\beta_n$, $\alphain=\alpha_n^+$, or with $\betaex=\beta_n$, $\alphaex= \alpha_n^- = e^{-\sigma_n} \alpha_n^+$, 
	and subtracting, we obtain 
	$$
	b_{n-1}-b_{n}= 
	\sin{\beta_n}  \, \left( \, f(\beta_n + \alpha_n^+  ) - f(\beta_n - e^{-\sigma_n} \alpha_n^+  ) \, \right)
	$$ 
	where  
	$
	f(t) := \sgn(t-\pi/2)  \, \sin{t} \, \arcosh(1/{\sin{t}}).  
	$
	Since 
	$\arcosh(1/{\sin{t}}) = \log \frac{1+|\cos t|}{\sin{t}} = \frac12  \log \frac{1+|\cos t|}{1-|\cos t|}$ 
	we see that we can expand in odd powers of $|\cos t|$. 
	Clearly then $f$ is a smooth function in the vicinity of $\pi/2$  and moreover its derivative at $\pi/2$ is positive. 
	This implies by continuity and the mean value theorem that $b_{n-1}-b_n\geq C\alpha_{n}$ for large $n$, which implies 
	that $\sum_{i=1}^\infty \alpha_{i}^+ < \infty$, a contradiction. 
	This completes the proof. 
\end{proof}

\begin{definition}[Symmetrically extendible catenoidal configurations]  
	\label{defcatse} 
	We define a \emph{symmetrically extendible catenoidal configuration} to be a complete catenoidal configuration $\Wcal$ as in \ref{defcatcom} which moreover satisfies 
	either $\beta_{k'+1}=\pi/2$ or $\beta_{k'+1}+\beta_{k'}=\pi$. 
	The \emph{total order} $k$ of $\Wcal$ is defined to be $2k'+1$ in the former case and $2k'$ in the latter. 
\end{definition}

\begin{cor}[Existence and uniqueness of {$\Wcal[\sigmaunder {:} k ]$}]    
	\label{lem::}
	Given $k\in\N\setminus\{1,2\}$ and 
	$\sigmaunder =\{\sigma_i\}_{i\in\N}\in \ell^1$ 
	with $\| \underline{\sigma} : \ell^1 \|< \delta''_\sigma\in (0,\delta'_\sigma)$, 
	where $\delta'_\sigma$ is as in \ref{lem:;}, 
	there is a unique $\betah=\betah[\sigmaunder {:} k ] \in(0,\pi/2)$   
	such that 
	$\Wcal = \Wcal[\betah  ; \sigmaunder ]$ is symmetrically extendible and has total order $k$ 
	(recall \ref{defcatse}). 
	Moreover $\betah[\sigmaunder {:} k ]$ is decreasing in $k$ for fixed $\sigmaunder$, $\betah\Confarg\to 0$ when $k\to \infty$, 
	and $\betah[\sigmaunder {:} k ]$ and $\Wcal[\sigmaunder {:} k ] $  
	depend continuously on $\left. \sigmaunder \right|_{[k/2]}$ for fixed $k$. 
\end{cor}

\begin{proof}
	In this proof,
	we define an \emph{extended complete catenoidal configuration} to be a catenoidal configuration same as the catenoidal configuration in \ref{defcatcom} except it satisfies $\beta_{k'+1} > \pi/2$ and $\beta_{k'+1} \leq \pi/2$. As in \ref{lem:;}, given $\forall\betah \in(0,2\pi/7)$ and 
	$\sigmaunder =\{\sigma_i\}_{i\in\N}\in \ell^1$ 
	with $\| \underline{\sigma} : \ell^1 \|< \delta'_\sigma$, 
	there is a unique extended complete catenoidal configuration satisfying the properties in \ref{lem:;}, which we denote by $\Wcal[\![\betah  ; \sigmaunder ]\!]$. As in \ref{lem:;}, we will use ``$ [\![\betah  ; \sigmaunder ]\!]$'' to specify the quantities associated to
	$\Wcal[\![\betah  ; \sigmaunder ]\!]$.
	
	If $\beta_1<2\pi/7$ and $\| \underline{\sigma} : \ell^1 \|\leq \delta'_{\sigma}$, by \ref{propsollimit}\ref{L:increasing}, each $\beta_i[\betah  ; \sigmaunder ]$ is an increasing function of $\beta_1$, which then implies the uniqueness of $\Conf$. By \ref{lem:;}\ref{itemcateconfassump}, the assumption \ref{A:ba} holds if $\Conf$ exists.
	
	We first show the existence of the configuration $\mathcal{W}[0 {:} 3]$ exists with $\betah[0 {:} 3]\in(\pi/4,2\pi/7)$. By the definition, the configuration $\mathcal{W}[0 {:} 3]$ exists if and only if there exists a catenoidal arc $\KBp[\betain,\betain]$ such that $\betaex[\betain,\betain]=\pi/2$. 
	
	By the formulae \ref{lemcate}\ref{itemcatepara}, the catenoidal arc $\KBp[\pi/4,\pi/4]$ satisfies the equation
	\begin{equation*}
		x=\frac{\sqrt{2}}{2}\cosh\left(\sqrt{2}y-1\right),
	\end{equation*}
	which intersects $y=0$ at $(\frac{\sqrt{2}\cosh{1}}{2},0)\approx(1.0911,0)$ outside $\mathbb{S}^1$; while the catenoidal arc $\KBp[2\pi/7,2\pi/7]$ satisfies the equation 
	\begin{equation*}
		x=\sin\frac{2\pi}{7}\sin\frac{4\pi}{7}\cosh\left(\frac{y-\cos\frac{2\pi}{7}-\sin\frac{2\pi}{7}\sin\frac{4\pi}{7}\arcosh\frac{1}{\sin\frac{4\pi}{7}}}{\sin\frac{2\pi}{7}\sin\frac{4\pi}{7}}\right),
	\end{equation*}
	which intersects $y=0$ at $\left(\sin\frac{2\pi}{7}\sin\frac{4\pi}{7}\cosh\left(\cot\frac{2\pi}{7}\csc\frac{4\pi}{7}+\arcosh\frac{1}{\sin\frac{4\pi}{7}}\right),0\right)\approx(0.8996,0)$ inside $\mathbb{S}^1$. Therefore, the angle $\betaex[\pi/4,\pi/4]<\pi/2$ while $\betaex[2\pi/7,2\pi/7]>\pi/2$. The result then follows by the intermediate theorem and the monotonicity \ref{propmonotone}.Moreover, the assumption \ref{A:ba} holds.
	
	Now when $\underline{\sigma}|_1$ has absolute value smaller than a value $\delta''_\sigma$ small enough, by implicit function theorem and \ref{propsollimit}\ref{L:increasing}, there exists a unique $\betah$ near $\betah[0;3]$, such that  $\beta_2[\underline{\sigma};\beta_1]=\pi/2$. Thus $\mathcal{W}[\underline{\sigma}{ : }3]$ exits and the assumption \ref{A:ba} holds by \ref{lem:;}\ref{itemcateconfassump}.
	
	Furthermore, suppose that the results hold for $k=2k'-1$. Then for $\underline{\sigma}\in \ell^1$ with $\| \underline{\sigma} : \ell^1 \|< \delta''_\sigma$, $\mathcal{W}[\underline{\sigma}{ : }2k'-1]$ exists. Therefore, in the extended complete catenoidal configuration $\mathcal{W}[\![\betah[\underline{\sigma}{ : }2k'-1] ;\underline{\sigma}]\!]$ with finite order $k'$, $\beta_{k'}=\pi/2$ and $\beta_{k'}+\beta_{k'+1}>\pi$. As in \ref{propsollimit}\ref{L:increasing}, $\beta_{k'}[\![\betah;\underline{\sigma}]\!]$ and $\beta_{k'+1}[\![\betah;\underline{\sigma}]\!]$ are strictly increasing functions of $\betah$. By intermediate theorem and \ref{propsollimit}\ref{isollimit}, there exists a $\betah[\underline{\sigma}{ : }2k']<\betah[\underline{\sigma}{ : }2k'-1]$ such that in $\mathcal{W}[\betah[\underline{\sigma}{ : }2k'];\underline{\sigma}]=\mathcal{W}[\![\betah[\underline{\sigma}{ : }2k'];\underline{\sigma}]\!]$, $\beta_{k'}+\beta_{k'+1}=\pi$. Thus by \ref{defcatse} $\mathcal{W}[\underline{\sigma}{ : }2k']$ exists. A similar argument then shows that $\mathcal{W}[\underline{\sigma}{ : }2k'+1]$ exists.
	
	By the implicit function theorem, \ref{propsollimit}\ref{L:increasing} and \ref{defcatse}, $\betah\Confarg$ is a continuous function of $\underline{\sigma}$ in a neighborhood of each $\sigmaunder$ with $\| \underline{\sigma} : \ell^1 \|< \delta''_\sigma$. Finally, by the monotonicity of $k'$ with respect to $\betah$ in \ref{propsollimit}\ref{inoaccum} and the uniqueness of $\betah\Confarg$ in \ref{lem::}, $\betah\Confarg\to 0$ when $k\to \infty$.
\end{proof}

\begin{notation}\label{nota:confarg}
We define $\Conf := \Wcal\cup \refl\Wcal$, where $\refl$ is the reflection with respect to the $xz$-plane and $\Wcal = \Wcal[\betah\Confarg  ; \sigmaunder ]$ 
is as in \ref{defcat}. 
For $1\le i \le k$ we define $\alpha^{\pm}_i\Confarg$ and $\beta_i\Confarg$ to be the values of $\alpha_i^{\pm}$ and $\beta_i$ respectively 
in the configuration $\Wcal$ for $i\le{k':=[k/2]}$ and by $\alpha_{i}^{\pm}\Confarg:=\alpha_{k-i+1}^{\mp}\Confarg$, $\beta_i\Confarg:=\pi-\beta_{k-i+1}\Confarg$ for $i>k'$;  
when $\underline{\sigma}=0$ we take $\alpha_i\zConfarg:=\alpha^{+}_i\zConfarg=\alpha^{-}_i\zConfarg$. 
For $0\le j \le k$ and $1\le l \le k$  
we define $\mathcal{A}_j\Confarg$ and $\mathcal{C}_l\Confarg$ to denote the catenoidal annulus (or disk) $\mathcal{A}_j$ 
and circle $\mathcal{C}_l$ respectively in the configuration $\Wcal$ if $j\le k'$ or $l\le k'+1$, 
and by  
$\mathcal{A}_j\Confarg:=\refl\mathcal{A}_{k-j}\Confarg$ and $\mathcal{C}_l\Confarg:=\refl\mathcal{C}_{k-l+1}\Confarg$ otherwise. 
Finally we define $\rr\Confarg : [0,\pi]\to (0,1]$ to be the extension of $\rr [\betah\Confarg  ; \sigmaunder ]$ as in 
\ref{lem:;}\ref{itemcateconfassump} so that its radial graph is $\Conf\cap\R^2_+$. 
\end{notation}

\begin{cor} 
\label{corbal}
Assuming that $k$ is large enough in terms of a given $\epsilon>0$ and $\sigmaunder$ is as in \ref{lem::}, $\rr=\rr\Confarg$ defined as in \ref{nota:confarg} satisfies 
$$ 
\norm {\rr-1:C^{0}[0,\pi/2]}+\norm {\mathrm{d}\rr/\mathrm{d}\beta:C^{0}[0,\pi/2]}\leq\epsilon.  
$$ 
\end{cor}

\begin{proof}
	For the catenary $\K'$ corresponding to each catenoidal annulus $\mathcal{A}_i$ in $\Conf$, 
define $l_i$ to be the chord tangent to $\K'$ at $(\sin\beta_i,\cos\beta_i)$. 
The radial distance between $\mathcal{A}_i$ and the unit sphere is then controlled by the radial distance between $l_i$ and the unit circle. 
From the inequality $\alpha^+_{i}<e^{\norm{\sigmaunder:\ell_1}}\betah$ by \ref{lem:;}\ref{ialphacontrol}, 
	\begin{equation*}
		\max (1-\rr(\beta))=\max_{i=1,\dots,k}\{1-\cos\alpha^+_i \}\leq 1-\cos(e^{\delta''_{\sigma}}\betah),
	\end{equation*}
	where $\alpha^+_i:=\alpha^+_i\Confarg$.  Thus $\max (1-\rr(\beta))\to 0$ when $k\to \infty$ by \ref{lem::}. 
	
	 Similarly, from \ref{lemcate}\ref{itemdecreasing}, \ref{lemcate}\ref{itemconcavity},
	 \begin{equation*}
	 	\max\abs{\frac{\mathrm{d} \rr}{\mathrm{d}\beta}(\beta)/\rr(\beta)} 
	 	=\max_{i=1,\dots,k}\{\tan \alpha_i^+\}\leq \tan(e^{\delta''_{\sigma}}\betah).
	 \end{equation*}
Thus from the result above, $\max \abs{\mathrm{d} \rr/\mathrm{d}\beta}\to 0$ when $k\to \infty$ by \ref{lem::}.	  
\end{proof}

\begin{lemma} 
\label{lembalanceconf}
	In the balanced configuration $\zConf$,
	\begin{enumerate}[label={(\roman*)},ref={(\roman*)}]
		\item For $i=1,\dots,k$, $\alpha_i\zConfarg\leq \beta_1\zConfarg$.\label{itemalphacontrol}
		\item For $i=1,\dots,k-1$, $\beta_{i+1}\zConfarg-\beta_{i}\zConfarg\leq 2\beta_1\zConfarg$.\label{itembetadiff}
	\end{enumerate}
\end{lemma}

\begin{proof}
	By the symmetry of the balanced configuration, it can be assumed that $i\leq \frac{k+1}{2}$. (i) then just follows by the balanced condition and \ref{ialphacontrol}. (ii) just follows by \eqref{eqbetadiff} and (i).
\end{proof}

\begin{lemma} 
\label{lemcatekernel}   
If $k$ is large enough in absolute terms, then the kernel of $\mathcal{L}_{\mathcal{A}_i\zConfarg}$ with Dirichlet boundary condition 
for each disk or catenoidal annulus $\mathcal{A}_i\zConfarg$ in the balanced configuration $\zConf$ is trivial. 
\end{lemma}

\begin{proof}
	For the disks $\mathcal{A}_0\zConfarg$ and $\mathcal{A}_k\zConfarg$, the result simply follows by the smallness of $\beta_1$. Suppose $\mathcal{A}_i\zConfarg$, $i=1,\dots,k-1$ satisfies the equation $x=a_i\cosh(\frac{y-b_i}{a_i})$, let $\alpha_i:=\alpha_i\zConfarg$, $\beta_i:=\beta_i\zConfarg$. For any rotationally invariant function $u$, $\mathcal{L}_{\mathcal{A}_i\zConfarg}u=0$ is equivalent with
	\begin{equation*}
		u''(y)+\frac{2}{a^2_i}\sech^2\left(\frac{y-b_i}{a_i}\right)u(y)=0.
	\end{equation*}
	The solution space is spanned by $u=\tanh{\frac{y-b_i}{a_i}}$ and $u=1-\frac{y-b_i}{a_i}\tanh{\frac{y-b_i}{a_i}}$. The Dirichlet boundary condition is equivalent with
	\begin{equation*}
		u\left(\cos\beta_i\right)=0,\quad u\left(\cos\beta_{i+1}\right)=0.
	\end{equation*}
	The solution exists if and only if the following condition holds
	\begin{equation}\label{eqdirichlet}
		0=\det
		\begin{pmatrix}
			\tanh z_1 & 1-z_1\tanh z_1 \\
			\tanh z_2 & 1-z_2\tanh z_2 
		\end{pmatrix}=\tanh z_1-\tanh z_2+(z_1-z_2)\tanh z_1 \tanh z_2,
	\end{equation}
	where $z_1=\frac{\cos\beta_i-b_i}{a_i}$ and $z_2=\frac{\cos\beta_{i+1}-b_i}{a_i}$.
	
	When $k$ is even, for the catenoidal annulus $\mathcal{A}_{\frac{k}{2}}\zConfarg$, 
the result follows by the smallness of the catenoidal pieces and the fact that $a_{\frac{k}{2}}$ is close to $1$ from \ref{corbal}. 
For all other catenoidal annuli, in odd or even case, by the symmetry it can be assumed that $\beta_{i+1}\leq \pi/2$, then by \ref{lemcate}\ref{itemcatepara},
	\begin{align*}
		\frac{\cos{\beta_i}-\cos{\beta_{i+1}}}{a_i}=\frac{\cos{\beta_i}-\cos{\beta_{i+1}}}{\sin{\beta_{i+1}}\sin({\beta_{i+1}-\alpha_{i+1}})}\leq \frac{(\beta_{i+1}-\beta_{i})\sin{\beta_{i+1}}}{\sin{\beta_{i+1}}\sin({\beta_{i+1}-\alpha_{i+1}})}=\frac{\beta_{i+1}-\beta_{i}}{\sin({\beta_{i+1}-\alpha_{i+1}})}.
	\end{align*}
	By \ref{lemcate}\ref{itemdecreasing} and the balanced condition, $\alpha_i$ is decreasing for $i\leq \frac{k+1}{2}$, and $\beta_i$ is increasing, therefore,
	\begin{equation*}
		\frac{\cos{\beta_i}-\cos{\beta_{i+1}}}{a_i}\leq \frac{\beta_{i+1}-\beta_{i}}{\sin({\beta_{2}-\alpha_{2}})}.
	\end{equation*}
	However, if $\beta_2-\alpha_2< \beta_1$, then as $\beta_1=\alpha_{1}>\alpha_{2}$ by \ref{lemcate}\ref{itemdecreasing}, $\beta_2-\beta_1< \beta_2-\alpha_2<\beta_1$, and thus $\beta_2<2 \beta_1$. But these contradict the formula $\sin\beta_1\sin(2\beta_1)=a_1=\sin\beta_2\sin(\beta_2-\alpha_2)$ by \ref{lemcate}\ref{itemcatepara}. This means $\beta_2-\alpha_2\geq \beta_1$ and thus by \ref{lembalanceconf}\ref{itembetadiff} 
	\begin{equation*}
		\frac{\cos{\beta_i}-\cos{\beta_{i+1}}}{a_i}\leq \frac{\beta_{i+1}-\beta_{i}}{\sin{\beta_{1}}}\leq  2\frac{\beta_1}{\sin{\beta_{1}}}.
	\end{equation*}
	By \ref{lem::}, $\lim_{k\to\infty}\frac{\beta_1}{\sin{\beta_{1}}}=1$. This means $z_1-z_2=\frac{(\cos{\beta_i}-b_i)-(\cos{\beta_{i+1}}-b_i)}{a_i}\leq 2+\epsilon$ when $k$ is big enough, where $\epsilon$ is a small positive number. However, then \ref{eqdirichlet} can not hold except when $z_1=z_2$.
\end{proof}

\begin{prop} 
\label{propconvconf} 
For $k\in \N$ large enough in absolute terms there are constants 
$ \delta_{\sigma}=\delta_{\sigma}(k)\in(0,\delta''_{\sigmaunder})$ (recall \ref{lem::}),  
$c=c(k)$, 
and $\delta'_{\theta}=\delta'_{\theta}(k)$,  
such that the following hold for all $\underline{\sigma}$ with $\norm{\sigmaunder:\ell_1}\leq\delta_{\sigma}$.
	\begin{enumerate}[label={(\roman*)},ref={(\roman*)}] 
\item 
\label{itemconfalpha}
$\alpha_i^+\Confarg\in[30\delta'_{\theta},\frac{\pi}{2}-30\delta'_{\theta}]$ and $\abs{\alpha_i^-\Confarg-\alpha_i^+\Confarg}\leq \delta'_{\theta}$ 
for each $i\in\N$ with $1\le i \le k$. 
\item 
\label{itemconfA}
The catenoid $\K_i$ containing $\mathcal{A}_i\Confarg$ satisfies $a_{\K_i}> \delta'_{\theta}$   
for each $i\in\N$ with $1\le i \le k-1$. 
\item 
\label{itemconfker}
There are no Dirichlet eigenvalues of $\mathcal{L}_{\mathcal{A}_i\Confarg}$ on ${\mathcal{A}_i\Confarg}$ in the interval $[-c,c]$  
for each $i\in\N_0$ with $0\le i \le k$. 
\end{enumerate}
\end{prop}

\begin{proof}
	For the balanced case, \ref{itemconfalpha} and \ref{itemconfA} follows by choosing $\delta'_{\theta}$ small enough. 
\ref{itemconfker} follows by \ref{lemcatekernel}.  
The general case then follows by the continuous dependence of solution to $\underline{\sigma}$ in \ref{lem::} and choosing $\delta_{\sigma}$ small enough.
\end{proof}

\section{The Desingularizing Surfaces}
In this section the desingularizing surface $\preiniDsingsur\Dsingarg$ is defined, which is similar to the surfaces $\Sigma[T,\phi,\tau]$ in \cite{kapouleas:finite}, 
to which we refer for some details and proofs.  

\subsection*{The Scherk Surfaces $\Sigma(\theta)$}
\begin{definition}
\label{D:scherk} 
For $\theta\in (0,\pi/2)$, the Scherk surface $\Sigma(\theta)$ is defined by the equation (recall \ref{D:EBS}) 
\begin{equation} 
\label{eqscherk}
    \sin^2\theta \cosh{\frac{x}{\sin{\theta}}}-\cos^2\theta \cosh{\frac{y}{\cos{\theta}}}=\cos{z}.
\end{equation}
\end{definition}

\begin{notation} 
\label{n4} 
$\Vec{e}_x$, $\Vec{e}_y$, $\Vec{e}_z$ are defined to be the coordinate unit vectors of the Cartesian coordinate system $O.xyz$. 
$\Vec{e}[\theta]$, $\Vec{e}'[\theta]$ are defined by 
\begin{equation*}
    \Vec{e}[\theta]=\cos{\theta}\Vec{e}_x+\sin{\theta}\Vec{e}_y,\quad \Vec{e}'[\theta]=-\sin{\theta}\Vec{e}_x+\cos{\theta}\Vec{e}_y.
\end{equation*}
The identity map is denoted by $\refl_1$, and reflections with respect to $yz$-plane, $z$-axis, $xz$-plane are denoted by $\refl_2$, $\refl_3$, $\refl_4$ respectively. 
Let $\rot_{\phi}$ denote the rotation around the $z$-axis by an angle $\phi$. 
\end{notation}

Note that $\refl_i$ is a symmetry of each $\Sigma(\theta)$ mapping the first quadrant to the $i$th quadrant for $i=1,2,3,4$. 

\begin{assumption}
\label{eq:theta}
To ensure uniform bounds in the geometry of 
the Scherk surfaces $\Sigma(\theta)$ we assume form now on that 
$
\theta\in [10\delta_{\theta},\pi/2-10\delta_{\theta}],
$ 
where $\delta_{\theta}>0$ is a small number to be determined later.
\end{assumption}

\begin{prop}[Properties of the Scherk surfaces]  
\label{propscherk}
$\Sigma(\theta)$ is a singly periodic embedded complete minimal surface which depends smoothly on $\theta$ and moreover satisfies the following for $\theta$ 
as in Assumption \ref{eq:theta}. 
\begin{enumerate}[label={(\roman*)},ref={(\roman*)}]
       \item $\Sigma(\theta)$ is invariant under $\refl_i$ for $i=1,2,3,4$ and $\refl_{\{ z=n\pi \}}$, the reflection with respect to the plane $\{ z=n\pi \}$, 
for $n\in \mathbb{Z}$. 
       \item \label{itemscherkwing} For given $\varepsilon\in(0,10^{-3})$ there is a constant $a=a(\varepsilon,\delta_{\theta})$ and smooth functions 
$f_{\theta}:=f_{\theta,a}:\HH\to \mathbb{R}$, $A_{\theta}:=A_{\theta,a}:\HH\to E^3$, and $F_{\theta}:=F_{\theta,a}:\HH\to E^3$ 
depending smoothly on $\theta$ and $a$, 
such that $W_{\theta}:=W_{\theta,a}:=F_{\theta}(\HH)\subset \Sigma(\theta)$ and 
        \begin{align*}
            &A_{\theta}(s,z)=(a+s)\Vec{e}[\pi/2-\theta]+z\Vec{e}_z+b_{\theta}\Vec{e}'[\pi/2-\theta],\\
            &F_{\theta}(s,z)=A(s,z)+f_{\theta}(s,z)\Vec{e}'[\pi/2-\theta],
        \end{align*}
    where $b_{\theta}=\sin{2\theta}\log(\tan{\theta})$. What's more,
       \item $\Sigma(\theta)\setminus\cup_{i=1}^4\refl_i(W_{\theta})$ is connected and lies within distance $a+1$ from the $z$-axis.
       \item $W_{\theta}\subset\{(r\cos{\phi},r\sin{\phi},z):r>a,\phi\in[9\delta_{\theta},\pi/2-9\delta_{\theta}],z\in \mathbb{R} \}$.
       \item \label{itemscherkdecay} $\norm{f_{\theta}:C^5(\HH,e^{-s})}\leq \varepsilon$ and $\norm{\mathrm{d}f_{\theta}/\mathrm{d}\theta:C^5(\HH,e^{-s})}\leq \varepsilon$.
       \item \label{itemscherkb} $\abs{b_{\theta}}+\abs{\mathrm{d}b_{\theta}/\mathrm{d}\theta}\leq \varepsilon a$.
       \item \label{itemscherkboundary} $\partial\Sigma(\theta)_{x\leq 0}$ is contained in the $yz$-plane, which is a disjoint union of topologically circles. 
     \end{enumerate}
\end{prop}
\begin{proof}
\ref{itemscherkboundary} follows by the equation \eqref{eqscherk}. All the others are the same as \cite[Proposition 2.4]{kapouleas:finite}. 
\end{proof}

\begin{definition}[Core, wings, and coordinates]
\label{def:wing}
The coordinates $(s,z)=(s_a,z)$ on $\refl_i(W_{\theta})$, $i=1,2,3,4$ are defined by $(s,z)=(\refl_i\circ F_{\theta})^{-1}$. 
$\refl_i(W_{\theta})$ will be called the $i$th wing of $\Sigma(\theta)$. 
The function $s$ is extended to a continuous function on $\Sigma(\theta)$, by requiring $s$ to vanish on $\Sigma(\theta)\setminus \cup_{i=1}^4\refl_i(W_{\theta})$. 
The part $\Sigma(\theta)_{s\leq 0}$ will be called the core of $\Sigma(\theta)$. 
Finally, the third and second wing of $\Sigma(\theta)$ will be called the $+$ and $-$ wing of half Scherk surface $\Sigma(\theta)_{x\leq 0}$ 
and the part $\Sigma(\theta)_{s\leq 0,x\leq 0}$ will be called the core of half Scherk surface $\Sigma(\theta)_{x\leq 0}$.
\end{definition}

\begin{prop}[The Gauss map of the Scherk surfaces]  
\label{propgauss}
The Gauss map $\nu$ of $\Sigma(\theta)$ has the following properties:
\begin{enumerate}[label={(\roman*)},ref={(\roman*)}]
\item $\nu$ restricts to a diffeomorphism from $\Sigma(\theta)_{z\in[0,\pi]}$ onto 
\begin{equation*}
    \mathbb{S}^2\cap \{z\geq 0\}\setminus \{(\pm\cos{\theta},\pm \sin{\theta},0)\}.
\end{equation*}
\item Let $E_i$, $i=1,2,3,4$ be the arcs into which the equator $\mathbb{S}^2\cap \{z=0\}$ is decomposed by removing the points $(\pm\cos{\theta},\pm \sin{\theta},0)$, numbered so that $(1,0,0)\in E_1$, $(0,1,0)\in E_2$, $(-1,0,0)\in E_3$, and $(0,-1,0)\in E_4$, then 
\begin{equation*}
    \nu(\Sigma(\theta)_{z=0})=E_1\cup E_3,\quad \nu(\Sigma(\theta)_{z=\pi})=E_2\cup E_4.
\end{equation*}
\item $\Sigma(\theta)$ has no umbilics and $\nu^*g_{\mathbb{S}^2} = \frac{1}{2}\abs{A}^2_{\Sigma(\theta)}g_{\Sigma(\theta)}$. 
\end{enumerate}
\end{prop}

\begin{proof}
See \cite[Proposition 2.6]{kapouleas:finite}.  
\end{proof}

\begin{definition} 
\label{defh}
For $\Sigma=\Sigma(\theta)$ we define $h:=\frac{1}{2}\abs{A}_{\Sigma}^2 g_{\Sigma} = \nu^*g_{\mathbb{S}^2} $ and $\mathcal{L}_h:=\Delta_h+2$.
\end{definition}

The following proposition for the eigenvalue of $\mathcal{L}_h$ is the same as \cite[Proposition 2.8]{kapouleas:finite}.  
We are interested in the eigenfunction invariant under the reflection with respect to the $yz$-plane $\refl_2$. 

\begin{prop}[Eigenvalues on Scherk surfaces modulo symmetries]  
\label{propscherkef}
Let $G$ be the group generated by reflections across the planes $\{z=0\}$ and $\{z=\pi\}$. 
Let $\Sigma:=\Sigma(\theta)$. 
There is an $\epsilon_{\lambda}=\epsilon_{\lambda}(\delta_{\theta})>0$, 
such that the only bounded eigenfunctions of $\mathcal{L}_h$ on $\Sigma/G$ whose eigenvalues lie in $[-\epsilon_{\lambda},\epsilon_{\lambda}]$, 
are the ones in the kernel of $\mathcal{L}_h$ on $\Sigma/G$. 
This kernel is two dimensional and is spanned by $\nu \cdot \Vec{e}_x$, $\nu \cdot \Vec{e}_y$. 
In particular, the only bounded eigenfunctions of $\mathcal{L}_h$ on $\Sigma/G$ whose eigenvalue 
lies in $[-\epsilon_{\lambda},\epsilon_{\lambda}]$ invariant under $\refl_2$ have the form $\lambda\nu \cdot \Vec{e}_y$, $\lambda\in \mathbb{R}$.
\end{prop}
\begin{proof}
See the proof of \cite[Proposition 2.8]{kapouleas:finite} .
\end{proof}

\subsection*{The Surfaces $\iniDsingsur\abDsingarg$}

The definition of $Z_y[\phi]$ below is similar to \cite[Definition 3.2]{kapouleas:finite} 
except for the addition of requirement \ref{defZy}\ref{itemZyid} to ensure that a neighborhood of the boundary of the half Scherk surface stays invariant.

\begin{definition} 
\label{defZy}
We fix a smooth family of diffeomorphisms $Z_y[\phi]:E^3\to E^3$ parametrized by $\phi\in [-2\delta_{\theta},2\delta_{\theta}]$ (recall \ref{eq:theta}) 
satisfying the following. 
\begin{enumerate}[label={(\roman*)},ref={(\roman*)}]
    \item \label{itemZyid} $Z_y[\phi]$ is the identity map on $L(\delta_{\theta}):=\{(x,y,z):\abs{x}\leq \frac{1}{2}\sin{2\delta_{\theta}}\}\cup \{(x,y,z):y\leq 0\}$.
    \item On $\{(r\cos{\theta'},r\sin{\theta'},z ):r>1,\theta'\in[9\delta_{\theta},\pi/2-9\delta_{\theta}] \}$ $Z_y[\phi]$ acts by rotation around the $z$-axis by $\phi$.
    \item \label{Zysym} $Z_y[\phi]$ is equivariant under $\refl_2$ (reflection with respect to the $yz$-plane).
\end{enumerate}
\end{definition}

The definition of $\iniDsingsur\abDsingarg$ below is similar to the one for $\Sigma[T]$ in \cite[Definition 3.4]{kapouleas:finite} with 
$T=\{\vec{e}[\eta_i]\}_{i=1}^4$, where $\eta_1=\pi-\alpha^--\beta$, $\eta_2=\pi+\alpha^--\beta$, $\eta_3=2\pi-\alpha^+-\beta$, $\eta_4=2\pi+\alpha^+-\beta$. 
Therefore in the notation of \cite{kapouleas:finite} $\theta(T)=\pi/2-\alpha^+$, $\theta_1(T)=-\theta_2(T)=(\alpha^--\alpha^+)/2$, $\theta_x(T)=0$, and $\theta_y(T)=\alpha^--\alpha^+$. 
The tetrad is defined so that the $yz$-plane is mapped by the appropriate bending map 
to a cone tangent to a sphere at its parallel circle of latitude $\beta$.

\begin{definition} 
\label{def:abdsing}
Given $\beta\in[0,\pi]$ and $\underline{\alpha}:=(\alpha^-,\alpha^+)\in \R^2$ such that $\alpha^+\in [20\delta_{\theta},\frac{\pi}{2}-20\delta_{\theta}]$ and 
$\alpha^+-\alpha^-\in [-2\delta_{\theta},2\delta_{\theta}]$ (recall \ref{eq:theta}) 
we define the surface 
\begin{equation*}
    \iniDsingsur\abDsingarg := \rot_{\pi/2-\beta}\circ \bend(\Scherk).
\end{equation*}
Moreover, $\rot_{\underline{\alpha},\beta,i}$ is defined to be the rotation by which $\rot_{\pi/2-\beta}\circ \bend$ acts on the $i$th wing of $\Scherk$. 
The images of the core and the $i$th wing of $\Scherk$ under $\rot_{\pi/2-\beta}\circ \bend$ will be called the core and the $i$th wing of $\iniDsingsur\abDsingarg$. 
The function $s$ on $\iniDsingsur\abDsingarg$ is defined to be the push forward of $s$ on $\Scherk$ by $\rot_{\pi/2-\beta}\circ \bend$. 
Finally to simplify the notation we take 
$    \iniDsingsur\aDsingarg:=\iniDsingsur[\underline{\alpha},0]$.
\end{definition}

\subsection*{The Desingularizing Surfaces $\preiniDsingsur\Dsingarg$}
\begin{definition}[The bending map $\mathcal{B}_{\tau}$ and the spheres $\mathscr{S}_{\tau,\beta}$]  
\label{D:Stb}
For $\tau\in \mathbb{R}$ and $\beta\in (0,\pi)$ we define $\mathcal{B}_{\tau}: \R^3\to \R^3$ and $\mathscr{S}_{\tau,\beta}\subset\R^3$  
by 
\begin{equation*}
\begin{gathered}
    \mathcal{B}_{\tau}(x,y,z) :=(\tau^{-1}+x)(\cos{\tau z},0,\sin{\tau z})+(-\tau^{-1},y,0),
\\ 
\mathscr{S}_{\tau,\beta} := \left\{ (x,y,z)\in\R^3 \left| (x+\tau^{-1})^2+(y+\frac{\cos\beta}{\tau\sin\beta})^2+z^2=(\frac{1}{\tau\sin\beta})^2 \right. \right\},   
\end{gathered}
\quad\text{ when } \tau\ne0, 
\end{equation*}
and by taking $\mathcal{B}_0$ to be the identity and $\mathscr{S}_{0,\beta}$ to be the $yz$-plane rotated by $\rot_{\pi/2-\beta}$. 
\end{definition}

If we bend the Scherk surfaces using $\mathcal{B}_{\tau}$, 
we would get wings which decay to cones, which are not minimal surfaces. 
To avoid this we bend instead the planes into catenoids as in \cite{kapouleas:finite}. 
This leads to the following definitions.

\begin{definition}[The bent asymptotic hapf-planes] 
\label{defwing}
Given $\theta\in [20\delta_{\theta},\pi/2-20\delta_{\theta}]$, $\phi\in[-\delta_{\theta},\delta_{\theta}]$, $R$ a Euclidean motion fixing the $z$-axis, 
and $\tau\in[0,1)$, define $A[\theta,\phi,R,\tau]:\HH\to E^3$ as follows: if $\tau=0$ 
\begin{equation*}
    A[\theta,\phi,R,0](s,z)=R'\circ R\circ A_{\theta}(s,z),
\end{equation*}
where $R'$ is the rotation by an angle $\phi$ around the line $R\circ A_{\theta}(\partial \HH)$. Otherwise,
\begin{equation*}
    A[\theta,\phi,R,\tau](s,z)=(r_0+\tau^{-1})(\cosh{\tau s}+\cos{\eta}\sinh{\tau s})(\cos{\tau z},0,\sin{\tau z})+(-\tau^{-1},y_0+s(1+r_0\tau)\sin{\eta},0),
\end{equation*}
where $(r_0,y_0)$ is chosen such that
\begin{equation*}
    \{x=r_0,y=y_0\}=R\circ A_{\theta}(\partial \HH),
\end{equation*}
and $\eta$ is determined by the requirement that $\Vec{e}[\eta]=\rot_{\phi}\circ R (\Vec{e}[\theta])$. Notice that under such choice $A[\theta,\phi,R,\tau]$ and $B_{\tau}\circ R\circ A_{\theta}$ agree on $\partial \HH$. The image of $\partial \HH$ is independent of $\phi$ and will be called the \emph{pivot} of $A[\theta,\phi,R,\tau]$.
\end{definition}

\begin{prop} 
\label{propwing}
There exists a small positive number $\delta'_{\tau}=\delta'_{\tau}(\delta_{\theta})$ such that when $\tau\leq \delta'_{\tau}$ the map $A[\theta,\phi,R,\tau]$ satisfies the following:
\begin{enumerate}[label={(\roman*)},ref={(\roman*)}]
    \item \label{itemwingmin} 
It is a conformal minimal immersion which depends smoothly on the parameters. 
The image lies on a catenoid or a plane. 
Moreover the pivot is the circle of radius $\tau^{-1} + r_0$ , centered at $(-\tau^{-1},y_0,0)$ , and parallel to the $xz$-plane. 
    \item $\Vec{e}[\eta]$ is the inward conormal of the image at $A[\theta,\phi,R,\tau](0,0)$.
\end{enumerate}
\end{prop}
\begin{proof}
These follow by \cite[Lemma 3.9(iv) and Definition 3.10]{kapouleas:finite}.
\end{proof}
Therefore, when $\tau\leq \delta'_{\tau}$, the Gauss map $\nu[\theta,\phi,R,\tau]$ could be chosen such that it depends smoothly on the parameters and satisfies the orientation choosing $\nu[\theta,0,R,0]\equiv R(\Vec{e}'[\pi/2-\theta])$. 

\begin{definition}[The bent wings] 
\label{D:wings} 
Fix $\delta_s$ a small positive constant which will be determined later, for $\theta$, $\phi$ as before, $R$ a Euclidean motion fixing $z$-axis, and $\tau\in[0,\delta'_{\tau}]$, define $F[\theta,\phi,R,\tau]:\HH\to E^3$ by
\begin{multline*}
    F[\theta,\phi,R,\tau](s,z)=\psi[1,0](s)\mathcal{B}_{\tau}\circ R \circ F_{\theta}(s,z)\\
    +(1-\psi[1,0](s))(A_{\theta}[\theta,\phi,R,\tau](s,z)+\psi_s(s)f_{\theta}(s,z)\nu[\theta,\phi,R,\tau](s,z)),
\end{multline*}
where $\psi_s(s):=\psi[4\delta_s\tau^{-1},3\delta_s\tau^{-1}](s)$. 
\end{definition}

We consider now the boundary of the image of the half Scherk surface under the bending map 
$\mathcal{B}_{\tau}\circ \rot_{\pi/2-\beta}$, where we assume $\tau>0$ and $\beta\in(0,\pi)$.  
The image of the $yz$-plane under this map, where this boundary lies, is the cone $\left\{ (x+\tau^{-1})^2+z^2=(\cot\beta y-\tau^{-1})^2 \right\} $,  
and not, as we would like, the sphere 
$\mathscr{S}_{\tau,\beta}$ (recall \ref{D:Stb}). 
Notice that the sphere $\mathscr{S}_{\tau,\beta}$ is tangent to the cone above along the circle which is the image of the $z$-axis under $\mathcal{B}_{\tau}\circ \rot_{\pi/2-\beta}$. 
On the other hand, the preimage of the sphere $\mathscr{S}_{\tau,\beta}$ under the same map is the cylinder 
\begin{equation} 
\label{E:cyl} 
\mathscr{C}_{\tau,\beta} := \left\{ (x,y,z) \in \R^3 \left| (x+\frac{1}{\tau\sin{\beta}})^2+y^2= (\frac{1}{\tau\sin{\beta}})^2 \right. \right\}. 
\end{equation} 
We would like therefore to modify the Scherk surface in the vicinity of the $yz$-plane so that the boundary of the half Scherk surface lies on $\mathscr{C}_{\tau,\beta}$; 
this leads to the definition of $\mathcal{D}[\beta,\tau]$ in \ref{defdeform}.

\begin{lemma}[Neighborhoods of the boundary of the half Scherk surfaces] 
\label{L:predeform} 
There exist constants $c_d$, $\epsilon_d=\epsilon_d(\delta_{\theta})$ and $\delta''_{\tau}=\delta''_{\tau}(\delta_\theta)$ 
such that when $\tau\in(0,\delta''_{\tau}]$, $\beta\in(0,\pi)$ the following hold. 
\begin{enumerate}[label={(\roman*)},ref={(\roman*)}]
\item $2c_d\epsilon_d\leq \frac{1}{4}\sin{2\delta_{\theta}}$.
\item 
The intersection of $\iniDsingsur\aDsingarg$ with the slab $R_{\epsilon_d} :=\{(x,y,z)|x\in[-c_d\epsilon_d,c_d\epsilon_d], y\in [-c_d,c_d]\}$ 
(or with $2R_{\epsilon_d}$) 
is a periodic union of annuli.
Moreover both intersections are contained in the core of $\iniDsingsur\aDsingarg$.
\item $\iniDsingsur\aDsingarg\cap \mathscr{C}_{\tau,\beta}\subset R_{\epsilon_d}$  (recall \eqref{E:cyl}). 
\end{enumerate}
\end{lemma}

\begin{proof}
From the equation of the Scherk surface and the definition of $\iniDsingsur\aDsingarg$ and \ref{defZy}\ref{itemZyid}, 
$\iniDsingsur\aDsingarg_{x=0}$ is the union of topological circles satisfying the equation $\sin^2{\theta}-\cos{z}=\cos^2{\theta}\cosh{\frac{y}{\cos{\theta}}}$ on the $yz$-plane. 
Therefore, $\abs{y}< 2$ on $\iniDsingsur\aDsingarg_{x=0}$. 
Thus $c_d$ can be chosen to be $2$ and then the existence of $\epsilon_d$ and $\delta''_{\tau}$ follows. 
\end{proof}

\begin{lemma}[Adjusting the boundary]  
\label{defdeform}
There is a smooth family of diffeomorphisms $\mathcal{D}[\beta,\tau]:E^3\to E^3$ parametrized by $\beta\in [0,\pi]$, $\tau\in [0,\delta''_{\tau}]$ such that:
\begin{enumerate}[label={(\roman*)},ref={(\roman*)}]
\item $\mathcal{D}[\beta,\tau]$ is the identity map on $E^3_{x\leq 0}\setminus 2R_{\epsilon_d}$.
\item When $\beta>0$, $\tau>0$, 
\begin{equation*}
    \mathcal{D}[\beta,\tau](x,y,z)=\left(\frac{2\epsilon_d c_d+\sqrt{(\frac{1}{\tau\sin{\beta}})^2-y^2}-\frac{1}{\tau\sin{\beta}}}{2\epsilon_d c_d}x+\sqrt{(\frac{1}{\tau\sin{\beta}})^2-y^2}-\frac{1}{\tau\sin{\beta}},y,z\right)
\end{equation*}
on $R_{\epsilon_d}\cap E^3_{x\leq 0}$.
\item $\mathcal{D}[\beta,\tau]$ is the identity map when $\beta=0$, $\beta=\pi$ or $\tau=0$.
\item $\norm{\mathcal{D}[\beta,\tau]-\mathrm{id}_{E^3}:C^4(E^3)}\leq C\tau\sin{\beta}$ for some constant $C:=C(\delta_{\theta})$.\label{itemdeformtau}
\item\label{itemdeformimage} 
The image of $\iniDsingsur\aDsingarg_{x=0}$ under $\mathcal{D}[\beta,\tau]$ is contained in the cylinder $\mathscr{C}_{\tau,\beta}$.
\end{enumerate}
\end{lemma}
\begin{proof}
	The results are clear when $\beta=0$ or $\tau=0$. In the other cases, $\mathcal{D}[\beta,\tau]$ can be constructed by connecting the map in (ii) with identity map in (i) by a cut-off function. (v) then follows by the lemma above. (iv) follows by an expansion of the formula in (ii).
\end{proof}	

We define now the maps which we will use to define the desingularizing surfaces which unlike in \cite{kapouleas:finite} use only half a Scherk surface. 
We also set the $\phi_i$ angles symmetrically so that when $\tau=0$, 
the image of a Scherk surface under the map $\wrap\Dsingarg$ is symmetric with respect to the image of the $yz$-plane. 

\begin{definition}[The maps $\wrap\Dsingarg$] 
\label{def:Z} 
Given $\underline{\alpha}:=(\alpha^-,\alpha^+)\in \R^2$, $\beta$, $\underline{\phi}=(\phi^-,\phi^+) \in\R^2$, and $\tau$ satisfying  
(with $\abs{\underline{\phi}}:=\max\{\abs{\phi^-} , \abs{\phi^+} \}$) 
\begin{align*}
    \alpha^+\in[30\delta_{\theta},\frac{\pi}{2}-30\delta_{\theta}],\quad \alpha^+-\alpha^-\in[-\delta_{\theta},\delta_{\theta}],\quad 
\beta\in[\delta_{\theta},\pi-\delta_{\theta}],  \quad \abs{\underline{\phi}}\leq \delta_{\theta},\quad \tau\in [0,\delta''_{\tau}],
\end{align*}
we define 
$\wrap\Dsingarg = \wrap_a\Dsingarg : \Scherk\to E^3$ by 
(recall \ref{def:wing} and for the rotation $\rot_{\underline{\alpha},\beta,i}$ \ref{def:abdsing})
\begin{equation*}
    \wrap_a\Dsingarg := \mathcal{B}_{\tau}\circ \rot_{\pi/2-\beta}\circ\mathcal{D}[\beta,\tau]\circ \bend 
\qquad 
\text{on the core of} \quad \Scherk, 
\end{equation*}
and by requesting that for $i=1,2,3,4$ we have on $\HH$ 
with $(\phi_1,\phi_2,\phi_3,\phi_4):=(\phi^-,-\phi^-,\phi^+,-\phi^+)$ that 
\begin{align*}
    \wrap_a\Dsingarg\circ \refl_i\circ F_{\alpha^+}=F_i\Dsingarg&:=F[\alpha^+,\phi_i,\rot_{\underline{\alpha},\beta,i}\circ\refl_i,\tau],\\
    A_i\Dsingarg&:=A[\alpha^+,\phi_i,\rot_{\underline{\alpha},\beta,i}\circ\refl_i,\tau]. 
\end{align*}

When $\tau=0$ and $\beta=\pi/2$, $\mathcal{Z}\apDsingarg=\mathcal{Z}_a\apDsingarg$ and $\iniDsingsur\apDsingarg=\iniDsingsur_a\apDsingarg$ are defined by
\begin{equation*}
 \mathcal{Z}_a\apDsingarg:=\wrap_a[\underline{\alpha},\pi/2,\underline{\phi},0],\quad \iniDsingsur_a\apDsingarg:=\mathcal{Z}_a\apDsingarg(\Scherk).
\end{equation*}
In particular, when $\underline{\phi}=0$, $\iniDsingsur[\underline{\alpha},\underline{\phi}]=\iniDsingsur\aDsingarg$.

Finally the \emph{$i$th pivot} of $\wrap\Dsingarg$ is defined to be $A_{i}\Dsingarg(\partial \HH)$.
\end{definition}

\begin{prop}[Properties of $\wrap\Dsingarg$] 
\label{propZ}
Each $\wrap\Dsingarg$ as in \ref{def:Z} satisfies the following.
\begin{enumerate}[label={(\roman*)},ref={(\roman*)}]
    \item It is a smooth immersion depending smoothly on its parameters.
    \item \label{itemZsym} For each $n\in \mathbb{Z}$, it is equivariant under the reflection of the domain with respect to the plane $\{z=n\pi \}$, 
and of the range with respect to the plane which is parallel to the $y$-axis, contains $(-\tau^{-1},0,0)$, 
and forms an angle $n\tau^{-1}\pi$ with the positive $x$-axis when $\tau\neq 0$ or is parallel to the $x$-axis when $\tau=0$.
    \item \label{itemZbd} 
The boundary components of $\wrap\Dsingarg(\Scherk_{x\leq 0})$ are contained in $\mathscr{S}_{\tau,\beta}$ (recall \ref{D:Stb}).  
    \item If $\tau^{-1}\in \mathbb{N}$, then $\wrap\Dsingarg(\Scherk_{x\leq 0})$ is an embedding surface.
    \item Under the assumption above, $\wrap\Dsingarg(\Scherk_{x\leq 0})$ contains $2\tau^{-1}$ fundamental regions in (ii), 
and the circles $\wrap\Dsingarg(\Scherk_{x\leq 0,s=\frac{5\delta_s}{\tau}})$ have catenoidal neighborhoods.
\end{enumerate}
\end{prop}

\begin{proof}
(iii) is followed by the fact that $\mathcal{B}_{\tau}\circ \rot_{\pi/2-\beta}(\mathscr{C}_{\tau,\beta})=\mathscr{S}_{\tau,\beta}$ when $\tau\neq 0$ and \ref{defdeform}\ref{itemdeformimage}. All others follow by the definitions.
\end{proof}

\begin{definition}[The desingularizing surfaces] 
\label{defdesing} 
Given $\Dsingarg$ as above, and $\tau^{-1}\in \mathbb{N}$, the smooth embedded surface with boundary $\preiniDsingsur\Dsingarg$ is defined by (recall \ref{def:Z}) 
\begin{equation*}
   \preiniDsingsur:= \preiniDsingsur\Dsingarg:=\wrap\Dsingarg(\Scherk_{x\leq 0,s\leq\frac{5\delta_s}{\tau}}).
\end{equation*}
The function $s$ on $\preiniDsingsur$ is defined to be the push forward by $\wrap\Dsingarg$ of $s$ on $\iniDsingsur_{x\leq 0}$. 
The \emph{$\pm$ wing} is defined to be the image under $\wrap\Dsingarg$ of the third or the second wing of $\Scherk_{x\leq 0,s\leq\frac{5\delta_s}{\tau}}$. 
Moreover, the \emph{$\pm$ boundary circle $\partial_{\pm}\Dsingarg$} of $\preiniDsingsur$ is defined to be the image of $\Scherk_{x\leq 0,s=\frac{5\delta_s}{\tau}}$ on the $\pm$ wing, 
and the \emph{$\pm$ pivot circle $\mathcal{C}_{\pm}\Dsingarg$} is defined to be the third or the second pivot of $\wrap\Dsingarg$. 
The \emph{$\pm$ catenoidal region $\Omega_{\pm}\Dsingarg$} is defined to be the image of $A_3\Dsingarg$ and $A_2\Dsingarg$ respectively. 
Finally, the \emph{Neumann boundary of $\preiniDsingsur$} is defined by
\begin{equation*}
    \partial_n \preiniDsingsur:=\wrap\Dsingarg(\Scherk_{x= 0}).
\end{equation*}
\end{definition}

\section{The Mean Curvature of the Desingularizing Surfaces}

In this section, the conditions
\begin{align}\label{eqcondition}
    \alpha^+\in[30\delta_{\theta},\frac{\pi}{2}-30\delta_{\theta}],\quad \alpha^+-\alpha^-\in[-\delta_{\theta},\delta_{\theta}],\quad \beta\in [2\delta_{\theta},\pi-2\delta_{\theta}], \quad \abs{\underline{\phi}}\leq \delta_{\theta},\quad \tau\in [0,\delta''_{\tau}],
\end{align}
for the parameters $\Dsingarg$ are always assumed. The constant $C$ could depend on $\delta_{\theta}$ and other quantities which will be denoted.

\subsection*{The substitute kernel and the functions $w$}\qquad\\ 
The mean curvature introduced by the bending which changes $\iniDsingsur(\alpha^+)$ to $\iniDsingsur\abDsingarg$ is close to a scalar multiple of $w$ we define next. 
Unlike in \cite{kapouleas:finite}, there is only one such function in this case because we have only two wings. 
The $w$'s for the desingularizing surfaces span the substitute kernel. 

\begin{definition}[The function $w$] 
\label{defw}
Let $H_{\phi}$ be the mean curvature of the surface $Z_y[\phi](\Scherk)$ and then $w:\Scherk\to \mathbb{R}$ is defined by
\begin{equation*}
    w:=\left. \frac{\mathrm{d}}{\mathrm{d}\phi}\right\vert_{\phi=0}H_{\phi}\circ Z_y[\phi].
\end{equation*}
The push forward of $w$ on $\iniDsingsur[\underline{\alpha}]$ from $\Scherk$ by $\bend$ will also be denoted by $w$.
\end{definition}

\begin{lemma}[Properties of $w$] 
\label{lemsubkernel}
The function $w$ defined on $\iniDsingsur:=\Scherk$ satisfies the following:
\begin{enumerate}[label={(\roman*)},ref={(\roman*)}]
    \item \label{itemsubkernalsupp} $w$ is supported on $\iniDsingsur_{s\leq 0}$ and $\norm{w:C^{0,\alpha}(\iniDsingsur)}\leq C$. 
    \item \label{itemw} 
$w$ is identical to $0$ on $L(\delta_{\theta})\cap \Scherk$ (recall \ref{defZy}\ref{itemZyid}). 
    \item \label{itemsubkernelproj} 
Let $\mathcal{P}:\mathbb{R}\to V$, where $V$ is the span of $\nu\cdot\Vec{e}_y$ on $\Scherk$, 
be the linear map which assigns to $\eta$ the orthogonal projection in $L^2(\iniDsingsur,\abs{A}^2_{\iniDsingsur}g_{\iniDsingsur}/2)$ of $\eta w/\abs{A}^2_{\iniDsingsur}$ into $V$. 
$\mathcal{P}$ is then invertible and $\mathcal{P}^{-1}\leq C$. 
\end{enumerate}
\end{lemma}
\begin{proof}
(i) and (ii) follow by the definitions of $w$ and $Z_y[\phi]$. 
(iii) follows by using the balancing formula to check that the integral
\begin{equation*}
    \frac{\mathrm{d}}{\mathrm{d\phi}}\int H_{\phi}\circ Z_y[\phi]\nu\cdot\Vec{e}_y,
\end{equation*}
is bounded below by a positive absolute constant similarly to the proof of \cite[Lemma 4.19(iii)]{kapouleas:finite}. 
\end{proof}

The next corollary is simpler than the corresponding \cite[Lemma 7.4]{kapouleas:finite}  
as we are on the Scherk surface now instead of the desingularizing surface. 

\begin{cor} 
\label{corsubker}
Let $\iniDsingsur:=\Scherk$, there is a positive constant $C$ such that given $E\in L^2(\iniDsingsur/G,h)$ there is $\theta_E$ such that $(E-\theta_E w)/\abs{A}_{\iniDsingsur}^2$ is $L^2(\iniDsingsur/G,h)$-orthogonal to the kernel of $\mathcal{L}_h$ invariant under $\refl_2$, moreover, 
\begin{equation*}
    \abs{\theta_E}\leq C\norm{E/\abs{A}^2_{\iniDsingsur}:L^2(\iniDsingsur/G,h)}.
\end{equation*}
\end{cor}
\begin{proof}
The result follows by \ref{lemsubkernel}\ref{itemsubkernelproj} and \ref{propscherkef}.
\end{proof}

\subsection*{The extended substitute kernel and the functions $\baru_\pm$ and {${\barw_\pm}$}}
We discuss now the functions $\barw_\pm$ on the desingularizing surfaces and the functions $\baru_\pm$ on {${\iniDsingsur[\underline{\alpha}]}$}. 
The functions $\barw_\pm$ on the desingularizing surfaces together with the substitute kernel span the extended substitute kernel $\skernel$ on the initial surfaces (see \ref{D:skernel}). 
Unlike in \cite{kapouleas:finite}, we need just two such functions for each desingularizing surface,   
and we also need to define the related functions $\baru_\pm$ only on the models $\iniDsingsur[\underline{\alpha}]$. 
The choice of $\underline{\alpha}'$ and $\underline{\phi}'$ in the following lemma is made such that for $i=1,2,3,4$  
the asymptotic planes of the $i$th wings of $\iniDsingsur[\underline{\alpha}',\underline{\phi}']$ and $\iniDsingsur[\underline{\alpha}]$ are parallel to each other. 

\begin{lemma} 
\label{lemgraph}
There is $\delta_{\phi}=\delta_{\phi}(\delta_{\theta})\in (0,\delta_{\theta})$ such that for a given $[\underline{\alpha},\underline{\phi}']\in \mathbb{R}^4$, where $\abs{\underline{\phi}'}\leq \delta_{\phi}$, there is an $\underline{\alpha}'$ which depends smoothly on $\underline{\alpha},\underline{\phi}'$ satisfying condition \ref{eqcondition}, and is characterized by the following properties:
\begin{enumerate}[label={(\roman*)},ref={(\roman*)}]
    \item $\underline{\alpha}'=\underline{\alpha}$ when $\underline{\phi}=0$.
    \item For $i,j=\pm$,
    \begin{equation*}
       \frac{\partial\alpha'^{i}}{\partial\phi'^{j}}=\delta_{ij}.
    \end{equation*}
   \item \label{itemgraphf} 
There is a smooth function $f_{\underline{\phi}'}:=f_{\underline{\phi}',a}$ on $\iniDsingsur[\underline{\alpha}]$ invariant under $\refl_2$ 
which depends smoothly on $\underline{\alpha},\underline{\phi}'$, 
and its graph on $\iniDsingsur[\underline{\alpha}]$ is contained in $\iniDsingsur[\underline{\alpha}',\underline{\phi}']$ (recall \ref{def:Z}). 
\end{enumerate}
\end{lemma}

\begin{proof}
By \ref{defwing} and \ref{propscherk}\ref{itemscherkdecay}, for given $\underline{\alpha}$, $\underline{\alpha}'$ and $\underline{\phi}'$ as in the lemma 
with $\abs{\underline{\alpha}-\underline{\alpha}'}\leq 2\delta_{\phi}$, 
there is a function $f$ on $\iniDsingsur\aDsingarg$ which depends smoothly on $\underline{\alpha}$, $\underline{\alpha}'$ and $\underline{\phi}'$, 
and whose graph over $\iniDsingsur\aDsingarg$ is contained in $\iniDsingsur[\underline{\alpha}',\underline{\phi}']$. 
The results then follow by choosing $\alpha'^{\pm}=\alpha^{\pm}+ \phi'^{\pm}$.
\end{proof}

We define now the functions $\baru$ and $\barw$ similarly to \cite{kapouleas:finite}. 
The negative signs in the definitions are chosen to improve the expressions in \ref{lemextker}\ref{itemextkerucomp} and \ref{propmeancurvdesing}.  

\begin{definition}[The functions $\baru_{\pm}$ and $\barw_{\pm}$]  
\label{defubar} 
\label{defwbar}
The function $\baru_{\pm} = \baru_{\pm} \aDsingarg = \baru_{\pm,a} \aDsingarg : \iniDsingsur[\underline{\alpha}] \to \mathbb{R}$ is defined by 
\begin{equation*}
    \baru_{\pm}:=-\left.\frac{\partial}{\partial  \phi'^{\pm}}\right\vert_{\underline{\phi'}=0} f_{\underline{\phi'}}.
\end{equation*}
We denote the pull back of $\baru_{\pm}$ to $\Scherk$ by $\bend$ also by $\baru_{\pm}= \baru_{\pm} \aDsingarg $ 
and we define the function $\barw_{\pm} =\barw_{\pm}\aDsingarg =\barw_{\pm,a}\aDsingarg :\Scherk\to \mathbb{R}$ by 
\begin{equation*}
    \barw_{\pm}:=-\mathcal{L}_{\Sigma}\baru_{\pm}.
\end{equation*}
Finally we denote  
the push forwards of $w$, $\baru_{\pm}$, $\barw_{\pm}$ on $\iniDsingsur\apDsingarg$ and $\preiniDsingsur\Dsingarg$ by $\mathcal{Z}\apDsingarg$ and $\wrap\Dsingarg$ 
also by $w$, $\baru_{\pm}$, $\barw_{\pm}$.
\end{definition} 

Before we discuss further the functions $\baru_{\pm}$ and $\barw_{\pm}$ in \ref{lemextker} we have two useful lemmas.

\begin{lemma} 
\label{lemcyleq} 
For $\alpha,\gamma\in (0,1)$,
let the cylinder $\Omega:=[0,\infty)\times\mathbb{R}/G'$, where $G'$ is the group generated by $(s,z)\to(s,z+2\pi)$. Then there is an $\varepsilon_0=\varepsilon_0(\alpha,\gamma)$, such that for any $\varepsilon<\varepsilon_0$, and correspondingly $a=a(\varepsilon)$, $s=s_a$, $F:= F_{\theta,a}\circ R$, $R$ is the reparametrization $(s,z)\to (s+s_0,z)$, where $s_0$ is any positive number, there is a linear map 
\begin{equation*}
    \mathcal{R}_{\Omega}:C^{0,\alpha}(\Omega,g_0,e^{-\gamma s})\to C^{2,\alpha}(\Omega,g_0,e^{-\gamma s}),
\end{equation*}
where $g_0$ is the flat metric on $\Omega$, and a constant $C=C(\alpha,\gamma)$, such that for $v=\mathcal{R}_{\Omega}(E)$ the following are
true, where the constants $C$ depends only on $\alpha$ and $\gamma$
\begin{enumerate}[label={(\roman*)},ref={(\roman*)}]
    \item $\underline{\mathcal{L}}v=E$,
    where the operator $\underline{\mathcal{L}}$ on $\Omega$ defined by:
    \begin{equation*}
        \underline{\mathcal{L}}:=\Delta_{F^*g_{\Sigma}}+\abs{A}^2_{\Sigma} \circ F.
    \end{equation*}
    \item $v$ is a constant on $\partial \Omega$.
    \item $\norm{v:C^{2,\alpha}(\Omega,g_0,e^{-\gamma s})}\leq C\norm{E:C^{2,\alpha}(\Omega,g_0,e^{-\gamma s})}$.
\end{enumerate}
\end{lemma}

\begin{proof}
By \ref{propscherk}\ref{itemscherkdecay}, the quantity 
\begin{equation} 
\label{eq:NL}
	N(\underline{\mathcal{L}}):=\norm{F^*g_{\Sigma}-g_0:C^2(\Omega,g_0,e^{-s})}+\norm{\abs{A}^2_{\Sigma} \circ F:C^2(\Omega,g_0,e^{-s})}\leq C\varepsilon.
\end{equation}		
The lemma then follows by using \cite[Proposition A.4]{kapouleas:finite}.
\end{proof}

\begin{lemma} 
\label{lemcylequnique} 
For $\alpha,\gamma\in (0,1)$, let the cylinder $\Omega$ as before. 
Then there is an $\varepsilon_0=\varepsilon_0(\alpha,\gamma)$, such that for any $\varepsilon<\varepsilon_0$, and correspondingly $a=a(\varepsilon)$, $s=s_a$, 
$F:=\refl_i\circ F_{\theta,a}\circ R$, $R$ is the reparametrization $(s,z)\to (s+s_0,z)$, where $s_0$ is any positive number, and for any $f\in C^{0,\alpha}(\partial\Omega,F^*h)$, 
there is a unique function $v\in C^{2,\alpha}(\Omega,F^*h)$ with $\norm{v:C^{2,\alpha}(\Omega,F^*h)}< \infty$, such that
\begin{enumerate}[label={(\roman*)},ref={(\roman*)}]
    \item $\underline{\mathcal{L}}v=0$,
    where the operator $\underline{\mathcal{L}}$ on $\Omega$ is as \ref{lemcyleq}. 
    \item $v=f$ on $\partial \Omega$.
\end{enumerate}  
Moreover, the limit $v_0:=\lim_{s\to \infty}v$ exists. And 
\begin{equation*}
     \norm{v-v_0:C^{2,\alpha}(\Omega,F^*g_{\Sigma},e^{-\gamma s})}\leq C\norm{f:C^{0,\alpha}(\Omega,F^*h)}.
    \end{equation*}
\end{lemma}

\begin{proof}
By \ref{defh}, $\underline{\mathcal{L}}v=0$ is equivalent to $\mathcal{L}_{F^*h}v=0$. 
Then by \ref{propgauss}, the problem could be transformed in to the Dirichlet problem on a small disk on $\mathbb{S}^2$, 
and the result then follows by the standard theory as Laplacian has no small eigenvalue for the disk when the radius is small enough and \eqref{eq:NL}. 
The estimate follows by the standard estimate, \ref{defh} and \ref{propscherk}\ref{itemscherkdecay}.
\end{proof}

Notice that our definition of ``wings'', and thus the construction of desingularizing surfaces, depend on the choice of the parameter $a$ (or $\varepsilon$) in \ref{propscherk}. 
The choice of $\varepsilon$ will not be specified until the study of linearized equations, therefore, we will mention the dependence on $\varepsilon$ explicitly until its choice is specified. 

\begin{lemma}[Properties of $\baru_{\pm}$ and $\barw_{\pm}$]  
\label{lemextker}
There exists $\varepsilon_0$, such that for $\varepsilon<\varepsilon_0$ and corresponding $a(\varepsilon)$, 
the functions $\baru_{\pm}$ and $\barw_{\pm}$ defined on $\iniDsingsur:=\Scherk$ satisfy the following: 
\begin{enumerate}[label={(\roman*)},ref={(\roman*)}]
\item They depend continuously on $\underline{\alpha}$ and $a$.
\item They are invariant under $\refl_2$. In particular, on $\iniDsingsur_{x=0}$, $\partial_x \baru_{\pm}=0$.
\item $\barw_{\pm}$ is supported on $\iniDsingsur_{s\leq 1}$.\label{itemextkersupp}
\item 
\label{itemextkerubdd} 
$\norm{\baru_{\pm}:C^{2,\alpha}(\iniDsingsur)}\leq Ca$  and $\norm{\barw_{\pm}:C^{0,\alpha}(\iniDsingsur)}\leq Ca$. 
\item 
$\norm{\baru_{\pm}: C^{2,\alpha}(\iniDsingsur,h)}< \infty$ and for $i,j=1,2,3,4$, $a_{ij}:=\lim_{s\to \infty}\baru_{i}$ exist on the $j$th wing, 
where $(\baru_1,\baru_2,\baru_3,\baru_4):=(\baru_-,\baru_-,\baru_+,\baru_+)$. 
Moreover
    \begin{equation*}
        \abs{a_{ij}-\delta_{ij}a}\leq C\varepsilon a.
    \end{equation*}
   \label{itemextkerucomp}
\end{enumerate}
\end{lemma}

\begin{proof}
(i) and (ii) follow by the definitions and \ref{lemgraph}\ref{itemgraphf}. 
The minimality of $\iniDsingsur\apDsingarg_{s\geq 1}$ from the definition implies that $\barw_{\pm}=0$ on $\iniDsingsur_{s\geq 1}$. This implies (iii).

From \ref{lemgraph}, \ref{propscherk}\ref{itemscherkwing} and \ref{defwing}, the distance between the asymptotic planes $A_i[\underline{\alpha}',\pi/2,\underline{\phi}',0]$ and $A_i[\underline{\alpha},\pi/2,0,0]$ is ($-$ for $i=1,2$, $+$ for $i=3,4$)
\begin{equation*}
    \abs{b_{\alpha'{+}}\cos{\phi'^{\pm}}-a\sin{\phi'^{\pm}}-b_{\alpha^+}}.
\end{equation*}
Therefore, from \ref{propscherk}\ref{itemscherkdecay}, \ref{propscherk}\ref{itemscherkb}, for $s=s_a$ defined on $\iniDsingsur\aDsingarg$,
\begin{equation} 
\label{eqfphi}
    \norm{f_{\underline{\phi}'}+a\sin{\phi'^{\pm}}:C^4(\iniDsingsur\aDsingarg_{s\geq 1,\mp y>0})}\leq C\varepsilon a.
\end{equation}
This implies (iv) when $s\geq 1$ along with \ref{defwbar}.

Now let $\varepsilon_0$ be a fixed small number and $a_0:=a(\varepsilon_0)$, $s_{a_0}$ correspondingly. Then the result
\begin{equation*}
    \norm{\baru_{\pm,a_0}:C^{2,\alpha}(\iniDsingsur)}\leq C=C(\varepsilon_0)
\end{equation*}
follows by the construction in \ref{defubar} and \eqref{eqfphi} when $s_{a_0}\geq 1$ and compactness when $s_{a_0}\leq 1$. For $a\geq a_0$, from the definitions in \ref{propscherk}\ref{itemscherkwing} and \ref{lemgraph},  $f_{\underline{\phi}',a}=f_{\underline{\phi}',a_0}$ when $s_{a_0}\leq 0$. 

On the other hand, the distance between the two asymptotic planes $A_{i,a_0}[\underline{\alpha}',\pi/2,\underline{\phi}',0]$ and $A_{i,a}[\underline{\alpha},\pi/2,\underline{\phi}',0]$ is $(a-a_0)\sin{\phi^{\pm'}}$ when $s_a\geq 0$, i.e. $s_{a_0}\geq (a-a_0)\cos{\phi^{\pm'}}$; and when $0\leq s_{a_0}\leq (a-a_0)\cos{\phi^{\pm'}}$, the asymptotic plane  $A_{i,a}[\underline{\alpha}',\pi/2,\underline{\phi}',0]$ is a graph of $A_{i,a_0}[\underline{\alpha},\pi/2,0,0]$ with graph function $s_{a_0}\sin{\phi^{\pm'}}$. So by \ref{propscherk}\ref{itemscherkdecay} and \ref{propscherk}\ref{itemscherkb} again (here $s_{a_0}$ is defined on $\iniDsingsur[\underline{\alpha}',\underline{\phi}']$)
\begin{equation*}
    \norm{f_{\underline{\phi}',a}-f_{\underline{\phi}',a_0}:C^4(\iniDsingsur[\underline{\alpha}',\underline{\phi}']_{0\leq s_{a_0},\mp y>0})}\leq\norm{s_{a_0}\sin{\phi'^{\pm}}:C^4(\iniDsingsur[\underline{\alpha}',\underline{\phi}']_{0\leq s_{a_0},\mp y>0})}+ C(\varepsilon_0) a.
\end{equation*}
The region $s\leq 1$ on $\iniDsingsur\aDsingarg$ is the same as the region $s_{a_0}\leq a-a_0+1$ and the function $s_{a_0}$ on $\iniDsingsur\aDsingarg$ 
is similar with the pull back of the function $s_{a_0}$ on $\iniDsingsur[\underline{\alpha}',\underline{\phi}']$ 
as a graph of $\iniDsingsur\aDsingarg$ as in \ref{lemextker} when $\underline{\phi}'$ is small. 
Therefore, 
$ \norm{\frac{\baru_{\pm,a}}{a}:C^{2,\alpha}(\iniDsingsur_{s_{a_0} \leq a-a_0+2})}\leq \frac{a_0}{a}\norm{\frac{\baru_{\pm,a_0}}{a_0}:
C^{2,\alpha}(\iniDsingsur_{s_{a_0} \leq a-a_0+2})}+\frac{a-a_0+2}{a}+C(\varepsilon_0)\leq C(\epsilon_0)$. 
This then implies (iv) when $s\leq 1$ along with \ref{defwbar}.

By (iv), $\baru_{\pm}\in H^2(\iniDsingsur,h)$ and thus $\baru_{\pm}\in C^{0,\alpha}(\iniDsingsur,h)$. 
Moreover, by (iii) and (iv), $\barw_{\pm}\in C^{0,\alpha}(\iniDsingsur,h)$. 
Therefore, by standard theory and the definition of $\barw_{\pm}$, $\baru_{\pm}\in C^{2,\alpha}(\iniDsingsur,h)$. 
By \ref{lemcylequnique} and (iii), the limits in (v) exist. 
Finally, \eqref{eqfphi} implies the last in equality in (v) along with the definition \ref{defubar}.
\end{proof}

\begin{notation} 
\label{nota:apm}
From \ref{lemextker}\ref{itemextkerucomp}, we define
	\begin{align*} 
		&a_{--}:=a_{22},\quad  a_{++}:=a_{33},\\
		&a_{-+}:=a_{23},\quad  a_{+-}:=a_{32}.
	\end{align*}
\end{notation}

\subsection*{The Decomposition of the Mean Curvature}

\begin{lemma}[Estimates on the geometry of the catenoidal annuli] 
\label{lemOmega}
Given $\Omega_{\pm}:=\Omega_{\pm}\Dsingarg$ as in \ref{defdesing} 
we have the following. 
\begin{enumerate}[label={(\roman*)},ref={(\roman*)}]
    \item $\norm{g_{\Omega_{\pm}}-g^0_{\pm}:C^l(\Omega_{\pm},1+s)}\leq C(l)\tau$ and $\norm{g_{\Omega_{\pm}}-g^0_{\pm}:C^l(\Omega_{\pm})}\leq C(l)\delta_s$, where $g^0_{\pm}$ is the push forward of the flat metric on $\HH$ by $A_3\Dsingarg$ and $A_2\Dsingarg$ respectively.\label{itemgOmega}
    \item $\norm{A_{\Omega_{\pm}}:C^l(\Omega_{\pm})}\leq C(l)\tau$.\label{itemAOmega}
\end{enumerate}
\end{lemma}
\begin{proof}
These follow by the definition of ${\Omega_{\pm}}$ and \ref{defwing}.
\end{proof}

\begin{corollary} 
\label{lemZ}
Given $\tau^{-1}\in \mathbb{N}$, $\preiniDsingsur':=\preiniDsingsur\Dsingarg$, $\iniDsingsur:=\Scherk$, $\wrap:=\wrap\Dsingarg $. Suppose $\tau$, $\abs{\underline{\phi}}$ small enough, then
\begin{enumerate}[label={(\roman*)},ref={(\roman*)}]
    \item $\norm{A_{\preiniDsingsur'}:C^3(\preiniDsingsur')}\leq C(\varepsilon)$.\label{itemASigma}
    \item $\norm{\wrap^* g_{\preiniDsingsur'}-g_{\iniDsingsur}:C^3(\iniDsingsur)}\leq C(\varepsilon)(\tau+\abs{\underline{\phi}}+\abs{\alpha^+-\alpha^-})+C(\delta_s+\varepsilon)$.\label{itemZg}
    \item $\norm{\wrap^*A_{\preiniDsingsur'}-A_{\iniDsingsur}:C^3(\iniDsingsur)}\leq C(\varepsilon)(\tau+\abs{\underline{\phi}}+\abs{\alpha^+-\alpha^-})+C\varepsilon$.\label{itemZA}
\end{enumerate}
\end{corollary}
\begin{proof}
On $\iniDsingsur_{s\leq 1}$, these follow by the smooth dependence on the parameters on a compact set in \ref{propZ}. On $\iniDsingsur_{s\geq 1}$, \ref{itemZg} follows by \ref{lemOmega}\ref{itemgOmega} and \ref{propscherk}\ref{itemscherkdecay}, and \ref{itemASigma}, \ref{itemZA} follow by \ref{lemOmega}\ref{itemAOmega} and \ref{propscherk}\ref{itemscherkdecay}. 
\end{proof}

\begin{lemma}[The mean curvature on the wings] 
\label{lemHwing}
Given $\gamma\in(0,1)$, $\tau^{-1}\in \mathbb{N}$, $\preiniDsingsur:=\preiniDsingsur\Dsingarg$, if $\tau$ small enough with respect to $\gamma$, then the mean curvature $H$ on $\preiniDsingsur=\preiniDsingsur\Dsingarg$ satisfies
\begin{equation*}
    \norm{H_{\preiniDsingsur}:C^2(\iniDsingsur_{s\geq 1},e^{-\gamma s})}\leq C\tau.
\end{equation*}
\end{lemma}

\begin{proof}
This follows by \ref{lemOmega}, \ref{propscherk}\ref{itemscherkdecay}, and \cite[Lemma B.1]{kapouleas:finite} similarly to the proof of \cite[Lemma 4.5]{kapouleas:finite}.
\end{proof}

\begin{prop}[The mean curvature $H$ on the desingularizing surfaces] 
\label{propmeancurvdesing}
$\phantom{l}$ 
There are 
\hfill 
$\phantom{l}$ 
\\ 
$\delta_{\phi}=\delta_{\phi}(\delta_{\theta})\in(0,\delta_{\theta})$ and $\delta_{\tau}=\delta_{\tau}(\delta_{\theta})\in(0,\delta''_{\tau})$ 
such that for given $\gamma\in(0,1)$, $\Dsingarg$ satisfying \eqref{eqcondition}, and
\begin{align*}
     \abs{\underline{\phi}}\leq \delta_{\phi},\quad \tau\in (0,\delta_{\tau}],\quad \tau^{-1}\in\mathbb{N},
\end{align*}
the mean curvature $H$ on $\preiniDsingsur=\preiniDsingsur\Dsingarg$ satisfies
\begin{align*}
    \norm{H-(\alpha^+-\alpha^--\phi^++\phi^-)w-\phi^+\barw_+-\phi^-\barw_-:C^{0,\alpha}(\preiniDsingsur,e^{-\gamma s})} 
   \,  \leq \, C(\varepsilon)(\tau+\abs{\alpha^+-\alpha^-}^2+\abs{\underline{\phi}}^2).
\end{align*}
\end{prop}

\begin{proof}
On $\iniDsingsur_{s\geq 1}$, this follows by \ref{lemHwing} along with \ref{lemsubkernel}\ref{itemsubkernalsupp} and \ref{lemextker}\ref{itemextkersupp}. 

On $\iniDsingsur_{s\leq 1}$, this follows by comparing the four surfaces 
$\iniDsingsur\Dsingarg$, $\iniDsingsur[\underline{\alpha},\beta,\underline{\phi},0]$, $\iniDsingsur[\underline{\alpha}-\underline{\phi}]$ and $\iniDsingsur(\alpha^+-\phi^+)$, 
the smooth dependence on the parameters in \ref{propZ}, 
the definitions of $w$ and $\barw^{\pm}$ in \ref{defw}, \ref{defwbar} and \ref{lemgraph}\ref{itemgraphf} along with \cite[Lemma B.1]{kapouleas:finite}.
\end{proof}

\section{The Initial Surfaces}
\subsection*{The construction of the pre-initial surfaces}\quad\\
In this section, the parameter $k$ in the initial configurations is fixed. So are $\delta'_{\theta}=\delta_{\theta}'(k)$ and $\delta_{\sigma}=\delta_{\sigma}(k)$ as in \ref{propconvconf} and all the constants depending on it. We then choose $\delta_{\theta}$ in \ref{n4} to be the same as $\delta'_{\theta}$. 
Moreover, $k$ is assumed to be big enough. 
We will construct the pre-initial surfaces as the initial surfaces in \cite{kapouleas:finite}.
To simplify the notation, in this article (unlike in \cite{kapouleas:finite}) we use the same $\tau$ for all desingularizing surfaces.  

\begin{definition} 
\label{deftau}
\label{defdelta}
In the rest of this article we take 
$\delta:=\delta_{\theta}/5$ and $\tau:=1/m$,  
where $\tau$ is assumed as small (equivalently $m$ big enough) as needed and such that 
\begin{equation} 
\label{eqtau}
    \tau\leq \min\{\delta/3\underline{c}, \delta_{\sigma}/3\underline{c}, \delta_{\tau}\},
\end{equation}
where $\delta_{\tau}$ is as \ref{propmeancurvdesing}, $\delta_{\sigma}$ is as \ref{propconvconf}, and  
the constant $\underline{c}>0$ will be determined later. 
Finally constants denoted by $C$ will depend on $k$ and any other quantities we mention.
\end{definition}

\begin{definition}[Parameters of the initial surfaces] 
\label{defXi}
We define the following vector spaces and we will use the maximum norms on them except for $\mathcal{V}_{\sigma}$ in (i). 
\begin{enumerate}[label={(\roman*)},ref={(\roman*)}]
    \item $\mathcal{V}_{\sigma} :=\{\underline{\sigma}\in\mathbb{R}^{k}|\sigma_{k-i+1}=-\sigma_i,i=1,\dots,[\frac{k}{2}]\}$ with 
$\abs{\underline{\sigma}}:=\| \underline{\sigma} : \ell^1 \|$.
    \item $\mathcal{V}_{\theta}:=\{\underline{\theta}\in\mathbb{R}^{k}|\theta_{k-i+1}=-\theta_i,i=1,\dots,[\frac{k}{2}]\}$. 
    \item $\mathcal{V}_{\phi}:=\{\underline{\varphi}=\{ \underline{\phi}_i\}_{i=1}^k\in\mathbb{R}^{2k} 
| \underline{\phi}_i=(\phi^-_i,\phi^+_i)\in \mathbb{R}^2,\phi^-_i=\phi^+_{k-i+1},\phi^+_i=\phi^-_{k-i+1},i=1,\dots,[\frac{k}{2}]\}$. 
\item 
$\mathcal{V}:=\mathcal{V}_{\sigma}\times \mathcal{V}_{\phi}$ and $\mathcal{V}':=\mathcal{V}_{\theta}\times \mathcal{V}_{\phi}$. 
\end{enumerate}
Finally we define $\Xi_{\mathcal{V}}:=\{\xi\in \mathcal{V}:\abs{\xi}\leq \underline{c} \tau  \}$.
\end{definition}

Given $\xi=(\underline{\sigma},\underline{\varphi})\in\Xi_{\mathcal{V}}$, the initial surface $M=M[\xi]=M\insurfarg$ could be constructed by the following. For the initial configuration $\Conf$, there is a unique composition of a homothety followed by a translation which maps the circle $\mathcal{B}_{\tau}$($z$-axis) onto $\mathcal{C}_i\Confarg$ (recall \ref{nota:confarg}). This map will be called $\mathcal{H}_i$ suppressing the parameters $\Confarg$ on which it depends. Moreover, each $\mathcal{C}_i\Confarg$ has latitude $\beta_i\Confarg$, and there are two angles $\alpha_i^{\pm}\Confarg$. The angles $\underline{\alpha}_i({\xi})$, $\beta_i(\xi)$ are defined by the relations:
 \begin{equation}\label{eqalphaxi}
     \underline{\alpha}_i({\xi})=(\alpha_i^{-}({\xi}),\alpha_i^{+}({\xi})):=(\alpha_i^{-}\Confarg-\phi_i^-,\alpha_i^{+}\Confarg-\phi_i^+),\quad \beta_i(\xi):=\beta_i\Confarg.
 \end{equation} 
 Then the surface $\mathring{\mathcal{S}}_i[\xi]$ is defined by 
 \begin{equation}
    \mathring{\mathcal{S}}_i[\xi]:=\mathcal{H}_i(\preiniDsingsur\inDsingarg),
 \end{equation}
where $\underline{\tilde{\phi}}_i(\xi)=\{\tilde{\phi}_i^-(\xi), \tilde{\phi}_i^+(\xi)\}$ will be determined later. 

From \ref{defdesing}, the $\pm$ pivot of $\Sigma\inDsingarg$ does not depend on $\underline{\tilde{\phi}}_i(\xi)$, hence $\mathcal{C}'_{\pm,i}[\xi]$ could be defined by
\begin{equation*}
    \mathcal{C}'_{\pm,i}[\xi]:=\mathcal{H}_i(\mathcal{C}_{\pm}\inDsingarg).
\end{equation*}

When $i=1,\dots,k-1$, $\mathcal{A}_i\Confarg$ is an annulus and \begin{equation*}
   \partial \mathcal{A}_i\Confarg=\mathcal{C}_i\Confarg\cup\mathcal{C}_{i+1}\Confarg.
\end{equation*}
For each $i=1,\dots,k-1$, there exists a unique catenoidal annulus $\mathcal{A}'_i[\xi]$, which is a small perturbation of $\mathcal{A}_i\Confarg$, such that 
\begin{equation*}
   \partial \mathcal{A}'_i[\xi]=\mathcal{C}'_{+,i}[\xi]\cup\mathcal{C}'_{-,i+1}[\xi].
\end{equation*}
Then there is a unique value for each $\tilde{\phi}_1^+$, $\tilde{\phi}_k^-$ and $\tilde{\phi}^{\pm}_i$, $i=2,\dots,k-1$ such that \begin{equation*}
    \mathcal{H}_{i+1}(\partial_{-}\iinDsingarg)\subset \mathcal{A}'_{i}[\xi],\quad \mathcal{H}_i(\partial_{+}\inDsingarg)\subset \mathcal{A}'_i[\xi]
\end{equation*}
This means the two surfaces $\mathring{\mathcal{S}}_i[\xi]$ and $\mathcal{A}'_i[\xi]$ (or $\mathcal{A}'_{i-1}[\xi]$) match together except a neighborhood of $\mathcal{C}'_{+,i}[\xi]$ (or
$\mathcal{C}'_{-,i}[\xi]$) in $\mathcal{A}'_i[\xi]$ (or $\mathcal{A}'_{i-1}[\xi]$). And $\mathcal{A}''_i[\xi]$ is defined by
\begin{align*}
   & \mathcal{A}''_i[\xi]:=\mathcal{A}'_i[\xi]\setminus \\
    &\left(\mathcal{H}_i\circ A_+\inDsingarg(\, [0, 4\delta_s/\tau] \times\R \, ) \, \cup \, 
\mathcal{H}_{i+1}\circ A_-\iinDsingarg(\, [0, 4\delta_s/\tau] \times \R \, ) \right)
\end{align*}

In the case $i=0$ or $k$, $\mathcal{A}_i\Confarg$ is a planar disk, and $\partial \mathcal{A}_0\Confarg=\mathcal{C}_1\Confarg$, $\partial \mathcal{A}_k\Confarg=\mathcal{C}_k\Confarg$. For $i=1,k$, there are unique maps $\mathcal{H}[c_i^+,c_i]:E^3\to E^3$, $\mathcal{H}[c_i^+,c_i](x,y,z):=c_i^+(x,y+c_i,z)$ such that 
\begin{equation*}
    \mathcal{H}[c_1^+,c_1](\mathcal{C}_1\Confarg)=\mathcal{C}'_{-,1}[\xi],\quad \mathcal{H}[c_k^+,c_k](\mathcal{C}_k\Confarg)=\mathcal{C}'_{+,k}[\xi],
\end{equation*}
and we define
\begin{equation*}
   \mathcal{A}'_{0}[\xi]:= \mathcal{H}[c_1^+,c_1](\mathcal{A}_0\Confarg),\quad \mathcal{A}'_{k}[\xi]:=\mathcal{H}[c_k^+,c_k](\mathcal{A}_k\Confarg).
\end{equation*}
As before, $\tilde{\phi}^-_{1}$ and $\tilde{\phi}^+_{k}$ are defined such that
\begin{equation*}
    \mathcal{H}_{1}(\partial_{-}\finDsingarg)\subset \mathcal{A}'_{0}[\xi],\quad \mathcal{H}_k(\partial_{+}\linDsingarg)\subset \mathcal{A}'_k[\xi],
\end{equation*}
and $\mathcal{A}''_0[\xi]$, $\mathcal{A}''_k[\xi]$ are defined by
\begin{align*}
\mathcal{A}''_0[\xi]&:=\mathcal{A}'_0[\xi]\setminus (\mathcal{H}_1\circ A_-\finDsingarg(\, [0, 4\delta_s/\tau] \times \R \, ) ) 
\\ 
\mathcal{A}''_k[\xi]&:=\mathcal{A}'_k[\xi]\setminus (\mathcal{H}_k\circ A_+\linDsingarg(\, [0, 4\delta_s/\tau] \times \R \, ) ).
\end{align*}
    
Now all angles $\tilde{\phi}^{\pm}_{i}$ have been defined.

\begin{definition}[The pre-initial surfaces] 
\label{def:preini} 
The surfaces $\mathring{M}[\xi]$, $M'[\xi]$ and $\mathring{S}[\xi]$ are defined by
\begin{align*}
    \mathring{M}=\mathring{M}[\xi]=\mathring{M}\insurfarg&:=\cup_{i=1}^k\mathring{\mathcal{S}}_i[\xi]\cup_{j=0}^k\mathcal{A}''_j[\xi],\\
    M'=M'[\xi]=M'\insurfarg&:=\cup_{j=0}^k\mathcal{A}'_j[\xi],\\
    \mathring{S}=\mathring{S}[\xi]=\mathring{S}\insurfarg&:=\cup_{i=1}^k\mathring{\mathcal{S}}_i[\xi].
\end{align*}

The push forward of the function $s$ by $\mathcal{H}_i$ to each $\mathring{\mathcal{S}}_i$ could be extended to a function on the whole $\mathring{M}$ by taking $s=\max_{S} s$ on the rest of $\mathring{M}$. The function will also be denoted by $s$. Similarly, $s$ can be defined on $M'$.

$\tilde{\underline{\sigma}}=\tilde{\underline{\sigma}}(\xi)\in \mathcal{V}_{\sigma}$ and $\tilde{\underline{\varphi}}=\tilde{\underline{\varphi}}(\xi)\in\mathcal{V}_{\phi}$ are defined as the following:
\begin{enumerate}[label={(\roman*)},ref={(\roman*)}]
    \item $\tilde{\underline{\sigma}}:=\{\tilde{\sigma}_i\}_{i=1}^{k}$ where $\tilde{\sigma}_i:=\log\frac{\alpha^+_i(\xi)}{\alpha^-_i(\xi)}$.
    \item $\tilde{\underline{\varphi}}:=\{\tilde{\underline{\phi}}_i\}_{i=1}^{k}$, where $\tilde{\underline{\phi}}_i(\xi)$ has already been defined.
\end{enumerate}

Finally let $\underline{a}:=8\abs{\log\tau}=8\log m$. Therefore for $\mathring{M}[\xi]$ as above, each component of $\mathring{M}_{s\geq \underline{a}}$ contains exactly one $\mathcal{A}_j''[\xi]$. This component will be called $\mathcal{N}_j[\xi]$. 
\end{definition}

\begin{prop}[Properties of the pre-initial surfaces] 
\label{proppreini}
$\mathring{M}=\mathring{M}[\xi]$ with $\xi\in \Xi_{\mathcal{V}}$ is well defined and has the following properties:
\begin{enumerate}[label={(\roman*)},ref={(\roman*)}]
    \item $\mathring{M}$ is a smooth embedding surface which depends smoothly on $\xi$.
    \item $\mathring{M}$ is contained in the unit ball $\mathbb{B}^3$ and $\partial\mathring{M}$ is contained in the unit sphere $\mathbb{S}^2$, which is a disjoint union of $km$ topologically circles.
    \item $\mathring{M}$ respects the symmetries in $\Grp$.
    \item $\abs {\underline{\varphi}+\underline{\tilde{\varphi}}}\leq C(\varepsilon)\tau$.\label{itempreiniphi}
    \item $\mathring{M}_{s\geq1}$ is a graph over $M'_{s\geq1}$ by a function $f_{M}$ such that $\norm{f_{M}:C^5(M'_{s\geq \underline{a}})}\leq \tau^3$.
\end{enumerate}
\end{prop}
\begin{proof}
The facts that $\mathring{M}$ is well defined and satisfies (i) and (ii) follow by the construction and \ref{propZ}\ref{itemZbd}. (iii) follows by \ref{propZ}\ref{itemZsym} and the symmetry of the initial configuration $\Conf$. (iv) follows by the fact that the conormals at the boundary of each $\mathcal{A}'_i$ differ the corresponding ones for $\mathcal{A}_i$ by $C(\varepsilon)\tau$ since the corresponding boundaries have been moved by $C(\varepsilon)\tau$. (v) follows by the construction and rescaling \ref{propscherk}\ref{itemscherkdecay}.
\end{proof}

\begin{cor}[Estimates on the parameters] 
\label{corini} 
Suppose $\tau$ is small enough, the parameters of each
\begin{equation*}
	\preiniDsingsur\Dsingarg=\mathcal{H}^{-1}_i(\mathring{\mathcal{S}}_i[\xi])
\end{equation*}
satisfy the following:
\begin{enumerate}[label={(\roman*)},ref={(\roman*)}]
    \item $50\delta\leq \alpha^+\leq \frac{\pi}{2}- 50\delta$.\label{iteminialpha}
    \item $\abs{\alpha^+-\alpha^-}\leq 3\underline{c}\tau\leq \delta$.\label{iteminidifalpha}
    \item $\abs{\underline{\phi}}\leq (\underline{c}+C(\varepsilon))\tau\leq  \delta_{\phi}$, where $\delta_{\phi}$ is as \ref{propmeancurvdesing}.\label{iteminiphi}
\end{enumerate}
\end{cor}
\begin{proof}
From the definitions $\underline{\alpha}=\underline{\alpha}_i(\xi)$, $\underline{\phi}=\underline{\tilde{\phi}}_i(\xi)$. (i) follows by \eqref{eqalphaxi}, \ref{propconvconf}\ref{itemconfalpha}, \ref{defdelta}, \ref{defXi} and \eqref{eqtau}. From \eqref{eqalphaxi}, 
\begin{equation*}
    \alpha_i^+(\xi)-\alpha_i^-(\xi)=\alpha_i^+\Confarg-\alpha_i^-\Confarg+(\phi^-_i-\phi^+_i)=(e^{\sigma_i}-1)\alpha_i^+ \Confarg+(\phi^-_i-\phi^+_i).
\end{equation*}
Moreover, from \ref{corbal}, $\abs{\alpha_i^{\pm}\Confarg}<1$ when $k$ big enough. (ii) then follows by \ref{defXi} and \eqref{eqtau}. (iii) follows by \ref{defXi}, \ref{proppreini}\ref{itempreiniphi} and by assuming $\tau$ small enough. 
\end{proof}

\subsection*{The boundary of the pre-initial surfaces}
Unlike in \cite{kapouleas:finite} there are no complete ends extending to infinity. 
On the other hand, we need to ensure that the initial surfaces intersect $\Sph^2$ orthogonally. 
To achieve this we study and modify the vicinity of the boundary of the pre-initial surfaces as in \cite{kapouleas:wiygul}. 

\begin{notation}
$\mathring{X}:=\mathring{X}[\xi]:\mathring{M}=\mathring{M}[\xi]\to E^3$ is defined to be the embedding map for $\mathring{M}$ and $\mathring{\nu}:=\mathring{\nu}[\xi]:\mathring{M}=\mathring{M}[\xi]\to E^3$ is defined to be the unit normal to $\mathring{M}$ which points outward on each annulus $\mathcal{A}''_i$. $\mathring{\Theta}:\partial \mathring{M}\to \mathbb{R}$ is defined to be the inner product
\begin{equation*}
    \mathring{\Theta}:=\langle \mathring{X},\mathring{\nu} \rangle.
\end{equation*}
\end{notation}
\begin{definition}\label{defprecoor}
On a neighborhood in $\mathring{\mathcal{S}}_i\subset\mathring{M}$ of each component of $\partial\mathring{M}$ small enough, such that within it the map of the nearest point projection onto $\partial\mathring{M}$ is well-defined and smooth, the smooth coordinates $(\mathring{\rho},\mathring{t})$ are defined as the following: for any point $P$ in this region, let $\mathring{\rho}$ be the distance of $P$ from $\partial\mathring{M}$, and let $\mathring{t}$ be the distance in $\partial\mathring{M}$ of the nearest point projection of $P$ onto $\partial\mathring{M}$ from an arbitrary fixed reference point on $\partial\mathring{M}$. In particular, along $\partial\mathring{M}$, $\partial_{\mathring{\rho}}$ is the inward unit conormal for $\partial\mathring{M}$, and $\{ \mathring{X}_*\partial_{\mathring{\rho}},\mathring{X}_*\partial_{\mathring{t}}, \mathring{\nu}\}$ is an orthonormal basis at each point of $\mathring{M}$.
\end{definition}  
From the definitions, $\left.\langle \mathring{X},\mathring{X}_*\partial_{\mathring{\rho}} \rangle\right\vert_{\partial\mathring{M}}=-\sqrt{1-\mathring{\Theta}^2} $.

\begin{lemma}\label{lemPhi}
For each $\mathring{\mathcal{S}}_i$, $i=1,\dots,k$, $\preiniDsingsur=\preiniDsingsur\inDsingarg=\mathcal{H}_i^{-1}(\mathring{\mathcal{S}}_i)$, the following results hold:\\
(i)There exists $\epsilon'=\epsilon'(k)$, such that $\mathring{\rho}$ is well defined and smooth in the range $[0,\epsilon'\tau]$. Moreover, $\epsilon'$ could be chosen to be small enough, such that $\mathring{M}_{\mathring{\rho}\leq \epsilon'\tau}\subset \mathring{M}_{s=0}$\\ 
(ii)$\norm{\mathring{\Theta}\circ \mathcal{H}_i:C^3(\partial_n\preiniDsingsur)}\leq C\tau$.
\end{lemma}
\begin{proof}
(i) follows by the construction of $\mathring{M}$ and a scaling. (ii) follows by the construction in \ref{defZy}\ref{Zysym}, \ref{defdeform}\ref{itemdeformtau} and the smooth dependence of parameters in \ref{propZ}.
\end{proof}

\subsection*{The Construction of the Initial Surfaces}

\begin{definition}[The initial surfaces]  
\label{defini} 
$\mathring{u}\in C^{\infty}(\mathring{M})$ is defined by 
\begin{equation*}
    \mathring{u}(\mathring{\rho},\mathring{t}):=-\mathring{\rho}\frac{\mathring{\Theta}(\mathring{t})}{\sqrt{1-\mathring{\Theta}^2(\mathring{t})}}\psi\left[\frac{\epsilon'\tau}{2},\frac{\epsilon'\tau}{4}\right](\mathring{\rho})
\end{equation*}
on $\mathring{\mathcal{S}}_i$ and $0$ on the rest of $\mathring{M}$.

Then the corresponding deformation $\mathring{X}_{\mathring{u}}[\xi]:\mathring{M}[\xi]\to E^3$ is defined by  
\begin{equation*}
    \mathring{X}_{\mathring{u}}[\xi](P):=\mathring{X}(P)+\mathring{u}(P)\mathring{\nu}(P),
\end{equation*}
and the resulting initial surface $M=M[\xi]$ is defined to be the image of $\mathring{M}$ under $\mathring{X}_{\mathring{u}}$:
\begin{equation*}
    M[\xi]:=\mathring{X}_{\mathring{u}}[\xi](\mathring{M}[\xi]).
\end{equation*}
Similarly, $\mathcal{S}_i[\xi]$, $S[\xi]$ are defined by
\begin{align*}
     \mathcal{S}_i[\xi]:=\mathring{X}_{\mathring{u}}[\xi](\mathring{\mathcal{S}}_i[\xi]),\quad S[\xi]:=\mathring{X}_{\mathring{u}}[\xi](\mathring{S}[\xi])=\cup_{i=1}^k\mathcal{S}_i[\xi].
    \end{align*}
\end{definition}

\begin{notation}
The embedding map for $M$ will be denoted by $X:=X[\xi]:M=M[\xi]\to E^3$. 
The unit normal to $M$ which points outward on each catenoidal annulus $\mathcal{A}''_i$ will be denoted by $\nu:=\nu[\xi]:M=M[\xi]\to E^3$. 
Moreover, the composition $\mathring{X}_{\mathring{u}}\circ \mathcal{H}_i\circ \wrap\inDsingarg:\Scherkin\to M$ will be denoted by $\mathcal{Y}_i[\xi]$.
\end{notation}

\begin{definition}
\label{D:rho} 
The coordinates $(\rho,t)$ are defined on a sufficiently small neighborhood in $M$ on each components of $\partial M$, 
so that for each $P$ in this neighborhood $\rho(P)$ is the distance in $M$ from $\partial M$ to $P$, 
while $t(P)$ is the distance along $\partial M$ from the same fixed reference point in $\partial M$ as in the definition of $\mathring{t}$ (see \ref{defprecoor}) 
to the nearest point projection of $P$ onto $\partial M$. Moreover, the function $s$ on $M$ is defined to be the push forward of $s$ on $\mathring{M}$.
\end{definition}

\begin{prop}[Properties of the initial surfaces]  
\label{propini}
$M=M[\xi]$ with $\xi\in \Xi_{\mathcal{V}}$ is well defined and has the following properties:
\begin{enumerate}[label={(\roman*)},ref={(\roman*)}]
    \item $M$ is a smooth embedded surface which depends smoothly on $\xi$. Moreover,
    \begin{equation*}
        M[\xi]=\cup_{i=1}^k\mathcal{S}_i[\xi]\cup_{j=0}^k\mathcal{A}''_j[\xi].
    \end{equation*}
    \item $M$ is contained in the unit ball $\mathbb{B}^3$ and $\partial M$ is contained in the unit sphere $\mathbb{S}^2$, which is a disjoint union of $km$ topologically circles.
    \item \label{iteminibd} $M$ intersects $\mathbb{S}^2$ orthogonally along $\partial M$.
    \item $M$ respects the symmetries in $\mathcal{G}$.\label{iteminisym}
    \item $M_{s\geq1}$ is a graph over $M'_{s\geq1}$ by a function $f_M$ such that $\norm{f_{M}:C^5(M'_{s\geq \underline{a}})}\leq \tau^3$.
\end{enumerate}
\end{prop}
\begin{proof}
(i), (ii), (iv), (v) come from the construction of $M$, \ref{proppreini} and \ref{lemPhi}(i). Moreover, from \ref{defprecoor} and \ref{defini}, $\mathring{u}|_{\partial \mathring{M}}=0$, $\partial\mathring{u}/\partial\mathring{t}|_{\partial \mathring{M}}=0$, $\partial\mathring{u}/\partial\mathring{\rho}|_{\partial \mathring{M}}=-\frac{\mathring{\Theta}}{\sqrt{1-\mathring{\Theta}^2}}$, and thus 
\begin{equation*}
    \left.\nu\circ\mathring{X}_{\mathring{u}}\right\vert_{\partial M}=\left.\frac{\mathring{\nu}-\partial\mathring{u}/\partial\mathring{\rho} \mathring{X}_*\partial_{\mathring{\rho}}}{\sqrt{1+(\partial\mathring{u}/\partial\mathring{\rho})^2}}\right\vert_{\partial \mathring{M}}.
\end{equation*}
Therefore,
\begin{equation*}
    \left.\langle X,\nu \rangle\right\vert_{\partial M}=\left.\langle \mathring{X},\nu\circ\mathring{X}_{\mathring{u}} \rangle\right\vert_{\partial \mathring{M}}=0,
\end{equation*}
as $\left.\langle \mathring{X},\mathring{\nu} \rangle\right\vert_{\partial M}=\mathring{\Theta}$, $\left.\langle \mathring{X},\mathring{X}_*\partial_{\mathring{\rho}} \rangle\right\vert_{\partial\mathring{M}}=-\sqrt{1-\mathring{\Theta}^2}$. This shows (iii). 
\end{proof}

\begin{lemma}\label{lemXu}
Let $\beta=\beta_i(\xi)$, $\mathcal{S}=\mathcal{S}_i[\xi]$ and $\mathring{\mathcal{S}}=\mathring{\mathcal{S}}_i[\xi]$, then
\begin{enumerate}[label={(\roman*)},ref={(\roman*)}]
    \item $\frac{1}{\tau^2}\norm{g\circ \mathring{X}_{\mathring{u}} -\mathring{g}:C^3(\mathring{\mathcal{S}},(\tau\sin{\beta})^{-2}\mathring{g})}\leq C\tau$, where $g$ and $\mathring{g}$ are the first fundamental forms on $\mathcal{S}$ and $\mathring{\mathcal{S}}$ respectively.\label{itemXug}
    \item $\frac{1}{\tau}\norm{A\circ \mathring{X}_{\mathring{u}} -\mathring{A}:C^3(\mathring{\mathcal{S}},(\tau\sin{\beta})^{-2}\mathring{g})}\leq C\tau$, where $A$ and $\mathring{A}$ are the second fundamental forms on $\mathcal{S}$ and $\mathring{\mathcal{S}}$ respectively.\label{itemXuA}
    \item $\tau\norm{H\circ \mathring{X}_{\mathring{u}} -\mathring{H}:C^3(\mathring{\mathcal{S}},(\tau\sin{\beta})^{-2}\mathring{g})}\leq C\tau$, where $H$ and $\mathring{H}$ are the mean curvatures on $\mathcal{S}$ and $\mathring{\mathcal{S}}$ respectively.\label{itemXuH}
\end{enumerate}
\end{lemma}
\begin{proof}
These follow by the construction of $M$ and \ref{lemPhi}(ii). 
\end{proof}

\subsection*{Graphs over the Initial Surfaces}
To study the perturbations of the initial surfaces in a convenient manner (including the free boundary condition) we will describe them as graphs in an auxiliary metric 
we define below following \cite{kapouleas:wiygul}. 

\begin{definition} 
\label{defga}
We define an auxiliary metric 
$g_A:=\Omega^2g$ on $\R^3$ where $\Omega:\R^3\to\R_+$ is a fixed smooth function such that $\Omega=r :=\sqrt{x^2+y^2+z^2}$ when $r>1/2$, 
so that the region $\{r\in(\frac{2}{3},1]\} \subset\B^3$ with the $g_A$ metric is isometric to the round cylinder $\mathbb{S}^2\times (-\log{{2}},0]$.
\end{definition}

Therefore, the $g_A$ unit normal on $M$ pointing in the same direction as $\nu$ is $(\Omega\circ X)^{-1}\nu$.

\begin{definition}
Given any function $\tilde{u}\in C^2(M)$, the perturbation $X_{\tilde{u}}$ is defined by
\begin{equation*}
    X_{\tilde{u}}(P):=\exp^{g_A}_{X(P)}\frac{\tilde{u}(P)\nu(P)}{\Omega\circ X(P)},
\end{equation*}
where $\exp^{g_A}$ is the exponential map for $g_A$ on $E^3$. For $\tilde{u}$ small enough, $X_{\tilde{u}}$ is an immersion with the Euclidean unit normal $\nu_{\tilde{u}}$, which is taken to have positive inner product with the velocity of $g_A$ geodesic generated by $\nu$. Also $H[\tilde{u}]:M\to \mathbb{R}$ is defined to be the mean curvature of the image of $X_{\tilde{u}}$ relative to $g$ and $\nu_{\tilde{u}}$. $\Theta[\tilde{u}]:\partial M\to \mathbb{R}$ is defined by
\begin{equation*}
    \Theta[\tilde{u}]:=(g\circ X_{\tilde{u}})(X_{\tilde{u}},\nu_{\tilde{u}}).
\end{equation*}
Moreover, the graph $M_{\tilde{u}}=M_{\tilde{u}}(\xi)$ is defined to be the image of $X_{\tilde{u}}$.
\end{definition}

\begin{lemma}
For each $\tilde{u}\in C^2(M)$ sufficiently small, $\Theta[{\tilde{u}}]=0$ if and only if $\partial_{\rho}\tilde{u}=0$.
\end{lemma}

\begin{proof}
See \cite[Lemma 5.6]{kapouleas:wiygul}.  
\end{proof}

\begin{notation}
The linearization of $H[\tilde{u}]$ at $\tilde{u}=0$ will be denoted by $\tilde{\mathcal{L}}$.
\end{notation}

\begin{lemma}\label{lemauu}
For each $\tilde{u}\in C^2(M)$, $u\in C^2(M)$ is defined by
\begin{equation*}
    u:=(\Omega\circ X)^{-1}\tilde{u},
\end{equation*}
then
\begin{enumerate}[label={(\roman*)},ref={(\roman*)}]
\item $\tilde{\mathcal{L}}\tilde{u}=\mathcal{L}_M u$.
\item $\partial_{\rho}\tilde{u}|_{\partial M}=(\partial_{\rho}+1)u|_{\partial M}$. In particular, $\Theta[\tilde{u}]=0$ if and only if $(\partial_{\rho}+1)u|_{\partial M}=0$.
\end{enumerate}
\end{lemma}

\begin{proof}
See \cite[Lemma 5.19]{kapouleas:wiygul}.  
\end{proof}

\section{The Linearized Equation}

In this section we study the linearized equation on the initial surfaces. 
Our approach is closer to the approach in \cite{kapouleas:li,kapouleas:equator} for example, rather than in \cite{kapouleas:finite}, 
in that we first solve the equation on the standard models (half Scherk surfaces and catenoidal annuli or discs), 
and then we transfer the solutions to the initial surfaces, 
where we combine them to get approximate solutions which we correct by iteration. 
This shifts some of the difficulty from understanding properties of the linearized equation on the initial surfaces, 
to comparing the equations and the norms on the initial surfaces versus the standard models. 
Besides the differences in the approach there are two more differences with \cite{kapouleas:finite}:  
first, we do not have to deal with ends extending to infinity, 
and second, unlike in \cite{kapouleas:finite}, we have to deal with a boundary and the free boundary condition.    

\subsection*{The linearized equation on $\Scherkin$}\quad\\
In this subsection, an $\mathcal{S}_i[\xi]$ is fixed, and so is the corresponding Scherck surface $\iniDsingsur:=\Scherkin$. On $\iniDsingsur$, $G$ is defined to be the group of symmetries which is generated by the reflections with respect to $xy$-plane and the plane $\{z=\pi\}$ as in \ref{propscherkef}. In the rest of this subsection, all the functions on $\iniDsingsur$ are assumed to be invariant under the action of $G$. Constants $\gamma\in(\frac{3}{4},1)$ and $\alpha\in (0,1)$ are fixed. Moreover, the Neumann boundary $\partial_n \Scherkin$ of half-$\Scherkin$ is defined to be
$\Scherkin_{x=0}$. 

\begin{prop}[The Jacobi equation on Scherk surfaces] 
\label{propLES}
Let $\iniDsingsur=\Scherkin$, there exists a linear map 
\begin{multline*}
    \mathcal{R}_{\iniDsingsur}=\mathcal{R}_{\iniDsingsur,a}[i,\xi]:
    C_{G}^{0,\alpha}(\Scherkin_{x\leq 0},e^{-\gamma s})\times C_{G}^{1,\alpha}(\partial_n\Scherkin) 
\\
    \to C_{G}^{2,\alpha}(\Scherkin_{x\leq 0},e^{-\gamma s})\times \mathbb{R}\times \mathbb{R}^2,
\end{multline*}
and a constant $C$, such that for $(E,f)\in C^{0,\alpha}_{G}(\iniDsingsur_{x\leq 0},e^{-\gamma s})\times C^{1,\alpha}_{G}(\partial_n\iniDsingsur)$, if $(u_E,\theta_E,\underline{\phi}_E):=\mathcal{R}_{\iniDsingsur}(E,f)$, where $\underline{\phi}_E:=\{\phi^-_{E},\phi^+_{E}\}\in \mathbb{R}^2$, then
\begin{enumerate}[label={(\roman*)},ref={(\roman*)}]
    \item $\mathcal{L}_{\iniDsingsur}u_E=E-\theta_E w-\phi^-_{E}\barw_--\phi^+_{E}\barw_+$ on $\iniDsingsur_{x\leq 0}$.
    \item $\partial_x u_E=-f$ on $\partial_n\iniDsingsur$.
    \item $\abs{\theta_E}\leq C\norm{E}$, where $\norm{E}:=\norm{E:C^{0,\alpha}(\iniDsingsur_{x\leq 0},e^{-\gamma s})}$.
    \item $\abs{\underline{\phi}_E}\leq \frac{C}{a}(\norm{E}+\norm{f})$, where $\norm{f}:=\norm{f:C^{1,\alpha}(\partial_n\iniDsingsur)}$.
    \item $\norm{u_E:C^{2,\alpha}(\iniDsingsur_{x\leq 0},e^{-\gamma s})}\leq C(\norm{E}+\norm{f})$.
\end{enumerate}
\end{prop}

\begin{proof}
Fix $\varepsilon_0$ in \ref{lemcyleq}, \ref{lemcylequnique} and \ref{lemextker}, and correspondingly $a_0$, $s_{a_0}$. First, assume that $E$ has support in $\iniDsingsur_{s_{a_0}\leq 2}$. Then there is the estimate:
\begin{equation*}
    \norm{E/\abs{A}^2_{\Sigma}:L^2(\Sigma/G,h)}\leq C(\varepsilon_0)\norm{E},
\end{equation*}
with $h$ as in \ref{corsubker}. Now \ref{corsubker} could be applied, and there is a $\theta_E$ as in \ref{corsubker} with (iii) satisfied. Moreover, by \ref{propscherkef}, the inhomogeneous term in the Neumann boundary value problem on $\Sigma_{x\leq 0}$:
\begin{align}
    &\mathcal{L}_h u'_E=(E-\theta_E w)/\abs{A}^2_{\Sigma},\label{eqhE}\\
    &\left.\partial_x u'_E\right\vert_{\partial_n\iniDsingsur}=-f,\nonumber
\end{align}
is $L^2(\Sigma/G,h)$-orthogonal to the eigenfunctions with eigenvalues smaller than $\epsilon_{\lambda}$, and is bounded above by $\norm{E/\abs{A}^2_{\Sigma}:L^2(\Sigma/G,h)}\leq C(\varepsilon_0)\norm{E}$ from \ref{lemsubkernel}\ref{itemsubkernalsupp} and (iii). Then by the implied $L^2$ estimate and the standard linear theory in $h$ metric, there is a unique solution $u'_E$ such that
\begin{equation}\label{equprimeh}
	\norm{u'_E:C^{2,\alpha}(\Sigma_{x\leq 0},h)}\leq C(\norm{E}+\norm{f}),
\end{equation}
where $C=C(\varepsilon_0)$. From the fact that \eqref{eqhE} is equivalent to 
\begin{equation*}
    \mathcal{L}_{\Sigma} u'_E=E-\theta_E w,
\end{equation*}
the standard linear theory in $h$ metric and the estimate \eqref{equprimeh}, the solution $u'_E$ also satisfies
\begin{equation}\label{equprimeg}
    \norm{u'_E:C^{2,\alpha}(\Sigma_{x\leq 0},g)}\leq C(\norm{E}+\norm{f}),
\end{equation}
where $C=C(\varepsilon_0)$.

Now for any $\varepsilon<\varepsilon_0$ and corresponding $a>a_0$, define $s:=s_a\leq s_{a_0}$. Let $F_{-}:=\refl_2\circ F_{\theta}\circ R$, and $F_{+}:=\refl_3\circ F_{\theta}\circ R$, where $\theta=\alpha_i^+(\xi)$ and $R$ is the reparametrization $(s,z)\to (s+2,z)$. For the cylinder $\Omega:=[0,\infty)\times\mathbb{R}/G'$, where $G'$ is the group generated by $(s,z)\to(s,z+2\pi)$ as in \ref{lemcylequnique}, by the definition, \ref{lemsubkernel}\ref{itemsubkernalsupp} and the assumption of the support of $E$, $\underline{\mathcal{L}}( u'_E\circ F_{\pm})=0$ on $\Omega$, where $\underline{\mathcal{L}}$ is as in \ref{lemcylequnique}. Then by \ref{lemcylequnique}, \eqref{equprimeh} for $j=\pm$, $a_{j}$, could be defined by
\begin{equation*}
    a_j:=\lim_{s\to \infty} u'_E\circ F_{j}.
\end{equation*} 
Recall \ref{nota:apm}
$u_E$ is then defined by
\begin{equation*}
    u_E=u'_E+\phi^-_{E}\baru_-+\phi^+_{E}\baru_+,
\end{equation*}
where $\phi_{E}^{j}$, $j=\pm$, is determined by the linear equations:
\begin{equation}\label{eqaij}
    a_j+\sum_{i=\pm}\phi_{E}^{i}a_{ij}=0.
\end{equation}
Then by \ref{lemextker}\ref{itemextkerucomp} along with \eqref{equprimeh} and the definition, $\phi^{\pm}_E$ is well defined and satisfies (iv). Moreover, by \ref{lemextker}\ref{itemextkerucomp}, the function $u_E$ is in $C^{2,\alpha}(\Sigma_{x\leq 0},h)$ with $\norm{u_E:C^{2,\alpha}(\Sigma_{x\leq 0},h)}<\infty$ and by \ref{lemextker}\ref{itemextkersupp}, \ref{lemextker}\ref{itemextkerucomp} and \eqref{eqaij} satisfies 
\begin{align*}
    \underline{\mathcal{L}} (u_E\circ F_j)=0,\quad
   \lim_{s\to \infty}u_E\circ F_j=0.
\end{align*}
Therefore, by (iv), \ref{lemextker}\ref{itemextkerubdd}, \eqref{equprimeg} the uniqueness and the estimate in \ref{lemcylequnique}, $u_E$ satisfies (v).

Now for general $E$, let the cylinder $\Omega$ as before, $F_-:=F_{-,a_0}:=\refl_2\circ F_{\theta,a_0}\circ R$, and $F_+:=F_{+,a_0}:=\refl_3\circ F_{\theta,a_0}\circ R$ but $R$ is the reparametrization $(s,z)\to (s+1,z)$. Then by \ref{lemcyleq}, there is a function $v'_{E\pm}$ such that $\underline{\mathcal{L}}v'_{E\pm}=E\circ F_{\pm}$. Moreover, 
\begin{equation}\label{eqEprime}
    \norm{v'_{E\pm}:C^{2,\alpha}(\Omega,g_0,e^{-\gamma s})}\leq C \norm{E\circ F_{\pm}:C^{0,\alpha}(\Omega,g_0,e^{-\gamma s})}.
\end{equation} 
Then $E'$ can be defined by 
\begin{equation*}
    E':=E-\mathcal{L}_{\iniDsingsur}(\psi[1,2](s_{a_0}) F_{\pm*}v'_{E\pm}).
\end{equation*}
Then $E'$ is supported on $\iniDsingsur_{s_{a_0}\leq 2}$, and the results above could be applied for $E'$ to get $(u_{E'},\theta_{E'},\underline{\phi}_{E'})=\mathcal{R}_{\Sigma}(E',f)$. Recall \ref{def:wing}, we can now define $u_E=u_{E'}$ on the core of $\Sigma_{x\leq 0}$ and $u_E=u_{E'}+(\psi[1,2](s_{a_0})) F_{\pm*} v'_{E\pm}$ on the $\pm$-wing of $\Sigma_{x\leq 0}$. And then $\mathcal{R}_{\Sigma}(E,f)$ is defined by
\begin{equation*}
    \mathcal{R}_{\Sigma}(E,f):=(u_{E},\theta_{E'},\underline{\phi}_{E'}).
\end{equation*}
The estimates follow by the definitions, \eqref{eqEprime} and \ref{propscherk}\ref{itemscherkdecay}.
\end{proof}

\subsection*{The linearized equation on $\mathcal{N}_j[\xi]$}\quad\\

In this subsection, an $\mathcal{N}=\mathcal{N}_j[\xi]$ is fixed. 

\begin{definition}[Norms on catenoidal annuli] 
\label{defLEN} 
	Given $u\in C^{r,\alpha}(\mathcal{N})$ $(r=0,2)$ define $\norm{u:\mathcal{N}}_r$ by
\begin{equation*}
    \norm{u:\mathcal{N}}_r:=\norm{u:C^{r,\alpha}(\mathcal{N})}.
\end{equation*}
Similarly define $\norm{u:\mathcal{N}}'_r$ by
\begin{equation*}
    \norm{u:\mathcal{N}}'_r:=\norm{u:C^{r,\alpha}(\mathcal{N}\setminus S[\xi])}.
\end{equation*}

\end{definition}

\begin{lemma}[The Jacobi equation on catenoidal annuli] 
\label{lemLEN} 
Let $\alpha\in (0,1)$, there exists a linear map 
\begin{align*}
    \mathcal{R}_{\mathcal{N}}=\mathcal{R}_{\mathcal{N},a}[j,\xi]:C_{\Grp}^{0,\alpha}(\mathcal{N}_j[\xi])\to C_{\Grp}^{2,\alpha}(\mathcal{N}_j[\xi]),
\end{align*}
and a constant $C$, such that for $E\in C_{\Grp}^{0,\alpha}(\mathcal{N}_j[\xi])$, 
if $u_E:=\mathcal{R}_{\mathcal{N}}E$, then $\mathcal{L}_{\mathcal{N}}u_E=E$ on $\mathcal{N}$, $u_E=0$ on $\partial\mathcal{N}$, and
\begin{equation*}
    \norm{u_E:\mathcal{N}}_2\leq C\norm{E:\mathcal{N}}_0.
\end{equation*}
\end{lemma}
\begin{proof}
By \ref{propconvconf}\ref{itemconfker} and the fact that $\mathcal{N}$ is a small perturbation of $\mathcal{A}_j$, 
the Dirichlet problem for $\mathcal{L}_{\mathcal{N}}$ has no small eigenvalues. 
Thus there is a unique solution $u_E$ to $\mathcal{L}_{\mathcal{N}}u_E=E$, $u_E=0$ on $\partial \mathcal{N}$, 
and it satisfied the required estimate by the standard linear theory and the uniform control of the geometry.
\end{proof}

\subsection*{The linearized equation globally on the initial surfaces}

\begin{lemma} 
\label{lemopcompare}
Let $\beta=\beta_i(\xi)$, $\mathcal{S}:=\mathcal{S}_i[\xi]$, $\mathcal{Y}=\mathcal{Y}_i[\xi]$, $\Sigma:=\Scherkin$, 
and let $\mathcal{Y}_{*}$ be the push forward by $\mathcal{Y}$ of a function or tensor on $\Scherkin_{x\leq 0,s\leq 5\delta_s/\tau}$ 
to a function or tensor on $\mathcal{S}_i[\xi]$, the following results hold:
\begin{enumerate}[label={(\roman*)},ref={(\roman*)}]
\item $\norm{(\tau\sin{\beta})^{-2}g_{\mathcal{S}})-\mathcal{Y}_*g_{\Sigma}:C^3(\mathcal{S}, (\tau\sin{\beta})^{-2}g_{\mathcal{S}})}\leq C(\varepsilon)(1+\underline{c})\tau+C(\delta_s+\varepsilon)$.\label{itemopcompareg}
\item Suppose $u\in C^{2,\alpha}_G(\Sigma_{x\leq 0,s\leq 5\delta_s/\tau})$, then
    \begin{align*}
        &\norm{(\tau\sin{\beta})^2\mathcal{L}_M\circ\mathcal{Y}_{*}u-\mathcal{Y}_*\mathcal{L}_{\iniDsingsur}u:C^{0,\alpha}(\mathcal{S}, (\tau\sin{\beta})^{-2}g_{\mathcal{S}},e^{-\gamma s})}\\
        \leq&\left(C(\varepsilon)(1+\underline{c})\tau+C(\delta_s+\varepsilon)\right)\norm{u:C^{2,\alpha}(\mathcal{S}, (\tau\sin{\beta})^{-2}g_{\mathcal{S}},e^{-\gamma s})}.
    \end{align*}\label{itemopcompareL}
\item \label{itemopcomparebd} 
Suppose $u\in C_G^{1,\alpha}(\partial_n\iniDsingsur)$, then
    \begin{align*}
        &\norm{\tau\sin{\beta}\partial_{\rho}\circ \mathcal{Y}_{i*}u+\mathcal{Y}_*\partial_{x}u:C^{0,\alpha}(\partial M\cap\mathcal{S}, (\tau\sin{\beta})^{-2}g_{\mathcal{S}})}\\
        \leq& C\tau\norm{u:C^{1,\alpha}(\partial M\cap\mathcal{S}, (\tau\sin{\beta})^{-2}g_{\mathcal{S}})}.
    \end{align*} 
\end{enumerate}
\end{lemma}

\begin{proof}
(i) follows by \ref{lemZ}\ref{itemZg},  \ref{lemXu}\ref{itemXug}, \ref{defXi} and a scaling. 
(ii) follows by \ref{lemZ}\ref{itemZA}, \ref{lemXu}\ref{itemXuA} along with (i). 
(iii) follows by the smooth dependence in \ref{propZ}, \ref{lemXu}\ref{itemXug} and the construction. 
\end{proof}

\begin{definition}[Global norms on the initial surfaces {${M=M[\xi]}$}] 
\label{D:norm}
Given $u\in C^{r,\alpha}(M)$ $(r=0,2)$, $\norm{u}_r$ is defined to be the maximum of the following quantities, where $b_0=e^{-5\gamma\delta_s/\tau}$ and $b_2=b_0/\tau^{10}$:
\begin{enumerate}[label={(\roman*)},ref={(\roman*)}]
    \item For each $i=1,\dots,k$, and $\mathcal{S}_i:=\mathcal{S}_i[\xi]$,
    \begin{equation*}
        \tau^{1-r}\norm{u:C^{r,\alpha}(\mathcal{S}_i,(\tau\sin{\beta_i})^{-2}g_{\mathcal{S}_i},\max\{e^{-\gamma s},b_r\})}.
    \end{equation*}
    \item For each $j=1,\dots,k+1$, 
    \begin{equation*}
        b_r^{-1}\norm{u:\mathcal{N}_j[\xi]}'_r.
    \end{equation*}
\end{enumerate}
And given $f\in C^{1,\alpha}(\partial M)$, $\norm{f}_1$ is defined to be the maximum of the following quantities: for each $i=1,\dots,k$, consider the quantity
\begin{equation*}
    \norm{f:C^{1,\alpha}(\partial M\cap \mathcal{S}_i,(\tau\sin{\beta_i})^{-2}g_{\mathcal{S}_i})}.
\end{equation*}
\end{definition}

\begin{definition}[$W$ and the extended substitute kernel on the initial surfaces] 
\label{D:skernel} 
We define the \emph{extended substitute kernel} 
$\skernel=\skernel[\,M[\xi]\,] := W(\mathcal{V}') \subset C^{\infty}_{\Grp}(M) $ (recall \ref{N:G}),  
where 
$W:\mathcal{V}'\to C^{\infty}_{\Grp}(M)$ is the linear map 
(actually a linear isomorphism from $\mathcal{V}'$ onto $\skernel$) 
defined by 
(recall \ref{defXi}) 
\begin{equation*}
    W(\underline{\theta}',\underline{\varphi}')=\sum_{i=1}^k\frac{1}{\tau\sin\beta_i}\mathcal{Y}_{i*}(\theta'_iw+\phi'^-_i\barw_-+\phi'^+_i\barw_+),
\end{equation*}
where $\underline{\theta}'=\{\theta'_i\}_{i=1}^k\in \mathcal{V}_{\theta}$, $\underline{\varphi}'=\{\underline{\phi}'_i=\{\phi'^{-}_i,\phi'^{+}_i \} \}_{i=1}^k\in \mathcal{V}_{\phi}$, 
and $\mathcal{Y}_{i*}$ is the push forward by $\mathcal{Y}_{i}$ of a function on $\Scherkin$ to a function on $\mathcal{S}_i[\xi]$ extended to vanish on the rest of $M$.
\end{definition}

\begin{definition}
$\psi'\in C^{\infty}_{\Grp}(M)$ is defined by
\begin{equation*}
    \psi':=\psi[5\delta_s\tau^{-1},5\delta_s\tau^{-1}-1]\circ s
\end{equation*}
on $\mathcal{S}_i$ and $0$ on the rest of $M$.
Also $\psi'\in C^{\infty}_{\Grp}(M)$ is defined by
\begin{equation*}
    \psi'':=\psi[\underline{a},\underline{a}+1]\circ s.
\end{equation*}
\end{definition}

Now the linear map
\begin{align*}
    \mathcal{R}_{M,appr}: 
    &\{E\in C^{0,\alpha}_{\Grp}(M)|\norm{E}_0< \infty\}\times C_{\Grp}^{1,\alpha}(\partial M)\\
    &\to \{u\in C^{2,\alpha}_{\Grp}(M)|\norm{u}_2< \infty\}\times \mathcal{V}' \times \{E_1\in C^{0,\alpha}_{\Grp}(M)|\norm{E_1}_0< \infty\}\times C_{\Grp}^{1,\alpha}(\partial M)
\end{align*}
is constructed by the following. Given $E\in C^{0,\alpha}_{\Grp}(M)$, $f\in C^{1,\alpha}_{\Grp}(\partial M)$ with $\norm{E}_0<\infty$, $\norm{f}_1<\infty$, let $\iniDsingsur:=\Scherkin$, then
\begin{equation*}
    \psi' E|_{\mathcal{S}_i}\circ \mathcal{Y}_i \in  C^{0,\alpha}_{G}(\iniDsingsur,e^{-\gamma s}),\quad f\circ \mathcal{Y}_i\in  C^{1,\alpha}_{G}(\partial_n\iniDsingsur).
\end{equation*}
By \ref{propLES}, there is $(v_{Si},\theta_{1i},\underline{\phi}_{1i})\in C^{2,\alpha}(\iniDsingsur)\times \mathbb{R}\times\mathbb{R}^2$ such that
\begin{equation*}
(v_{Si},\theta_{1i},\underline{\phi}_{1i})=\mathcal{R}_{\iniDsingsur}(\tau\sin{\beta_i}\psi' E|_{\mathcal{S}_i}\circ \mathcal{Y}_i,f\circ\mathcal{Y}_i).    
\end{equation*}
Define $u_{1S}:=\tau \sin{\beta_i} \mathcal{Y}_{i*} v_{Si}$ on $\mathcal{S}_i$ and $0$ on the rest, then $u_{1S}\in C_{\Grp}^{2,\alpha}(S)$ with $\norm{u_{1S}}_2< \infty$. By \ref{lemLEN}, there is $u_{Nj}\in C_{\Grp}^{2,\alpha}(\mathcal{N}_j)$, such that 
\begin{equation*}
    u_{Nj}=\mathcal{R}_{\mathcal{N}}((1-\psi'^2)E-[\mathcal{L}_{M},\psi']u_{1S}).
\end{equation*}
Again, let $u_{1N}= u_{Nj}$ on $\mathcal{N}_j$ and 0 on the rest. 
\begin{definition}
$(u_1,\underline{\theta}_1,\underline{\varphi}_1,E_1,f_1)=\mathcal{R}_{M,appr}(E,f)$ are defined by the following.
\begin{enumerate}[label={(\roman*)},ref={(\roman*)}]
\item $u_1:=\psi' u_{1S}+\psi'' u_{1N}$.
\item $\underline{\theta}_1:=\{\theta_{1i}\}_{i=1}^k$.
\item $\underline{\varphi}_1:=\{\underline{\phi}_{1i}\}_{i=1}^k$.
\item $E_1:=\mathcal{L}_{M}u_1-E+W(\underline{\theta}_1,\underline{\varphi}_1)$.
\item $f_1:=\partial_{\rho} u_1 + u_1-f $
\end{enumerate}
\end{definition}

\begin{prop}[The Jacobi equation on the initial surfaces] 
\label{propLEM} 
For some small $\delta_s$, $\varepsilon$, when $\tau$ small enough, a linear map 
\begin{equation*}
    \mathcal{R}_{M}: 
    \{E\in C_{\Grp}^{0,\alpha}(M)|\norm{E}_0< \infty\} \to \{u\in C^{2,\alpha}_{\Grp}(M)|\norm{u}_2< \infty\}\times \mathcal{V}'
\end{equation*}
can be defined by 
\begin{equation*}
    \mathcal{R}_{M}E:=(u_E,\underline{\theta}_E,\underline{\varphi}_E):=\sum_{n=1}^{\infty}(u_n,\underline{\theta}_n,\underline{\varphi}_n)
\end{equation*}
for $E\in C^{0,\alpha}(M)$ with $\norm{E}_0<\infty$, where the sequence $\{(u_n,\underline{\theta}_n,\underline{\varphi}_n,E_n,f_n) \}_{n\in\mathbb{N}}$ is defined inductively for $n\in\mathbb{N}$ by
\begin{equation*}
    (u_n,\underline{\theta}_n,\underline{\varphi}_n,E_n,f_n):= -\mathcal{R}_{M,appr}(E_{n-1},f_{n-1}),\quad E_0:=-E,\quad f_0=0.
\end{equation*}
Moreover, 
\begin{align*}
    &\mathcal{L}_M u_E=E-W(\underline{\theta}_E,\underline{\varphi}_E),\quad (\partial_{\rho}+1) u_E|_{\partial M}=0,\\
    &\abs{\underline{\theta}_E}\leq C\norm{E}_0,\quad \abs{\underline{\varphi}_E}\leq C\norm{E}_0,\quad \norm{u_E}_2\leq C\norm{E}_0.
\end{align*}
\end{prop}
\begin{proof}
By the definitions, \ref{lemsubkernel}\ref{itemsubkernalsupp}, \ref{lemextker}\ref{itemextkersupp} and the facts $\psi''(1-\psi'^2)=1-\psi'^2$, $\psi''\nabla\psi'=\nabla\psi'$,
\begin{align*}
    &\mathcal{L}_Mu_1=-E_0-W(\underline{\theta}_1,\underline{\varphi}_1)+[\mathcal{L}_M,\psi'']u_{1N}+(\tau\sin\beta_i)^{-1}\psi'\sum_i((\tau\sin\beta_i)^2\mathcal{L}_M\circ \mathcal{Y}_{i*} v_{Si}-\mathcal{Y}_{i*}\mathcal{L}_{\Sigma}v_{Si}),\\
    &\partial_{\rho}u_1+u_1=-f_0+\sum_i\tau\sin\beta_i\mathcal{Y}_{i*}v_{Si}+\sum_i(\tau\sin\beta_i\partial_{\rho}\circ \mathcal{Y}_{i*} v_{Si}+ \mathcal{Y}_{i*}\partial_x v_{Si}),
\end{align*}
where $(u_1,\underline{\theta}_1,\underline{\varphi}_1,E_1,f_1):=-\mathcal{R}_{M,appr}(E_0,f_0)$. Therefore, by the definitions,
\begin{align*}
   E_1&=[\mathcal{L}_M,\psi'']u_{1N}+(\tau\sin\beta_i)^{-1}\psi'\sum_i((\tau\sin\beta_i)^2\mathcal{L}_M\circ \mathcal{Y}_{i*} v_{Si}-\mathcal{Y}_{i*}\mathcal{L}_{\Sigma}v_{Si}),\\
    f_1&=\sum_i\tau\sin\beta_iv_{Si}+\sum_i(\tau\sin\beta_i\partial_{\rho}\circ \mathcal{Y}_{i*} v_{Si}+ \mathcal{Y}_{i*}\partial_x v_{Si}).
\end{align*}

However, as $\supp\nabla\psi'\subset\left[\frac{5\delta_s}{\tau}-1,\frac{5\delta_s}{\tau}\right]$ 
and $\supp\nabla\psi''\subset [\underline{a},\underline{a}+1]=\left[8\log\frac{1}{\tau},8\log\frac{1}{\tau}+1\right]$, by \ref{lemLEN}, the fact that $\gamma\in(3/4,1)$ and the definitions
\begin{align*}
    \norm{[\mathcal{L}_M,\psi'']u_{1N}}_0&\leq\max_i \tau^{-1}\norm{u_{1N}:C^{2,\alpha}(\mathcal{S}_i\cap M_{s\in[\underline{a},\underline{a}+1] },(\tau\sin\beta_i)^{-2}g_{\mathcal{S}_i},e^{-\gamma s})}\\
    &\leq C\tau^{-9} \max_i \norm{u_{1N}:C^{2,\alpha}(\mathcal{S}_i\cap M_{s\in[\underline{a},\underline{a}+1] },(\tau\sin\beta_i)^{-2}g_{\mathcal{S}_i}}\\
    &\leq C\tau^{-9}\norm{u_{1N}:C^{2,\alpha}(M,g_M)}\leq C\tau^{-9}\norm{-(1-\psi'^2)E_0-[\mathcal{L}_{M},\psi']u_{1S}:C^{0,\alpha}(M,g_M)}\\
    &\leq C\tau^{-9}\left(\tau^{-1}\norm{E_0:C^{0,\alpha}(M_{s\geq\frac{5\delta_s}{\tau}-1},g_M)}+\tau^{-3}\norm{u_{1S}:C^{2,\alpha}(M_{s\in \left[\frac{5\delta_s}{\tau}-1,\frac{5\delta_s}{\tau}\right]},g_M)}\right).
\end{align*}
By the definitions,
\begin{align*}
    &\norm{E_0:C^{0,\alpha}(M_{s\geq\frac{5\delta_s}{\tau}-1},g_M)}\\
    \leq &\max_i Ce^{-5\gamma \delta_s/\tau}\norm{E_0:C^{0,\alpha}(\mathcal{S}_i,g_{\mathcal{S}_i},e^{-\gamma s})}+\max_j \norm{E_0:C^{0,\alpha}(\mathcal{N}_j)}'\\
    \leq &\max_i Ce^{-5\gamma \delta_s/\tau}/\tau\norm{E_0:C^{0,\alpha}(\mathcal{S}_i,(\tau\sin\beta_i)^{-2}g_{\mathcal{S}_i},e^{-\gamma s})}+\max_j \norm{E_0:C^{0,\alpha}(\mathcal{N}_j)}'\\
    \leq& Ce^{-5\gamma \delta_s/\tau}/\tau^2 \norm{E_0}_0.
\end{align*}
And by \ref{propLES} and \ref{lemopcompare}\ref{itemopcompareg},
\begin{align*}
    &\norm{u_{1S}:C^{2,\alpha}(M_{s\in \left[\frac{5\delta_s}{\tau}-1,\frac{5\delta_s}{\tau}\right]},g_M)}\leq Ce^{-5\gamma \delta_s/\tau}\norm{u_{1S}:C^{2,\alpha}(S,g_M,e^{-\gamma s})}\\
    \leq& \max_i Ce^{-5\gamma \delta_s/\tau}/\tau^3\norm{u_{1S}:C^{2,\alpha}(\mathcal{S}_i,(\tau\sin\beta_i)^{-2}g_{\mathcal{S}_i},e^{-\gamma s})}\\
    \leq & \max_i Ce^{-5\gamma \delta_s/\tau}/\tau^2\norm{v_{Si}:C^{2,\alpha}(\Sigma,g_{\Sigma},e^{-\gamma s})}\\
    \leq& \max_i Ce^{-5\gamma \delta_s/\tau}/\tau^2 \left( \tau \norm{E_0\circ \mathcal{Y}_i:C^{0,\alpha}(\Sigma_{x\leq 0},g_{\Sigma},e^{-\gamma s})}+\norm{f_0\circ \mathcal{Y}_i:C^{1,\alpha}(\partial_n\iniDsingsur,g_{\Sigma})}\right)\\
    \leq & Ce^{-5\gamma \delta_s/\tau}/\tau^2 \left(  \norm{E_0}_0+\norm{f_0}_1\right).
\end{align*}
Moreover, by \ref{lemopcompare}\ref{itemopcompareL} and the inequality $\norm{v_{Si}:C^{2,\alpha}(\Sigma_{x\leq 0}, g_{\Sigma},e^{-\gamma s})}\leq C\left( \norm{E_0}_0+\norm{f_0}_1\right)$ above
\begin{align*}
    &\norm{\psi'((\tau\sin\beta_i)^2\mathcal{L}_M\circ \mathcal{Y}_{i*} v_{Si}-\mathcal{Y}_{i*}\mathcal{L}_{\Sigma}v_{Si})}_0\\
    =&\tau\norm{\psi'((\tau\sin\beta_i)^2\mathcal{L}_M\circ \mathcal{Y}_{i*} v_{Si}-\mathcal{Y}_{i*}\mathcal{L}_{\Sigma}v_{Si}):C^{0,\alpha}(\mathcal{S}_i,(\tau\sin\beta_i)^{-2}g_{\mathcal{S}_i},e^{-\gamma s})}\\
    \leq& (C(\varepsilon)(1+\underline{c})\tau+C(\delta_s+\varepsilon))\tau\norm{v_{Si}:C^{2,\alpha}(\Sigma_{x\leq 0}, g_{\Sigma},e^{-\gamma s})}\\
    \leq& (C(\varepsilon)(1+\underline{c})\tau+C(\delta_s+\varepsilon))\tau\left( \norm{E_0}_0+\norm{f_0}_1\right).
\end{align*}
Therefore, 
\begin{equation*}
    \norm{E_1}_0\leq \left(Ce^{-5\gamma \delta_s/\tau}/\tau^{14}+(C(\varepsilon)(1+\underline{c})\tau+C(\delta_s+\varepsilon)) \right)\left(  \norm{E_0}_0+\norm{f_0}_1\right).
\end{equation*}
By choosing small $\delta_s$, $\varepsilon$ and then $\tau$ small enough, $\norm{E_1}_0\leq\frac{1}{2}\left( \norm{E_0}_0+\norm{f_0}_1\right)$.

Similarly, by \ref{lemopcompare}\ref{itemopcomparebd},
\begin{align*}
    \norm{f_1}_1\leq C\tau\norm{v_{Si}:C^{1,\alpha}(\partial_n\iniDsingsur),g_{\Sigma})}\leq C\tau\left( \norm{E_0}_0+\norm{f_0}_1\right).
\end{align*}
By choosing $\tau$ small enough, $\norm{f_1}_1\leq\frac{1}{2}\left( \norm{E}_0+\norm{f_0}_1\right)$.

Finally, from the inequalities above and the definitions,
\begin{align*}
    \norm{u_1}_2\leq C\left( \norm{E_0}_0+\norm{f_0}_1\right),\quad \abs{\underline{\theta}_1}\leq C\left( \norm{E_0}_0+\norm{f_0}_1\right),\quad \abs{\underline{\varphi}_1}\leq C\left( \norm{E_0}_0+\norm{f_0}_1\right).
\end{align*}
By using the estimates and iterating we conclude that the operator $\mathcal{R}_M$ is well defined and satisfies the requirements.
\end{proof}

From now on the small constants $\delta_s$, $\varepsilon$ are fixed as in the proposition.

\begin{cor}[The linearized equation for the mean curvature] 
\label{corLEmeancurv} 
There are $u_H\in C^{2,\alpha}_{\Grp}(M)$ and $(\underline{\theta}_H,\underline{\varphi}_H)\in \mathcal{V}'$ such that 
\begin{align*}
    &\mathcal{L}_M u_H=H-W(\underline{\theta}_H,\underline{\varphi}_H),\quad (\partial_{\rho}+1) u_H|_{\partial M}=0,\\
    &\abs{\underline{\theta}_H-\underline{\theta}}\leq C\tau,\quad \abs{\underline{\varphi}_H+\underline{\varphi}}\leq C\tau,\quad \abs{\underline{\sigma}_H-\underline{\sigma}}\leq C\tau,\quad \norm{u_H}_2\leq C\tau,
\end{align*}
where $H$ is the mean curvature of $M$, $M:=M[\xi]=M(\underline{\sigma},\underline{\varphi})$, $\underline{\theta}:=\{ \alpha^+_i\Confarg-\alpha^-_i\Confarg \}_{i=1}^k$, $\underline{\sigma}_H:=\left\{\frac{\theta_{Hi}}{\alpha_i^+\Confarg} \right\}_{i=1}^k\in \mathcal{V}_{\sigma}$.
\end{cor}
\begin{proof}
By \ref{propmeancurvdesing}, \ref{corini} and \ref{lemXu}\ref{itemXuH},
\begin{equation}\label{eqmeancurvM}
    \norm{H-W(\underline{\theta}',\underline{\tilde{\varphi}})}_0\leq C\tau,
\end{equation}
where $\underline{\theta}'=\{\theta'_i \}_{i=1}^k$, with
\begin{align*}
    \theta'_i&=\alpha_i^+(\xi)-\alpha_i^-(\xi)-\tilde{\phi}^+_i(\xi)+\tilde{\phi}^-_i(\xi)\\
    &=\alpha_i^+\Confarg-\alpha_i^-\Confarg-\phi^+_i+\phi^-_i-\tilde{\phi}^+_i(\xi)+\tilde{\phi}^-_i(\xi).
\end{align*}
Applying the proposition with $E:=H-W(\underline{\theta}',\underline{\tilde{\varphi}})$, there are $(u_E,\underline{\theta}_E,\underline{\varphi}_E):=\mathcal{R}_M E$. By the definitions $u_H=u_E$, $\underline{\theta}_H=\underline{\theta}'+\underline{\theta}_E$, $\underline{\varphi}_H=\underline{\tilde{\varphi}}+\underline{\varphi}_E$, all the estimates except the one for $\underline{\sigma}_H$ follow by \eqref{eqmeancurvM}, \ref{propLEM} and \ref{proppreini}\ref{itempreiniphi}.

From \ref{corini}\ref{iteminialpha} and \ref{corini}\ref{iteminidifalpha}, $\alpha_i^+\Confarg\in [30\delta, \frac{\pi}{2}- 30\delta]$ when $\tau$ small enough. Moreover, by the definition, $\sigma_i=\log\frac{\alpha_i^+\Confarg}{\alpha_i^-\Confarg}$. The result then follows by the estimate for $\underline{\theta}_H$.
\end{proof}

\section{The Main Results}

\begin{prop}[The nonlinear terms of the mean curvature {${H[\tilde{u}]}$}]
\label{propnonlinear}
Given $\tilde{u}\in C^{2,\alpha}_{\Grp}(M)$, with $\norm{\tilde{u}}_2$ small enough, then
\begin{equation*}
    \norm{H[\tilde{u}]-H-\mathcal{L}_Mu}_0\leq C\norm{u}^2_2,
\end{equation*}
where $H$ is the mean curvature of $M$, $H[\tilde{u}]$ is the mean curvature of $M_{\tilde{u}}$, $u=(\Omega\circ X)^{-1}\tilde{u}$ as in \ref{lemauu}.
\end{prop}

\begin{proof}
Let $\mathcal{S}:=\mathcal{S}_i[\xi]$, $\beta:=\beta_i(\xi)$, $i=1\dots,k$, by \ref{lemXu}\ref{itemXug}, \ref{lemXu}\ref{itemXuA} and \ref{lemZ}\ref{itemASigma}, $\norm{A_\mathcal{S}:C^3(\mathcal{S},(\tau\beta)^{-2}g_\mathcal{S})}\leq C$. Therefore, by the fact that when $0\leq s\leq \frac{5\delta_s}{\tau}$,
\begin{equation*}
    e^{\gamma s}\leq \frac{2}{(e^{-\gamma s}+b_2)^2}
\end{equation*}
if $\tau$ small enough, by \ref{defga}, \ref{lemauu}, and \cite[Lemma B.1]{kapouleas:finite}, 
\begin{equation} 
\label{eqHS}
    \tau\norm{H[\tilde{u}]-H-\mathcal{L}_Mu : 
C^{0,\alpha}\left(\mathcal{S},(\tau\beta)^{-2}g_\mathcal{S},e^{-\gamma s}\right)}\leq C\tau^{-2} \norm{u:C^{2,\alpha}\left(\mathcal{S},(\tau\beta)^{-2}g_\mathcal{S},e^{-\gamma s}+b_2\right)}^2,
\end{equation}
when $\tau^{-1} \norm{u:C^{2,\alpha}\left(\mathcal{S},(\tau\beta)^{-2}g_\mathcal{S},e^{-\gamma s}+b_2\right)}$ small enough.

On the other hand, let $\mathcal{A}:=\mathcal{N}_j[\xi]\setminus S[\xi]$, $j=1,\dots,k+1$, by \ref{propconvconf}\ref{itemconfA} and the construction, 
$\norm{A_{\mathcal{A}}:C^{3}(\mathcal{A},g_{\mathcal{A}})}\leq C$. 
Therefore, by \ref{defga}, \ref{lemauu}, and \cite[Lemma B.1]{kapouleas:finite}, 
\begin{align}\label{eqHA}
    \norm{H[\tilde{u}]-H-\mathcal{L}_Mu:C^{0,\alpha}\left(\mathcal{A},g_\mathcal{A}\right)}\leq C \norm{u:C^{2,\alpha}\left(\mathcal{A},g_\mathcal{A}\right)}^2,
\end{align}
when $\norm{u:C^{2,\alpha}\left(\mathcal{A},g_\mathcal{A}\right)}$ small enough. 
The result then follows by combining \eqref{eqHS} and \eqref{eqHA} 
and using the definitions.
\end{proof}

We combine now the results of the previous sections with the estimates above to prove the main theorem of this article.

\begin{theorem}
\label{thm}
Given $k\in \N$ large enough in absolute terms and $m\in\N$ large enough in terms of $k$, 
there is a compact embedded two-sided free boundary minimal smooth surface (FBMS) $\Mmin_{k,m}$ of genus zero and $km$ connected components, 
which is symmetric under the action of $\Grp=D_{2m}\times\Z_2$ (Definition \ref{D:Grp}) and is the graph of a small function $\tilde{v}$ 
in an auxiliary ambient metric $g_A$ (Definition \ref{defga}) over a smooth initial surface $M[\xi_0]$ (see \ref{defini} and \ref{propini}),  
where $\xi_0\in \Xi_{\mathcal{V}}$, the constant $\underline{c}>0$ in the definition of $\Xi_{\mathcal{V}}$ (Definition \ref{defXi}) depends only on $k$,  
and $\norm{\tilde{v}}_2\leq{C(k)}/{m}$ (recall Definition \ref{D:norm} for the definition of the norm). 
Moreover $\Mmin_{k,m}$ converges on compact subsets of the interior of $\B^3$ in all norms to the configuration $\Wcal\zConfarg$ as $m\to\infty$; 
the Hausdorff distance of $\Mmin_{k,m} $ from $\Wcal\zConfarg$ tends to $0$ as $m\to\infty$; 
and in turn the Hausdorff distance of $\Wcal\zConfarg$ from $\Sph^2=\partial\B^3$ tends to $0$ as $k\to\infty$.  
Finally $\lim_{k\to\infty}\lim_{m\to\infty} \Mmin_{k,m} = \Sph^2$ in the varifold sense. 
\end{theorem}	

\begin{proof}
Let $\mathcal{X}$ be the Banach space $C^{2,\alpha'}(M[0])$, where $\alpha'\in (0,\alpha)$ is a fixed constant. Given a function $f\in\mathcal{X}$, the norm $\norm{f:\mathcal{X}}$ is defined by
\begin{equation*}
    \norm{f:\mathcal{X}}:=\norm{f:C^{2,\alpha'}(M[0])}.
\end{equation*}
There is a family of smooth diffeomorphisms $D_{\xi}:M[0]\to M[\xi]$ for $\xi\in \Xi_{\mathcal{V}}$, which depend smoothly on $\xi$ and satisfy the following: there is constant $C$ such that for $f\in C^{2,\alpha}(M[0])$ and $f'\in C^{2,\alpha}(M[\xi])$,
\begin{equation}\label{eqdiffmor}
    \norm{f\circ D_{\xi}^{-1}}_2\leq C\norm{f}_2, \quad \norm{f'\circ D_{\xi}}_2\leq C\norm{f'}_2.
\end{equation}
Let $\Xi:=\{(\xi,u)\in\mathcal{V}\times \mathcal{X}:\abs{\xi}\leq \underline{c} \tau, \norm{u}_2\leq  \underline{c} \tau \}$. By the definition, $\Xi\subset \mathcal{V}\times \mathcal{X}$ is compact and convex.

Given $(\xi,u)\in \Xi$. Let $v:=u\circ D_{\xi}^{-1}-u_H$, where $u_H\in C^{2,\alpha}(M[\xi])$ as in \ref{corLEmeancurv}. Then by \ref{corLEmeancurv} and \eqref{eqdiffmor}, 
\begin{equation}\label{eqvzeta} 
    \norm{v}_2\leq C(\underline{c}+1)\tau.
\end{equation}
By \ref{propnonlinear}, the surface $M_{\tilde{v}}$ is well defined and 
\begin{equation}\label{eqnonlinear} 
    \norm{H[\tilde{v}]-H-\mathcal{L}_Mv}_0\leq C(\underline{c}+1)^2\tau^2,
\end{equation}
where $\tilde{v}:=(\Omega\circ X) v$ as in \ref{lemauu}. Let $E:=H[\tilde{v}]-H-\mathcal{L}_Mv$, by \ref{propLEM}, there is $(v_E,\underline{\theta}_E,\underline{\varphi}_E)=\mathcal{R}_{M}E$. Let  $\underline{\sigma}_E:=\left\{\frac{\underline{\theta}_{Ei}}{\alpha_i^{+}\Confarg}\right\}_{i=1}^k$, by \ref{propLEM}, \ref{corLEmeancurv} and \eqref{eqnonlinear},
\begin{align}
&H[\tilde{v}]=\mathcal{L}_M(v_E+u\circ D_{\xi}^{-1})+W(\underline{\theta}_H+\underline{\theta}_E,\underline{\varphi}_H+\underline{\varphi}_E),\label{eqHv}\\ 
&\partial_{\rho}\tilde{v}=(\partial_{\rho}+1)u\circ D_{\xi}^{-1},\label{eqPhiv}\\
&\abs{\underline{\sigma}-\underline{\sigma}_E-\underline{\sigma}_H}\leq C\tau+C(\underline{c}+1)^2\tau^2,\label{eqsigmatau}\\ 
&\abs{\underline{\varphi}+\underline{\varphi}_E+\underline{\varphi}_H}\leq C\tau+C(\underline{c}+1)^2\tau^2,\label{eqvarphitau}\\ 
&\norm{v_E}_2\leq C(\underline{c}+1)^2\tau^2. \label{eqvtau}
\end{align}
Then define the map $\mathcal{J}:\Xi\to \mathcal{V}\times\mathcal{X}$ by 
\begin{equation}\label{eqJ}
    \mathcal{J}(\xi,u):=((\underline{\sigma}-\underline{\sigma}_E-\underline{\sigma}_H,\underline{\varphi}+\underline{\varphi}_E+\underline{\varphi}_H),-v_E\circ D_{\xi}).
\end{equation}

\eqref{eqsigmatau}, \eqref{eqvarphitau} and \eqref{eqvtau} imply that by choosing $\underline{c}$ large enough and $\tau$ small enough, 
$\mathcal{J}(\Xi)\subset\Xi$. 
As $\mathcal{J}$ is continuous by the construction, 
the Schauder fixed point theorem could be applied and there must be a fixed point $(\xi_0,u)$ of $\mathcal{J}$. 
By \eqref{eqHv} and \eqref{eqJ}, the corresponding $\Mmin_{k,m}:=M_{\tilde{v}}$ is a minimal surface. 
Moreover, by \ref{lemauu}, \ref{propLEM} and \eqref{eqPhiv}, it has free boundary in $\mathbb{B}^3$ and by \ref{propini}\ref{iteminibd}, 
there are $km$ boundary components. The estimate of $\tilde{v}$ follows by \eqref{eqvzeta} and the uniform boundedness of $\Omega$ in the Definition \ref{defga}. 
The smoothness of $\tilde{v}$ follows by the standard regularity theorem. 
The embeddedness follows by the embeddedness of $M[\xi_0]$ and the smallness of $\tilde{v}$. 
The symmetry of $\Mmin_{k,m}$ follows by the construction and \ref{propini}\ref{iteminisym}. 
The convergence of $\Mmin_{k,m}$ follows by \ref{corbal} and the smallness of $\tilde{v}$.
\end{proof}

\bibliographystyle{amsplain}
\bibliography{paper}
\end{document}